\documentclass[11pt]{article}
\usepackage[latin1]{inputenc}
\usepackage[T1]{fontenc}
\usepackage{lmodern}
\usepackage[a4paper]{geometry}
\usepackage{amstext}
\usepackage{amsmath, amssymb}
\usepackage{amsthm}
\usepackage{dsfont}
\usepackage[english]{babel}
\usepackage{a4wide}
\usepackage{graphicx,color}
\usepackage{array}
\usepackage[nottoc, notlof, notlot]{tocbibind}
\usepackage{subfigure}
\usepackage[colorlinks]{hyperref}
\usepackage{enumitem,rotating}
\hypersetup{
linkcolor=blue,
citecolor=blue,
}

\makeatletter
\newsavebox\myboxA
\newsavebox\myboxB
\newlength\mylenA

\newcommand*\xoverline[2][0.75]{%
    \sbox{\myboxA}{$\m@th#2$}%
    \setbox\myboxB\null% Phantom box
    \ht\myboxB=\ht\myboxA%
    \dp\myboxB=\dp\myboxA%
    \wd\myboxB=#1\wd\myboxA% Scale phantom
    \sbox\myboxB{$\m@th\overline{\copy\myboxB}$}%  Overlined phantom
    \setlength\mylenA{\the\wd\myboxA}%   calc width diff
    \addtolength\mylenA{-\the\wd\myboxB}%
    \ifdim\wd\myboxB<\wd\myboxA%
       \rlap{\hskip 0.5\mylenA\usebox\myboxB}{\usebox\myboxA}%
    \else
        \hskip -0.5\mylenA\rlap{\usebox\myboxA}{\hskip 0.5\mylenA\usebox\myboxB}%
    \fi}
\makeatother

\theoremstyle{plain}
\newtheorem{thm}{Theorem}[section]
\theoremstyle{plain}

\theoremstyle{plain}

\theoremstyle{plain}
\newtheorem{proposition}{Proposition}[section]
\theoremstyle{definition}
\newtheorem{defi}[thm]{Definition}
\theoremstyle{plain}

\newtheorem{cor}[thm]{Corollary}
\newtheorem{rmq}{Remark}
\theoremstyle{definition}

\newcommand{\E}{{\mathbb E}}
\newcommand{\R}{{\mathbb R}}

\renewcommand{\P}{{\mathbb P}}

\newcommand{\PP}{{\mathcal P}}

\newcommand{\XX}{{\mathcal X}}
\newcommand{\LL}{{\mathcal L}}
\newcommand{\NN}{{\mathcal N}}
\renewcommand{\SS}{{\mathcal S}}
\newcommand{\1}{{\mathds 1}}

\newcommand{\WB}{\xoverline{W}}

\def\diag{\mathop{\rm diag}\nolimits}%
\newcommand{\umin}[1]{\underset{#1}{\min}\;}
\newcommand{\umax}[1]{\underset{#1}{\max}\;}
\newcommand{\uinf}[1]{\underset{#1}{\inf}\;}
\DeclareMathOperator{\diam}{diam}

\newcommand{\ualpha}{u}
\newcommand{\vbeta}{v}

\title{Central limit theorems for entropy-regularized optimal transport on finite spaces and statistical applications
\footnote{This work has been carried out with financial support from the French
State, managed by the French National Research Agency (ANR) in the frame
of the GOTMI project (ANR-16-CE33-0010-01). }}
\author{J\'er\'emie Bigot\footnote{J. Bigot is a member of Institut Universitaire de France.}, Elsa Cazelles \& Nicolas Papadakis   \\
\\  Institut de Math\'ematiques de Bordeaux et CNRS  (UMR 5251)   \\ Universit\'e de Bordeaux }

\begin{document}

\maketitle

\begin{abstract}
The notion of entropy-regularized optimal transport, also known as Sinkhorn divergence, has recently gained popularity in machine learning and statistics, as it makes feasible the use of smoothed optimal transportation distances for data analysis. The Sinkhorn divergence allows the fast computation of an entropically regularized Wasserstein distance between two probability distributions supported on a finite metric space of (possibly) high-dimension. For data sampled from one or two unknown probability distributions, we derive the distributional limits of the empirical Sinkhorn divergence and its centered version (Sinkhorn loss). We also propose a bootstrap procedure which allows to obtain new test statistics for measuring the discrepancies between multivariate probability distributions. Our work is inspired by the results of Sommerfeld and Munk in \cite{sommerfeld2016inference} on the asymptotic distribution of empirical Wasserstein distance on finite space using unregularized transportation costs. Incidentally we also analyze the asymptotic distribution of entropy-regularized Wasserstein distances when the regularization parameter tends to zero. Simulated and real datasets are used to illustrate our approach.
\end{abstract}

\section{Introduction}\label{sec:intro}

\subsection{Motivations}

In this paper, we study the convergence (to their population counterparts) of empirical probability measures supported on a finite metric space with respect to entropy-regularized transportation costs. Transport distances are widely employed for comparing probability measures since they capture in a instinctive manner the geometry of distributions (see e.g \cite{villani2003topics} for a general presentation on the subject). In particular, the Wasserstein distance is well adapted to deal with discrete probability measures (supported on a finite set), as its computation reduces to solve a linear program. Moreover, since data in the form of histograms may be represented as discrete measures,  the Wasserstein distance has been shown to be a relevant statistical measure in various fields such as clustering of discrete distributions \cite{ye2015fast}, nonparametric Bayesian modelling \cite{nguyen2011wasserstein}, fingerprints comparison \cite{sommerfeld2016inference}, unsupervised learning \cite{arjovsky2017wasserstein} and principal component analysis \cite{bigot2017geodesic, NIPS2015_5680,GPCA}.

However, the computational cost to evaluate a transport distance is generally of order $\mathcal{O}(N^3\log N)$ for discrete probability distributions with a support of size $N$. To overcome the computational cost to evaluate a transport distance, Cuturi \cite{cuturi} has proposed to add an entropic regularization term  to the linear program corresponding to a standard optimal transport problem, leading to the notion of Sinkhorn divergence between  probability distributions. Initially, the purpose of transport plan regularization was to efficiently compute  a divergence term close to the Wasserstein distance between two probability measures,  through an iterative scaling algorithm where each iteration costs $\mathcal{O}(N^2)$. This proposal has recently gained popularity in machine learning and statistics, as it makes feasible the use of smoothed optimal transportation distance for data analysis.
It has found various applications such as generative models \cite{2017-Genevay-AutoDiff} and more generally for high dimensional data analysis in multi-label learning \cite{frogner2015learning}, dictionary learning \cite{rolet2016fast}, image processing \cite{CuturiPeyre,rabin2015convex}, text mining via bag-of-words comparison \cite{2016-genevay-nips}, averaging of neuroimaging data \cite{gramfort2015fast}.

The goal of this paper is to analyze the potential benefits of the Sinkhorn divergence and its centered version \cite{feydy2018interpolating,2017-Genevay-AutoDiff} for statistical inference from empirical probability measures. We derive novel results on the asymptotic distribution of such divergences for data sampled from (unknown) distributions supported on a finite metric space. The main application is to obtain new test statistics (for one or two samples problems) for the comparison of multivariate probability distributions.

\subsection{Previous work and main contributions}

The derivation of distributional limits of an empirical measure towards its population counterpart in $p$-Wasserstein distance $W_p(\mu,\nu)$ is well understood for probability measures $\mu$ and $\nu$ supported on $\R$ \cite{MR2161214,MR1740113,MR2121458}. These results have then been extended for specific parametric distributions supported on $\R^{d}$ belonging to an elliptic class, see \cite{MR3545279} and references therein. Recently, a central limit theorem has been established in \cite{delBarrio2017} for empirical transportation cost, and  data sampled from absolutely continuous measures on $\R^{d}$,  for any $d \geq 1$. The case of discrete measures supported on a finite metric space has also been considered in \cite{sommerfeld2016inference} with the proof of the convergence (in the spirit of the central limit theorem) of empirical Wasserstein distances toward the optimal value of a linear program. Additionally,  Klatt et al.\ \cite{klatt2018empirical} analyzed, in parallel with our results, the distributional limit of regularized  optimal transport divergences between empirical distributions. In particular, the work in \cite{klatt2018empirical} extends the study of distributional limits of regularized empirical transportation cost to general penalty functions (beyond entropy regularization). The authors of \cite{ramdas2017wasserstein} also studied the link between nonparametric tests and the Wasserstein distance, with an emphasis on distributions with support in $\R$.

However, apart from the one-dimensional case ($d=1$), and the work of \cite{klatt2018empirical}, these results lead to test statistics whose numerical implementation become prohibitive for empirical measures supported on $\R^{d}$ with $d \geq 2$. The computational cost required to evaluate a transport distance is indeed only easily tractable in $\R$.
It is therefore of interest to propose  test statistics based on fast Sinkhorn divergences \cite{cuturi}. In this context, this paper focuses on the study of inference from discrete distributions in terms of entropically regularized transport costs, the link with the inference through unregularized transport, and the construction of tests statistics that are well suited to investigate the equality of two distributions. The results are inspired by the work in \cite{sommerfeld2016inference}  on the asymptotic distribution of empirical Wasserstein distance on finite space using unregularized transportation costs.

Our main contributions may be summarized as follows. First, for data sampled from one or two unknown  measures  $\mu$ and $\nu$ supported on a finite space, we derive central limit theorems for the Sinkhorn divergence between their empirical counterpart. These results allow to build new test statistics for measuring the discrepancies between multivariate probability distributions. Notice however that the Sinkhorn divergence denoted $W_{p,\varepsilon}^p(\mu,\nu)$ (where $\varepsilon > 0$ is a regularization parameter) is not a distance since $W_{p,\varepsilon}^p(\mu,\mu) \neq 0$. This is  a serious drawback for testing the hypothesis of equality between distributions. Thus, as introduced in \cite{feydy2018interpolating,2017-Genevay-AutoDiff},  we further consider the centered version of the Sinkhorn divergence  $\WB_{p,\varepsilon}^p(\mu,\nu)$, referred to as Sinkhorn loss, which satisfies
$\WB_{p,\varepsilon}^p(\mu,\mu) = 0$. This study thus  constitutes an important novel contribution with respect to  the work of \cite{klatt2018empirical}. We present new results on the asymptotic distributions of the  Sinkhorn loss between empirical measures. Interestingly, under the hypothesis that $\mu = \nu$, such statistics do not converge to a Gaussian random variable but to a mixture of chi-squared distributed random variables. To illustrate the applicability of the method to the analysis of real data, we  propose a bootstrap procedure to estimate unknown quantities of interest on the distribution of these  statistics such as their non-asymptotic variance and quantiles. Simulated and real datasets are used to illustrate our approach. Finally, one may stress that an advantage of the use of test statistics based regularized Wasserstein distance (rather than other losses or divergences) is to allow further statistical inference from the resulting optimal transport plan as demonstrated in  \cite{klatt2018empirical} for the analysis of protein interaction networks.

\subsection{Overview of the paper}

In Section \ref{sec:results}, we briefly recall the optimal transport problem between probability measures, and we introduce the notions of Sinkhorn divergence  and Sinkhorn loss. Then, we derive the  asymptotic distributions for the empirical Sinkhorn divergence and  the empirical Sinkhorn loss. We also give the behavior of such statistics when the regularization parameter $\varepsilon$ tends to zero at a rate depending on the number of available observations. A bootstrap procedure is discussed in Section \ref{sec:bootstrap}. Numerical experiments are reported in Section \ref{sec:numeric} for synthetic data and in Section \ref{sec:real} for  real data.

\section{Distributional limits for entropy-regularized optimal transport}\label{sec:results}

In this section, we give  results on the  asymptotic distributions of the empirical Sinkhorn divergence and  the empirical Sinkhorn loss. The proofs rely on the use of the delta-method and on the property that $W_{p,\varepsilon}^p(\mu,\nu)$ is a differentiable function with respect to $\mu$ and $\nu$.

\subsection{Notation and definitions}\label{subsec:def}

We first introduce various notation and definitions that will be used throughout the paper.

\subsubsection{Optimal transport, Sinkhorn divergence and Sinkhorn loss}

Let $(\XX,d)$ be a complete metric space with $d:\XX\times\XX\rightarrow\R_+$. We denote by $\PP_p(\XX)$ the set of Borel probability measures $\mu$ supported on $\XX$ with finite moment of order $p \geq 1$, in the sense that $\int_{\XX} d^{p}(x,y) d\mu(x)$ is finite for some (and thus for all) $y \in \XX$. The $p$-Wasserstein distance between two measures $\mu$ and $\nu$ in $\PP_p(\XX)$ is defined by
\begin{equation}
\label{def_Wass}
W_p(\mu,\nu)=\left(\uinf{\pi\in\Pi(\mu,\nu)}\ \iint_{\XX^2}d^{p}(x,y)d\pi(x,y)\right)^{1/p}
\end{equation}
where the infimum is taken over  the set  $\Pi(\mu,\nu)$ of probability measures $\pi$ on the product space $\XX\times\XX$ with respective marginals $\mu$ and $\nu$.

In this work, we consider the specific case where $\XX=\{x_1,\ldots,x_N\}$ is a finite metric space of size $N$. In this setting, a measure $\mu \in \PP_p(\XX)$ is discrete, and we write  $\mu=\sum_{i=1}^N a_i\delta_{x_i}$ where $(a_1,\ldots,a_N)$ is a vector of positive weights belonging to the simplex $\Sigma_N:=\{a=(a_i)_{i=1,\ldots,N}\in\R_+^N\ \mbox{such that}\ \sum_{i=1}^Na_i=1 \}$ and $\delta_{x_i}$ is a Dirac measure at location  $x_i$. 
Therefore, computing the $p$-Wasserstein distance between discrete probability measures supported on $\XX$ amounts to solve a linear program  whose solution is constraint to belong to the convex set $\Pi(\mu,\nu)$. However, the cost of this convex minimization becomes  prohibitive for moderate to large values of $N$.
Regularizing a complex problem with an entropy term is a classical approach in optimization  to reduce its  complexity \cite{wilson1969use}.
This is the approach followed in  \cite{cuturi} by adding  an entropy regularization to the transport matrix, which yields the strictly convex (primal) problem \eqref{primal} presented below.

As the space $\XX$ is  fixed, a probability measure supported on $\XX$  is entirely characterized by a vector of weights in the simplex. By a slight abuse of notation, we thus identify a measure $\mu  \in \PP_p(\XX)$ by its vector of weights $a=(a_1,\ldots,a_n)\in\Sigma_N$ (and we sometimes write $a = \mu$).

\begin{defi}[Sinkhorn divergence]
Let $\varepsilon > 0$ be a regularization parameter. The Sinkhorn divergence \cite{cuturi} between two  probability measures  $\mu = \sum_{i=1}^N a_i\delta_{x_i}$ and $\nu = \sum_{i=1}^N b_i\delta_{x_i}$ in $\PP_p(\XX)$ is defined by
\begin{equation}
\label{primal}
W_{p,\varepsilon}^p(a,b)=\umin{T\in U(a,b)}\langle T,C\rangle+\varepsilon H(T|a\otimes b), \mbox{ with $a$ and $b$ in $\Sigma_N$},
\end{equation}
where $\langle \cdot , \cdot \rangle$ denotes the usual inner product between matrices, $a\otimes b$ denotes the tensor product $(x_i,x_j)\mapsto a_ib_j$ and
\begin{description}
\item[-] $U(a,b)=\{T\in\R_{+}^{N\times N} \ \vert T\1_N=a,T^T\1_N=b\}$ is the set of transport matrices with marginals $a$ and $b$ (with $\1_N$ denoting the vector of $\R^{N}$ with all entries equal to one) ,
\item[-] $C \in \R_{+}^{N\times N}$  is the pairwise cost matrix associated to the metric space $(X,d)$ whose $(i,j)$-th entry is $c_{ij} =d(x_i,x_j)^p$,
\item[-] the regularization function $H(T|a\otimes b)=\sum_{i,j}\log\left(\frac{t_{ij}}{a_ib_j}\right)t_{ij}$ is the relative entropy for a transport matrix $T \in U(a,b)$.
\end{description}
\end{defi}

\begin{rmq}
This entire section is also valid for symmetric positive cost matrices $C$ for which $C(x_i,x_i)=0$.
\end{rmq}

The dual version of problem \eqref{primal} is introduced in the following definition.

\begin{defi}[Dual problem] Following \cite{cuturi2013fast}, the dual version of the minimization problem \eqref{primal} is given by 
\begin{equation}
\label{dual}
W_{p,\varepsilon}^p(a,b)=\umax{\ualpha,\vbeta\in\R^N}\ualpha^Ta+\vbeta^Tb-\varepsilon\sum_{i,j} \left(e^{-\frac{1}{\varepsilon}(c_{ij}- 1-\ualpha_i-\vbeta_j )} \right)a_ib_j.
\end{equation}
We denote by $\SS_{\varepsilon}(a,b)$ the set of optimal solutions of the maximization problem \eqref{dual}.
\end{defi}

It is now well known that there exists an explicit relationship between the optimal solutions of  primal \eqref{primal} and dual \eqref{dual} problems. These solutions can be computed through an iterative method called Sinkhorn's algorithm \cite{cuturi2013fast} that is described below and which explicitly gives this relationship.
\begin{proposition}[Sinkhorn's algorithm]\label{prop:sinkhorn}
Let  $K=\exp(- C/\varepsilon-\1_{N\times N})$ be the elementwise exponential of the matrix cost $C$ divided by $-\varepsilon$ minus the matrix with all entries equal to $1$. Then, there exists a pair of vectors $(\tilde{u},\tilde{v}) \in \R_{+}^{N} \times \R_{+}^{N}$ such that the optimal solutions $T_{\varepsilon}^{\ast}$ and $(\ualpha_{\varepsilon}^{\ast},\vbeta_{\varepsilon}^{\ast})$ of problems \eqref{primal} and \eqref{dual} are respectively given by
$$
T_{\varepsilon}^{\ast} = [\diag(\tilde{u}) K \diag(\tilde{v})]\odot(a\otimes b), \mbox{ and } \ualpha_{\varepsilon}^{\ast} = \varepsilon\log(\tilde{u}), \;  \vbeta_{\varepsilon}^{\ast} = \varepsilon\log(\tilde{v}).
$$
where $\odot$ is the pointwise multiplication. Moreover, such a pair $(\tilde{u},\tilde{v})$ is unique up to scalar multiplication (or equivalently $(\ualpha_{\varepsilon}^{\ast},\vbeta_{\varepsilon}^{\ast})$ is unique up to translation), and it can be recovered as a fixed point of the Sinkhorn map
\begin{equation}
(\tilde{u},\tilde{v})\in\R^N\times\R^N\mapsto (a/(K\tilde{v}),b/(K^T\tilde{u})).
\label{Sinkhorn_map}
\end{equation}
where $K^T$ is the transpose of $K$ and $/$ stands for the component-wise division.
\end{proposition}

\begin{rmq}
When the cost matrix  $C$ is defined as $c_{ij} =\|x_i-x_j\|^2$ and the grid points $x_i$ are uniformly spread, the matrix vector products involving $\exp(- C/\varepsilon)$ within the Sinkhorn algorithm can be efficiently performed via separated one dimensional convolutions \cite{2015-solomon-siggraph} without storing $C$ .
\end{rmq}
As discussed in the introduction, an important issue regarding the use of  Sinkhorn divergence for testing the equality of two distributions is that it leads to a biased statistics in the sense that  $W_{p,\varepsilon}^p(a,b)$ is not equal to zero under the null hypothesis $a = b$. A possible alternative to avoid this issue is to consider the so-called notion of Sinkhorn loss \cite{feydy2018interpolating,2017-Genevay-AutoDiff} as defined below.

\begin{defi}[Sinkhorn loss] Let $\varepsilon > 0$ be a regularization parameter. The Sinkhorn loss  between two  probability measures  $\mu = \sum_{i=1}^N a_i\delta_{x_i}$ and $\nu = \sum_{i=1}^N b_i\delta_{x_i}$ in $\PP_p(\XX)$ is defined by
\begin{equation}
\WB_{p,\varepsilon}^p(a,b) := W_{p,\varepsilon}^p(a,b) - \frac{1}{2} \left(W_{p,\varepsilon}^p(a,a) +W_{p,\varepsilon}^p(b,b)\right),
\end{equation}
\label{def_Sloss}
\end{defi}

The Sinkhorn loss is not a distance between probability distributions, but it satisfies various interesting properties for the purpose of this paper, that are summarized below.

\begin{proposition} \label{prop:sinkloss}
The Sinkhorn loss satisfies the following three key properties (see Theorem 1 in \cite{feydy2018interpolating}):
\begin{enumerate}
\item[(i)] $\WB_{p,\varepsilon}^p(a,b)\geq 0$,
\item[(ii)] $\WB_{p,\varepsilon}^p(a,b) = 0 \Leftrightarrow a = b$,
\item[(iii)] $\WB_{p,\varepsilon}^p(a,b)\underset{\varepsilon\rightarrow 0}{\longrightarrow} W_p^p(a,b)$.
\end{enumerate}
\end{proposition}

From  Proposition \ref{prop:sinkloss}, we have that $a=b$ is equivalent to $\WB_{p,\varepsilon}^p(a,b)=0$, therefore the function $(a,b)\mapsto \WB_{p,\varepsilon}^p(a,b)$ reaches its global minimum at $a=b$, implying that the gradient of the Sinkhorn loss is zero when $a=b$ which is summarized  in the following corollary.

\begin{cor}\label{cor:gradient_null}
For any $a\in\Sigma_N$, the gradient of the Sinkhorn loss satisfies $\nabla \WB_{p,\varepsilon}^p(a,a)=0$.
\end{cor}

\subsubsection{Statistical notations}

We denote by $\overset{\LL}{\longrightarrow}$ the convergence in distribution of a random variable and $ \overset{\P}{\longrightarrow}$ the convergence in probability. The notation $G\overset{\LL}{\sim}a$ means that $G$ is a random variable taking its values in $\XX$ with law $a=(a_1,\ldots,a_n)\in\Sigma_N$ (namely that $\P(G = x_{i}) = a_{i}$ for each $1 \leq i \leq N$). Likewise $G\overset{\LL}{\sim}H$ stands for the equality in distribution of the random variables $G$ and $H$.

Let $a,b\in\Sigma_N$ and $\hat{a}_n$ and $\hat{b}_m$ be the empirical measures respectively generated by iid samples $X_1,\ldots,X_n\overset{\LL}{\sim} a$ and $Y_1,\ldots,Y_m\overset{\LL}{\sim}  b$, that is
\begin{equation}
\label{empi_measure}
\hat{a}_n=(\hat{a}_{n}^{x})_{x\in \XX},\ \mbox{where}\ \hat{a}_{n}^{x_i}=\frac{1}{n}\sum_{j=1}^n\1_{\{X_j=x_i\}}=\frac{1}{n}\#\{j:X_j=x_i\}\ \mbox{for all}\ 1\leq i\leq N .
\end{equation}
We also define the multinomial covariance matrix
\[\Sigma(a)=\begin{bmatrix}
a_{x_1}(1-a_{x_1}) & -a_{x_1}a_{x_2} & \cdots & -a_{x_1}a_{x_N} \\
-a_{x_1}a_{x_2} & a_{x_2}(1-a_{x_2}) & \cdots & -a_{x_2}a_{x_N}\\
\vdots & \vdots & \ddots & \vdots\\
-a_{x_1}a_{x_N} & -a_{x_2}a_{x_N} & \cdots & a_{x_N}(1-a_{x_N})
\end{bmatrix}\]
and the independent Gaussian random vectors $G\sim \mathcal{N}(0,\Sigma(a))$ and $H\sim \mathcal{N}(0,\Sigma(b))$. As classically done in statistics, we say that
$$\left\{\begin{array}{ll}
        H_0  : a=b \mbox{ is the null hypothesis}, \\
        H_1 :  a\neq b \mbox{ is the alternative hypothesis}.
    \end{array} \right.$$
    
\begin{rmq}
As stated in Proposition \ref{prop:sinkhorn}, the dual variables $(\ualpha_{\varepsilon}^{\ast},\vbeta_{\varepsilon}^{\ast})$ solutions of \eqref{dual} for $a$ and $b$ in the simplex are unique up to a scalar addition. Hence for any $t\in\R$,
$$\langle G,\ualpha_{\varepsilon}^{\ast}+ t\1_N \rangle \overset{\LL}{\sim} \langle G,\ualpha_{\varepsilon}^{\ast}\rangle,$$
since $G$ is centered in $0$ and $\1_N'\Sigma(a)\1_N=0$ for $a$ in the simplex.
\end{rmq}

\subsubsection{Notations for differentiation}
For a sufficiently smooth function $f:(x,y)\in\R^N\times\R^N\longmapsto\R$, we denote by $\nabla f$ and $\nabla^2 f$ the gradient and the hessian of the function $f$. In particular, the gradient of $f$ at the point  $(x,y)\in\R^N\times\R^N$ in the direction $(h_1,h_2)\in\R^N\times\R^N$ is denoted by $\nabla f(x,y)(h_1,h_2)$ (this notation also holds for the hessian). Moreover, the first-order partial derivative with respect to the first variable $x$ (resp. $y$) is given by $\partial_{\scriptscriptstyle{1}} f$ (resp. $\partial_{\scriptscriptstyle{2}} f$). Equivalently, the second-order partial derivative is denoted $\partial^2_{\scriptscriptstyle{ij}}f$, with $i\in\{1,2\}, j\in\{1,2\}$.

\subsection{Differentiability of $W_{p,\varepsilon}^p$}\label{subsec:deriv}

As stated at the beginning of the section, the differentiability of $W_{p,\varepsilon}^p$ (in the usual Fr\'echet sense) is needed in order to apply the delta-method. This is proved in the following proposition. %Note that the authors of \cite{luise2018differential} proved that $W_{p,\varepsilon}^p$ is actually $C^{\infty}$ on the interior of its domain.

\begin{proposition}
\label{grad_W}
The functional $(a,b)\mapsto W_{p,\varepsilon}^p(a,b)$ is differentiable in $\Sigma_N\times\Sigma_N$ with gradient
$$\nabla W_{p,\varepsilon}^p(a,b)(h_1,h_2)=\langle \ualpha_{\varepsilon},h_1\rangle + \langle \vbeta_{\varepsilon},h_2\rangle,$$
where $(\ualpha_{\varepsilon},\vbeta_{\varepsilon})\in\SS_{\varepsilon}(a,b)$, the set of optimal solutions of \eqref{dual}.
\end{proposition}
\begin{proof}
From Proposition 2 in \cite{feydy2018interpolating}, $W_{p,\varepsilon}^p$ is G\^{a}teaux differentiable and its derivative reads
$$\nabla W_{p,\varepsilon}^p(a,b)(h_1,h_2)=\langle \ualpha_{\varepsilon},h_1\rangle + \langle \vbeta_{\varepsilon},h_2\rangle,$$
for $(\ualpha_{\varepsilon},\vbeta_{\varepsilon})\in\SS_{\varepsilon}(a,b)$. In order to prove its differentiability (or Fr\'echet differentiability, since $\R^N$ is a finite dimensional space) at the point $(a,b)$, we only need to prove that the operator $\nabla W_{p,\varepsilon}^p$ is continuous (see \textit{e.g.} Prop. 3.2.3. in \cite{zalinescu2002convex}) in $(a,b)$. Suppose that $(a^n,b^n)$ tends to $(a,b)$ when $n$ tends to infinity. Therefore, this convergence also holds in the weak$^{\ast}$ topology for  the probability measures $\mu^n = \sum_{i=1}^N a_i^n\delta_{x_i}, \nu^n = \sum_{i=1}^N b_i^n\delta_{x_i}$ and $\mu = \sum_{i=1}^N a_i\delta_{x_i}, \nu = \sum_{i=1}^N b_i\delta_{x_i}$. We denote by $(\ualpha^n,\vbeta^n)$ the unique couple in $\SS_{\varepsilon}(a^n,b^n)$ such that for an arbitrary $i_0\in\{1,\ldots,N\}, \ualpha^n_{i_0}=0$. Then, we can apply Cauchy-Schwarz inequality and then use Proposition 13 in \cite{feydy2018interpolating} on the convergence of the pair $(\ualpha^n,\vbeta^n)$ of dual variables towards $(\ualpha,\vbeta)\in\SS_{\varepsilon}(a,b)$ (such that $\ualpha_{i_0}=0$),  to obtain that
\begin{align*}
\lim_{(a^n,b^n)\to (a,b)} \Vert \nabla W_{p,\varepsilon}^p(a^n,b^n) -\nabla W_{p,\varepsilon}^p(a,b)\Vert &= \lim_{(a^n,b^n)\to (a,b)}\sup_{\Vert (h_1,h_2)\Vert\leq 1} \vert \langle \ualpha^n-\ualpha,h_1\rangle+\langle\vbeta^n-\vbeta,h_2\rangle\vert\\
\leq \lim_{(a^n,b^n)\to (a,b)}\sup_{\Vert (h_1,h_2)\Vert\leq 1}& \Vert\ualpha^n-\ualpha\Vert \ \Vert h_1\Vert +\Vert\vbeta^n-\vbeta\Vert \ \Vert h_2\Vert \underset{(a^n,b^n)\to (a,b)}{\longrightarrow} 0,
\end{align*}
 which concludes the proof.
\end{proof}

\subsection{Distributional limits for the empirical Sinkhorn divergence}

\subsubsection{Convergence in distribution}
The following theorem is our main result on distributional limits of empirical Sinkhorn divergences. 
\begin{thm}
\label{Th_CLT}
%Let $K=\exp(-C/ \varepsilon-\1_{N\times N})$ be the matrix obtained by elementwise exponential of $-C/\varepsilon-\1_{N\times N}$.
For $a,b\in\Sigma_N$, let $(\ualpha_{\varepsilon},\vbeta_{\varepsilon})\in\SS_{\varepsilon}(a,b)$ be an optimal solution of the dual problem \eqref{dual} and $\hat{a}_n,\hat{b}_m$ be  the empirical measures defined in \eqref{empi_measure}.
Then, the following central limit theorems hold for empirical Sinkhorn divergences.
\begin{enumerate}
\item One sample. As $n \to +\infty$, one has that
\begin{equation}
\label{alt_one}
\sqrt{n}(W_{p,\varepsilon}^p(\hat{a}_n,b)-W_{p,\varepsilon}^p(a,b))\overset{\LL}{\longrightarrow} \ \langle G,\ualpha_{\varepsilon}\rangle.
\end{equation}
\item Two samples.  For $\rho_{n,m}=\sqrt{(nm/(n+m))}$ and $m/(n+m)\rightarrow \gamma\in (0,1)$ as $\min(n,m) \to + \infty$, one has that
\begin{equation}
\label{alt_two}
\rho_{n,m}(W_{p,\varepsilon}^p(\hat{a}_n,\hat{b}_m)-W_{p,\varepsilon}^p(a,b))\overset{\LL}{\longrightarrow} \sqrt{\gamma}\langle G,\ualpha_{\varepsilon}\rangle + \sqrt{1-\gamma} \langle H,\vbeta_{\varepsilon}\rangle.
\end{equation}
\end{enumerate}
\end{thm}

\begin{proof}
Following the proof of Theorem 1 in \cite{sommerfeld2016inference}, we have that (see e.g.\ Theorem 14.6 in \cite{wasserman2013all})
$$\sqrt{n}(\hat{a}_n-a)\overset{\LL}{\longrightarrow} G,\ \mbox{where} \ G\overset{\LL}{\sim}\NN(0,\Sigma(a)),$$
since $n\hat{a}_n$ is a sample of a multinomial probability measure with probability $a$. For the two samples case, we use that
$$\rho_{n,m}((\hat{a}_n,\hat{b}_m)-(a,b))\overset{\LL}{\longrightarrow} (\sqrt{\gamma}G,\sqrt{1-\gamma}H),$$
where $\rho_{n,m}$ and $\gamma$ are the quantities defined  in the statement of  Theorem \ref{Th_CLT}. From Proposition \ref{grad_W}, we can directly apply the delta-method:
\begin{equation}
\sqrt{n}(W_{p,\varepsilon}^p(\hat{a}_n,b)-W_{p,\varepsilon}^p(a,b)) \overset{\LL}{\longrightarrow} \langle G,\ualpha_{\varepsilon}\rangle, \mbox{ as } n \to +\infty, \label{int_alt_one_proof}
\end{equation}
while, for $n$ and $m$ tending to infinity such that $n\wedge m\rightarrow \infty$ and $m/(n+m)\rightarrow \gamma\in (0,1)$, we obtain that
\begin{equation}
\label{int_null_two_proof}
\rho_{n,m}(W_{p,\varepsilon}^p(\hat{a}_n,\hat{b}_m)-W_{p,\varepsilon}^p(a,b))\overset{\LL}{\longrightarrow}  \sqrt{\gamma}\langle G,\ualpha_{\varepsilon}\rangle+ \sqrt{1-\gamma} \langle H,\vbeta_{\varepsilon}\rangle.
\end{equation}
This completes the proof of  Theorem \ref{Th_CLT}.
\end{proof}

\subsubsection{Convergence in probability}
Distributional limits of empirical Sinkhorn divergences may also be characterized by a convergence in probability by the following result which directly follows from the delta-method (see \textit{e.g.} Theorem 3.9.4 in \cite{van1996weak}).

%\textcolor{red}{
\begin{thm}\label{th:cv_prob_div}
The following asymptotic results hold for empirical Sinkhorn divergences, for any $(\ualpha_{\varepsilon},\vbeta_{\varepsilon})\in\SS_{\varepsilon}(a,b)$.
\begin{enumerate}
\item One sample. As $n \to +\infty$, one has that
$$\sqrt{n}\left(W_{p,\varepsilon}^p(\hat{a}_n,b)-W_{p,\varepsilon}^p(a,b)-\langle \hat{a}_n-a,\ualpha_{\varepsilon}\rangle\right) \overset{\P}{\longrightarrow} 0.$$
\item Two samples -  For $\rho_{n,m}=\sqrt{(nm/(n+m))}$ and $m/(n+m)\rightarrow \gamma\in (0,1)$ as $\min(n,m) \to + \infty$, one has that
$$\rho_{n,m}\left(W_{p,\varepsilon}^p(\hat{a}_n,\hat{b}_m)-W_{p,\varepsilon}^p(a,b)-(\langle\hat{a}_n-a,\ualpha_{\varepsilon}\rangle+\langle \hat{b}_m-b,\vbeta_{\varepsilon}\rangle) \right)\overset{\P}{\longrightarrow} 0.$$
\end{enumerate}
\end{thm}

\begin{proof}
As the map $(h_1,h_2) \mapsto\nabla W_{p,\varepsilon}^p(a,b)(h_1,h_2)$ is defined, linear and continuous on $\R^N\times\R^N$, Theorem 3.9.4 in \cite{van1996weak} allows us to conclude.
\end{proof}

\subsection{Distributional limits for the empirical Sinkhorn loss}\label{subsec:convergence}

\subsubsection{Convergence in distribution}
The following theorems are our main results on distributional limits of the empirical Sinkhorn loss, for which we now distinguish the cases $a\neq b$ (alternative hypothesis) and $a=b$ (null hypothesis). 
\begin{thm}
\label{Th_CLT_loss_H1}
Let $a\neq b$ be two probability distributions in $\Sigma_N$.  Let us denote by $\hat{a}_n,\hat{b}_m$ their  empirical counterparts and by $(\ualpha_{\varepsilon}^{a,b},\vbeta_{\varepsilon}^{a,b})\in\SS_{\varepsilon}(a,b)$ the dual variables which are the optimal solutions of the dual problem $\eqref{dual}$. Then, the following asymptotic results hold.
\begin{enumerate}
\item One sample. As $n \to +\infty$, one has that
\begin{equation}
\label{alt_one_loss}
\sqrt{n}(\WB_{p,\varepsilon}^p(\hat{a}_n,b)-\WB_{p,\varepsilon}^p(a,b)) \overset{\LL}{\longrightarrow} \ \langle G, \ualpha_{\varepsilon}^{a,b}-\frac{1}{2}(\ualpha_{\varepsilon}^{a,a}+ \vbeta_{\varepsilon}^{a,a})\rangle.
\end{equation}
\item Two samples. For $\rho_{n,m}=\sqrt{(nm/(n+m))}$ and $m/(n+m)\rightarrow \gamma\in (0,1)$  as $\min(n,m) \to + \infty$, one has that
\begin{align*}
\label{alt_two_loss}
\rho_{n,m}(\WB_{p,\varepsilon}^p(\hat{a}_n,\hat{b}_m)-\WB_{p,\varepsilon}^p(a,b))\overset{\LL}{\longrightarrow} \sqrt{\gamma} & \langle G,\ualpha_{\varepsilon}^{a,b}-\frac{1}{2}(\ualpha_{\varepsilon}^{a,a}+ \vbeta_{\varepsilon}^{a,a})\rangle\\
+& \sqrt{1-\gamma} \langle H,\vbeta_{\varepsilon}^{a,b}-\frac{1}{2}(\ualpha_{\varepsilon}^{b,b}+ \vbeta_{\varepsilon}^{b,b})\rangle.
\end{align*}
\end{enumerate}
\end{thm}

\begin{proof}
The only difference with the proof of Theorem \ref{Th_CLT} is  the computation of the gradient of $\WB_{p,\varepsilon}^p$, which is given by
\begin{equation}
\label{diff_loss}
\nabla\WB_{p,\varepsilon}^p(a,b)(h_1,h_2)=\langle \ualpha_{\varepsilon}^{a,b}-\frac{1}{2}(\ualpha_{\varepsilon}^{a,a}+ \vbeta_{\varepsilon}^{a,a}),h_1\rangle + \langle \vbeta_{\varepsilon}^{a,b}-\frac{1}{2}(\ualpha_{\varepsilon}^{b,b}+ \vbeta_{\varepsilon}^{b,b}),h_2\rangle.
\end{equation}
The proof of Theorem \ref{Th_CLT_loss_H1} then follows from the same arguments as those used in the proof of Theorem \ref{Th_CLT} .
\end{proof}

Under the null hypothesis $a=b$, the derivation of the distributional limit of either $\WB_{p,\varepsilon}^p(\hat{a}_n,a)$ or $\WB_{p,\varepsilon}^p(\hat{a}_n,\hat{b}_m)$ requires further attention. Indeed, thanks to Proposition \ref{prop:sinkloss}, one has that the function $(a,b)\mapsto \WB_{p,\varepsilon}^p(a,b)$ reaches its global minimum at $a=b$, and therefore the gradient of the Sinkhorn loss satisfies $\nabla\WB_{p,\varepsilon}^p(a,a)=0$. Hence, to obtain the 
distributional limit of the empirical Sinkhorn loss it is necessary to apply a second-order delta-method yielding an asymptotic distribution which is not Gaussian anymore.

\begin{thm}
\label{Th_CLT_loss_H0_one}
Let $a=b$ be a probability distribution on $\Sigma_N$, and denote by $\hat{a}_n$ an empirical measures obtained by independent sampling data from $a$.  Then, as $n$ tends to infinity, the following asymptotic result holds
\begin{equation}
\label{null_one}
n \WB_{p,\varepsilon}^p(\hat{a}_n,a)\overset{\LL}{\longrightarrow} \ \frac{1}{2}\sum_{i=1}^N \lambda_i\chi^2_{\scriptscriptstyle{i}}(1)
\end{equation}
where $\lambda_1,\ldots,\lambda_N$ are the non-negative eigenvalues of the matrix $$\Sigma(a)^{1/2}\partial_{\scriptscriptstyle{11}}^2\WB_{p,\varepsilon}^p(a,a) \Sigma(a)^{1/2},$$ and $\chi^2_{\scriptscriptstyle{1}}(1),\ldots,\chi^2_{\scriptscriptstyle{N}}(1)$ are independent random variables with chi-squared distribution of degree $1$.
\end{thm}
\begin{proof}
From Corollary \ref{cor:gradient_null}, we have that $\nabla\WB_{p,\varepsilon}^p(a,a)=0$. In order to apply a second order delta-method, the Hessian matrix $\nabla^2 \WB_{p,\varepsilon}^p(a,b)$ of the Sinkhorn loss $\WB_{p,\varepsilon}^p(a,b)$ needs to be non-singular in the neighborhood of $a=b$. Note that the Sinkhorn loss is at least $C^3$ (admitting a third continuous differential) on the interior of its domain, as proved in Theorem 2 by \cite{luise2018differential}. Moreover, since the function $a\mapsto W_{p,\varepsilon}^p(a,b)$ is $\varepsilon$-strongly convex (Theorem 3.4, \cite{BCP18}) and $a\mapsto - \frac{1}{2}W_{p,\varepsilon}^p(a,a)$ is (strictly) convex (Proposition 4, \cite{feydy2018interpolating}), we have that the Hessian matrix of $a\mapsto \WB_{p,\varepsilon}^p(a,b)$ is non-singular. We can thus apply Theorem 17 in \cite{sourati2017asymptotic} which states that from second order delta-method, the distributional limits of $n \WB_{p,\varepsilon}^p(\hat{a}_n,a)$ is given by
$$\frac{1}{2}\NN(0,\Sigma(a))^T\partial_{\scriptscriptstyle{11}}^2 \WB_{p,\varepsilon}^p(a,a)\NN(0,\Sigma(a))$$
that can be rewritten as
$$\frac{1}{2}\sum_{i=1}^N \lambda_i\chi^2_{\scriptscriptstyle{i}}(1),$$
where $\lambda_1,\ldots,\lambda_N$ are the eigenvalues of the matrix $\Sigma(a)^{1/2}\partial_{\scriptscriptstyle{11}}^2\WB_{p,\varepsilon}^p(a,a) \Sigma(a)^{1/2}$. This concludes the distributional limit presented in relation \eqref{null_one}.
\end{proof}

In the two samples case, the Hessian matrix is not guaranteed to be non-singular, in which case the asymptotic distribution is degenerated. Nevertheless, we have the following theorem.

\begin{thm}
\label{Th_CLT_loss_H0_two}
Let $a=b$ be a probability distribution on $\Sigma_N$, and denote by $\hat{a}_n,\tilde{a}_m$ two empirical measures obtained by independent sampling data from $a$. Then, let us write the Hessian matrix
$$\nabla^2 \WB_{p,\varepsilon}^p(a,b) =\begin{pmatrix}
A &B\\
B & C 
\end{pmatrix},$$
with $A=\partial_{\scriptscriptstyle{11}}^2 \WB_{p,\varepsilon}^p(a,b), C =\partial_{\scriptscriptstyle{22}}^2  \WB_{p,\varepsilon}^p(a,b)$ and $B= \partial_{\scriptscriptstyle{12}}^2  W_{p,\varepsilon}^p(a,b)$.
If its Schur complement $S = C - B^TA^{-1}B$ is non-singular in a neighborhood of $a=b$, then one has for $m/(n+m)\rightarrow \gamma\in (0,1)$   as $\min(n,m) \to + \infty$
\begin{equation}
\label{null_two}
\frac{nm}{n+m}\ \WB_{p,\varepsilon}^p(\hat{a}_n,\tilde{a}_m)\overset{\LL}{\longrightarrow}  \ \frac{1}{2}\sum_{i=1}^N \tilde{\lambda}_i\chi^2_{\scriptscriptstyle{i}}(1),
\end{equation}
where  $\tilde{\lambda}_1,\ldots,\tilde{\lambda}_N$ are the eigenvalues of the matrix of size $\R^{2N}\times \R^{2N}$ given by
$$
(\sqrt{\gamma}\Sigma(a)^{1/2},\sqrt{1-\gamma}\Sigma(a)^{1/2})\nabla^2\WB_{p,\varepsilon}^p(a,a)(\sqrt{\gamma}\Sigma(a)^{1/2},\sqrt{1-\gamma}\Sigma(a)^{1/2}),
$$
and $\chi^2_{\scriptscriptstyle{1}}(1),\ldots,\chi^2_{\scriptscriptstyle{N}}(1)$ are independent random variables with chi-squared distribution of degree $1$.
\end{thm}

\begin{proof}
As in the proof of Theorem \ref{Th_CLT_loss_H0_one}, we have that both $A$ and $C$ are $\varepsilon$-strongly convex and therefore non-singular. The determinant $det(\nabla^2 \WB_{p,\varepsilon}^p(a,b))=det(A)det(S)$ is therefore non-zero in a neighborhood of $a=b$ if and only if the Schur complement $S$ is invertible in a neighborhood of $a=b$. Therefore, applying  Theorem 17 in \cite{sourati2017asymptotic} as in the one sample case, we obtain the  distributional limit \eqref{null_two}. This completes the proof of Theorem \ref{Th_CLT_loss_H0_two}.
\end{proof}

\begin{rmq}
A sufficient condition to ensure the non-singularity of the Schur matrix $S$ comes from the $\varepsilon$-strong convexity of $A$ and $C$, implying that for any $x\in\R^N$
$$
x^TSx=x^TCx-x^T B^T A B x>\varepsilon\Vert x\Vert^2-\varepsilon^{-1}\Vert Bx\Vert^2.
$$
A sufficient condition for the non-singularity of $S$ is therefore $\varepsilon> \sup_{x \in \R^N} \frac{\Vert Bx\Vert}{\Vert x\Vert}$ at the points $a=b$. Remark that since the global minimum is attained in the critical points $a=b$, we have that the Hessian $\WB_{p,\varepsilon}^p$ is symmetric semi-definite positive at these points. Therefore its Schur complement $S = C - B^TA^{-1}B$ is also semi-definite positive (see \textit{e.g.} Section A.5.5. in \cite{boyd2004convex}).
\end{rmq}

\subsubsection{Convergence in probability}
Limits for empirical Sinkhorn loss can again be established from a corollary of the Delta-method as done in Theorem \ref{th:cv_prob_div}.

\begin{thm}\label{th:cv_prob_loss}
Using the same  notations as introduced in the statement of Theorem \ref{Th_CLT_loss_H1}, the following asymptotic results hold for all $a,b\in\Sigma_N$.
\begin{enumerate}
\item One sample. As $n \to +\infty$, one has that
$$\sqrt{n}\left(\WB_{p,\varepsilon}^p(\hat{a}_n,b)-\WB_{p,\varepsilon}^p(a,b)- \langle \hat{a}_n-a,\ualpha_{\varepsilon}^{a,b}-\frac{1}{2}(\ualpha_{\varepsilon}^{a,a}+ \vbeta_{\varepsilon}^{a,a})\rangle\right) \overset{\P}{\longrightarrow} 0.$$
\item Two samples - For $\rho_{n,m}=\sqrt{(nm/(n+m))}$ and $m/(n+m)\rightarrow \gamma\in (0,1)$ as $\min(n,m) \to + \infty$, one has that
\begin{align*}
\sqrt{n}(\WB_{p,\varepsilon}^p(\hat{a}_n,\hat{b}_m) -\WB_{p,\varepsilon}^p(a,b)-&(\sqrt{\gamma}  \langle \hat{a}_n-a,\ualpha_{\varepsilon}^{a,b}-\frac{1}{2}(\ualpha_{\varepsilon}^{a,a}+ \vbeta_{\varepsilon}^{a,a})\rangle \\
+ & \sqrt{1-\gamma} \langle \hat{b}_m-b,\vbeta_{\varepsilon}^{a,b}-\frac{1}{2}(\ualpha_{\varepsilon}^{b,b}+ \vbeta_{\varepsilon}^{b,b})\rangle))\overset{\P}{\longrightarrow} 0.
\end{align*}
\end{enumerate}
\end{thm}

Note that in the case $a=b$, this simplifies into
\begin{align*}
&\sqrt{n} \WB_{p,\varepsilon}^p(\hat{a}_n,a) \overset{\P}{\longrightarrow} 0\\
& \rho_{n,m} \WB_{p,\varepsilon}^p(\hat{a}_n,\hat{b}_m)\overset{\P}{\longrightarrow} 0.
\end{align*}

\subsection{Link with unregularized optimal transport }

A natural question that arises is the behavior of distributional limits when we let $\varepsilon$ tends to $0$ at an appropriate rate depending on the sample size. Under such conditions, we recover the distributional limit given by  Theorem 1 in Sommerfeld and Munk \cite{sommerfeld2016inference}  in the setting of  unregularized optimal transport.
\begin{thm}
\label{Th_CLT_unreg}
Suppose that $\XX \subset \R^q$, and consider the cost matrix $C$ such that $c_{ij} = \|x_i-x_j\|^p$ where  $\| \cdot \|$ stands for the Euclidean norm. We recall that $\SS_0(a,b)\subset\R^N\times\R^N$ is the set of optimal solutions of the dual problem \eqref{dual} for $\varepsilon=0$.
\begin{enumerate}
\item One sample. Suppose that $(\varepsilon_n)_{n \geq 1}$ is a sequence of positive reals tending to zero such that
\begin{equation}
\lim_{n \to + \infty} \sqrt{n} \varepsilon_n\log(1/\varepsilon_{n}) = 0. \label{eq:condepsn}
\end{equation}
Then, we have that
\begin{equation}
\label{alt_one_unreg}
\sqrt{n}(\WB_{p,\varepsilon_n}^p(\hat{a}_n,b)-\WB_{p,\varepsilon_n}^p(a,b))\overset{\LL}{\longrightarrow} \ \umax{(\ualpha,\vbeta)\in \SS_0(a,b)}\ \langle G,\ualpha\rangle.
\end{equation}
\item Two samples. Suppose that $(\varepsilon_{n,m})$ is a sequence of positive reals tending to zero as $\min(n,m) \to + \infty$ such that
\begin{equation}
\lim_{\min(n,m) \to + \infty} \sqrt{\rho_{n,m}} \varepsilon_{n,m}\log(1/\varepsilon_{n,m}) = 0, \label{eq:condepsnm}
\end{equation}
for $\rho_{n,m}=\sqrt{(nm/(n+m))}$ and $m/(n+m)\rightarrow \gamma\in (0,1)$. Then, one has that
\begin{equation}
\label{alt_two_unreg}
\rho_{n,m}(\WB_{p,\varepsilon_{n,m}}^p(\hat{a}_n,\hat{b}_m)-\WB_{p,\varepsilon_{n,m}}^p(a,b))\overset{\LL}{\longrightarrow} \umax{(\ualpha,\vbeta)\in \SS_0(a,b)}\ \sqrt{\gamma}\langle G,\ualpha\rangle + \sqrt{1-\gamma} \langle H,\vbeta\rangle.
\end{equation}
\end{enumerate}
\end{thm}
\begin{proof}
We will only prove the one sample case as both proofs work similarly. For that purpose, let us consider the decomposition
\begin{align}
\sqrt{n}(\WB_{p,\varepsilon}^p(\hat{a}_n,b)-&\WB_{p,\varepsilon}^p(a,b))=\sqrt{n}(\WB_{p,\varepsilon}^p(\hat{a}_n,b)-W_{p}^p(\hat{a}_n,b) \label{eq:decompW})\\
+&\sqrt{n}(W_{p}^p(\hat{a}_n,b)-W_{p}^p(a,b))+\sqrt{n}(W_{p}^p(a,b)-\WB_{p,\varepsilon}^p(a,b)). \nonumber
\end{align}
From Theorem 1 in \cite{sommerfeld2016inference}, we have that
\begin{equation}
\sqrt{n}(W_{p}^p(\hat{a}_n,b)-W_{p}^p(a,b))\overset{\LL}{\longrightarrow} \ \umax{(\ualpha,\vbeta)\in \SS_0(a,b)} \langle G,\ualpha\rangle. \label{eq:conveps1}
\end{equation}
Since $\XX$ is a finite set, it follows that the cost $c$ is a $L$-Lipschitz function separately in $x \in \XX$ and $y \in \XX$ with respect to the Euclidean distance. Therefore, it satisfies the assumptions of Theorem 1 in \cite{genevay2018sample} that gives a bound on the error between the Sinkorn divergence and the unregularized transport for a given pair of distributions. It follows that for any $a,b\in\Sigma_N$ (possibly random),  
$$0\leq W_{p,\varepsilon}^p(a,b)-W_{p}^p(a,b)\leq 2\varepsilon q \log\left(\frac{e^2 L \diam(\XX)}{\varepsilon\sqrt{q}}\right)$$
where $q$ is the dimension of the support space, and $\diam(\XX)$ is the diameter of $\XX$ (i.e. $\diam(\XX)=\sup_{x,y\in\XX} \Vert x-y\Vert$) which is always finite in the discrete case.
Then, as soon as the sequence  $(\varepsilon_n)_{n \geq 1}$  satisfies \eqref{eq:condepsn}, we obtain that
\begin{equation}
\sup_{(a,b) \in \Sigma_N \times \Sigma_N} \sqrt{n}(W_{p,\varepsilon_n}^p(a,b)-W_{p}^p(a,b))\xrightarrow[n\rightarrow \infty]{}  0. \label{eq:boundeps}
\end{equation}
By definition of the Sinkhorn loss, one has that $ W_{p,\varepsilon}^p(a,b) - \WB_{p,\varepsilon}^p(a,b) =  \frac{1}{2} \left(W_{p,\varepsilon}^p(a,a) +W_{p,\varepsilon}^p(b,b)\right)$. Therefore, using the upper bound \eqref{eq:boundeps}, we get
\begin{equation}
\sqrt{n}(\WB_{p,\varepsilon_n}^p(\hat{a}_n,b)-W_{p}^p(\hat{a}_n,b) \xrightarrow[n\rightarrow \infty]{\mbox{a.s.}} 0 \mbox{ and } \sqrt{n}(W_{p}^p(a,b)-\WB_{p,\varepsilon_n}^p(a,b)) \xrightarrow[n\rightarrow \infty]{\mbox{a.s.}}  0. \label{eq:conveps2}
\end{equation}
Combining \eqref{eq:decompW} with \eqref{eq:conveps1}  and \eqref{eq:conveps2}, and using  Slutsky's theorem allow to complete the proof of Theorem \ref{Th_CLT_unreg}.

\end{proof}

\section{Use of the bootstrap for statistical inference}\label{sec:bootstrap}

The results obtained in Section \ref{sec:results} on the distribution of the empirical Sinkhorn divergence and  Sinkhorn loss are only asymptotic. It is thus of interest to estimate their non-asymptotic distribution using a bootstrap procedure. The bootstrap consists in drawing new samples from an empirical distribution $\hat{\P}_n$ that has been obtained from an unknown distribution $\P$. Therefore, conditionally on $\hat{\P}_n$,  it allows to obtain new observations (considered as approximately sampled from $\P$) that can be used to approximate the distribution of a test statistics using Monte-Carlo experiments. We refer to \cite{efron93bootstrap} for a general  introduction to the bootstrap procedure.

We can apply the delta-method to prove the consistency of the bootstrap in our setting using the bounded Lipschitz metric as defined below.
\begin{defi}
The Bounded Lipschitz (BL) metric between two probability measures   $\mu,\nu$  supported on $\Omega$  is defined by
$$d_{BL}(\mu,\nu)=\underset{h\in BL_1(\Omega)}{\sup} \int_{\Omega} h d(\mu-\nu)$$
where $BL_1(\Omega)$ is the set of real functions $\Omega\rightarrow \R$ such that $\Vert h\Vert_{\infty}+\Vert h\Vert_{\mbox{Lip}}\leq 1$.
\end{defi}
Our main result on the consistency of bootstrap samples can then be stated.  Notice that similar results for the Sinkhorn divergence are obtained straightforward using the same arguments.
\begin{thm}
\label{bootstrap}
Let $a\neq b$ be in the simplex $\Sigma_N$. For $X_1,\ldots,X_n\overset{\LL}{\sim}a$ and $Y_1,\ldots,Y_m\overset{\LL}{\sim}b$, let $\hat{a}_n^{\ast}$ (resp. $\hat{b}_m^{\ast}$) be a bootstrap empirical distribution sampled from $\hat{a}_n$  (resp. $\hat{b}_m$) of size $n$ (resp. $m$).
\begin{enumerate}
\item One sample case:
$\sqrt{n}(\WB_{p,\varepsilon}^p(\hat{a}^{\ast}_n,b)-\WB_{p,\varepsilon}^p(\hat{a}_n,b))$ converges  in distribution (conditionally on $X_1,\ldots,X_n$) to $\langle G, \ualpha_{\varepsilon}^{a,b}-\frac{1}{2}(\ualpha_{\varepsilon}^{a,a}+ \vbeta_{\varepsilon}^{a,a})\rangle$ for the BL metric, in the sense that
\begin{align*}
\underset{h\in BL_1(\R)}{\sup}\vert & \E [h (\sqrt{n}(\WB_{p,\varepsilon}^p(\hat{a}^{\ast}_n,b)-W_{p,\varepsilon}^p(\hat{a}_n,b)))\vert X_1,\ldots,X_n]-\\
&\E [ h\langle G, \ualpha_{\varepsilon}^{a,b}-\frac{1}{2}(\ualpha_{\varepsilon}^{a,a}+ \vbeta_{\varepsilon}^{a,a})\rangle]\vert \overset{\P}{\longrightarrow} 0.
\end{align*}
\item Two samples case:
$\rho_{n,m}(\WB_{p,\varepsilon}^p(\hat{a}^{\ast}_n,\hat{b}^{\ast}_m)-\WB_{p,\varepsilon}^p(\hat{a}_n,\hat{b}_m))$ converges  in distribution (conditionally on $X_1,\ldots,X_n,\, Y_1,\ldots,Y_m$) to $$\sqrt{\gamma}\langle G, \ualpha_{\varepsilon}^{a,b}-\frac{1}{2}(\ualpha_{\varepsilon}^{a,a}+ \vbeta_{\varepsilon}^{a,a})\rangle + \sqrt{1-\gamma}\langle H, \vbeta_{\varepsilon}^{a,b}-\frac{1}{2}(\ualpha_{\varepsilon}^{b,b}+ \vbeta_{\varepsilon}^{b,b})\rangle$$ for the BL metric, in the sense that
\begin{align*}
& \underset{h\in BL_1(\R)}{\sup}\vert \E [h(\rho_{n,m}(\WB_{p,\varepsilon}^p(\hat{a}^{\ast}_n,\hat{b}^{\ast}_m)-  \WB_{p,\varepsilon}^p(\hat{a}_n,\hat{b}_m)))\vert X_1,\ldots,X_n,Y_1,\ldots,Y_m] \\
& -\E [ h(\sqrt{\gamma}\langle G, \ualpha_{\varepsilon}^{a,b}-\frac{1}{2}(\ualpha_{\varepsilon}^{a,a}+ \vbeta_{\varepsilon}^{a,a})\rangle + \sqrt{1-\gamma}\langle H, \vbeta_{\varepsilon}^{a,b}-\frac{1}{2}(\ualpha_{\varepsilon}^{b,b}+ \vbeta_{\varepsilon}^{b,b})\rangle)]\vert \overset{\P}{\longrightarrow} 0
\end{align*}
\end{enumerate}
\end{thm}
\begin{proof}
We only prove the one sample case since the convergence for the two samples case can be shown with similar arguments. We know that $\sqrt{n}(\hat{a}_n-a)$ tends in distribution to $G\sim\mathcal{N}(0,\Sigma(a))$. Moreover $\sqrt{n}(\hat{a}_n^{\ast}-\hat{a}_n)$ converges (conditionally on $X_1,\ldots,X_n$) in distribution to $G$ by Theorem 3.6.1 in \cite{van1996weak}. Then, applying Theorem 3.9.11 in \cite{van1996weak} on the consistency of the delta-method combined with the  bootstrap allows us to obtain the statement of the present Theorem \ref{bootstrap} in  the case $a \neq b$.
\end{proof}

As explained in \cite{chen2018inference}, the standard bootstrap fails under first order degeneracy, meaning for the null hypothesis case $a=b$. However, the authors propose a corrected version -called the Babu correction- of the bootstrap in their Theorem 3.2 given for the one sample case by
\begin{align*}
\underset{h\in BL_1(\R)}{\sup}\vert & \E [h (n\{\WB_{p,\varepsilon}^p( \hat{a}^{\ast}_n,a)-\WB_{p,\varepsilon}^p(\hat{a}_n,a) -\partial_{\scriptscriptstyle{1}}\WB_{p,\varepsilon}^p(\hat{a}_n,a)(\hat{a}^{\ast}_n-\hat{a}_n,a) \})) \vert X_1,\ldots,X_n]-\\
&-\E[h(\partial_{\scriptscriptstyle{11}}^2\WB_{p,\varepsilon}^p(a,a)(G,a)] \vert \overset{\P}{\longrightarrow} 0,
\end{align*}
and for the two samples case by
\begin{align*}
\underset{h\in BL_1(\R)}{\sup}\vert & \E [h (n\{\WB_{p,\varepsilon}^p( \hat{a}^{\ast}_n,\hat{b}^{\ast}_m)-\WB_{p,\varepsilon}^p(\hat{a}_n,\hat{b}_m) - \nabla\WB_{p,\varepsilon}^p(\hat{a}_n,\hat{b}_m)(\hat{a}^{\ast}_n-\hat{a}_n,\hat{b}^{\ast}_m-\hat{b}_m) \})) \vert X_1,\ldots,X_n]-\\
&-\E[h(\nabla^2\WB_{p,\varepsilon}^p(a,a)(\sqrt{\gamma}G,\sqrt{1-\gamma}G)] \vert \overset{\P}{\longrightarrow} 0.
\end{align*}
Note that most of the requirements to apply Theorem 3.2 in  \cite{chen2018inference} are trivial since the distributions are defined on a subset of $\R^N$ and the function $(a,b)\mapsto W_{p,\varepsilon}^p(a,b)$ is twice differentiable on all $\Sigma_N\times\Sigma_N$. However, the (\textit{Assumption 3.3 in \cite{chen2018inference}}) on the second derivative  requires a finer study that is left for future work.  Hence, we stress that the Babu-bootstrap approach that we use in our numerical experiments is missing theoretical guarantees. Nevertheless, the results reported from our experiments on simulated and real data illustrate its correctness.

As
\begin{align*}
\partial_{\scriptscriptstyle{1}}\WB_{p,\varepsilon}^p(\hat{a}_n,a)(\hat{a}^{\ast}_n-\hat{a}_n,a) & = \langle\ualpha^{\hat{a}_n,a},\hat{a}_n^{\ast}-\hat{a}_n\rangle\\
\nabla\WB_{p,\varepsilon}^p(\hat{a}_n,\hat{b}_m)(\hat{a}^{\ast}_n-\hat{a}_n,\hat{b}^{\ast}_m-\hat{b}_m) & = \langle\ualpha^{\hat{a}_n,\hat{b}_m},\hat{a}_n^{\ast}-\hat{a}_n\rangle + \langle\vbeta^{\hat{a}_n,\hat{b}_m},\hat{b}_m^{\ast}-\hat{b}_m\rangle
 \end{align*}
 we can reformulate the Babu bootstrap as follows.
\begin{enumerate}
\item One sample case. For  $(\ualpha_{\varepsilon}^{\hat{a}_n,a},\vbeta_{\varepsilon}^{\hat{a}_n,a})\in\SS_{\varepsilon}(\hat{a}_n,a)$, we have that
\begin{equation}
\label{eq:boot_null_loss_one}
n \left\{\WB_{p,\varepsilon}^p( \hat{a}^{\ast}_n,a)-\WB_{p,\varepsilon}^p(\hat{a}_n,a)-\langle\ualpha^{\hat{a}_n,a},\hat{a}_n^{\ast}-\hat{a}_n\rangle\right\}
\end{equation}
converges  in distribution (conditionally on $X_1,\ldots,X_n,$) to $\partial_{\scriptscriptstyle{11}}^2\WB_{p,\varepsilon}^p(a,a)(G,a)$ for the BL metric.

\item Two samples case. For $(\ualpha_{\varepsilon}^{\hat{a}_n,\hat{b}_m},\vbeta_{\varepsilon}^{\hat{a}_n,\hat{b}_m})\in\SS_{\varepsilon}(\hat{a}_n,\hat{b}_m)$ and $m/(n+m)\rightarrow \gamma\in (0,1)$, the quantity 
\begin{equation}
\label{eq:boot_null_loss_two}
\frac{nm}{n+m} \left\{\WB_{p,\varepsilon}^p(\hat{a}^{\ast}_n,\hat{b}^{\ast}_m)-\WB_{p,\varepsilon}^p(\hat{a}_n,\hat{b}_m)- (\langle\ualpha^{\hat{a}_n,\hat{b}_m},\hat{a}_n^{\ast}-\hat{a}_n\rangle + \langle\vbeta^{\hat{a}_n,\hat{b}_m},\hat{b}_m^{\ast}-\hat{b}_m\rangle)\right\}
\end{equation}
converges  in distribution (conditionally on $X_1,\ldots,X_n,\, Y_1,\ldots,Y_m$) to $$\nabla^2\WB_{p,\varepsilon}^p(a,a)(\sqrt{\gamma}G,\sqrt{1-\gamma}G)$$ for the BL metric.
\end{enumerate}

\section{Numerical experiments with synthetic data}\label{sec:numeric}

We propose to illustrate Theorem \ref{Th_CLT_loss_H1}, Theorem \ref{Th_CLT_loss_H0_one}, Theorem \ref{Th_CLT_loss_H0_two} and  Theorem \ref{bootstrap} with simulated data consisting of random measures supported on a $l \times l$ square lattice (of regularly spaced points) $(x_{i})_{i=1,\ldots,N}$ in $\R^{2}$ (with $N = l^2$) for $l$ ranging from 5 to 20. We use the squared  Euclidean distance as the cost function $C$ which therefore scales with the size of the grid. The range of interesting values for $\varepsilon$ is thus closely linked to the size of the grid, as it can be seen in the expression of $K=\exp(-C/\varepsilon-\1_{N\times N})$. Hence, $\varepsilon=100$ for a $5\times 5$ grid corresponds to more regularization than $\varepsilon=100$ for a $20\times 20$ grid.

We ran our experiments on Matlab using the accelerated version \cite{thibault2017overrelaxed}\footnote{\url{http://www.math.u-bordeaux.fr/~npapadak/GOTMI/codes.html}}  of the  Sinkhorn transport algorithm \cite{cuturi}.  Furthermore, we considered the numerical logarithmic stabilization described in \cite{schmitz2017wasserstein} which allows to handle relatively small values of $\varepsilon$. Indeed, in small regularization regimes, the Sinkhorn algorithm quickly becomes unstable, even more for large grids with a small number of observations.

\subsection{Convergence in distribution}\label{subsec:CVinLaw}
We first illustrate the convergence in distribution of the empirical Sinkhorn loss (as stated in Theorem \ref{Th_CLT_loss_H1}) for the hypothesis $a  \neq b$ with either one sample or two samples.

\subsubsection{Alternative $a\neq b$ - One sample.}
We consider the case where $a$ is the uniform distribution on a square grid and
$$
b \propto \1_N + \theta (1,2,\ldots,N)
$$
is a distribution with linear trend depending on a slope parameter $\theta \geq 0$ that is fixed to $0.5$, see Figure \ref{fig:distributions}.
\begin{figure}[ht!]
\centering
\includegraphics[width=0.4 \textwidth,height=0.3\textwidth]{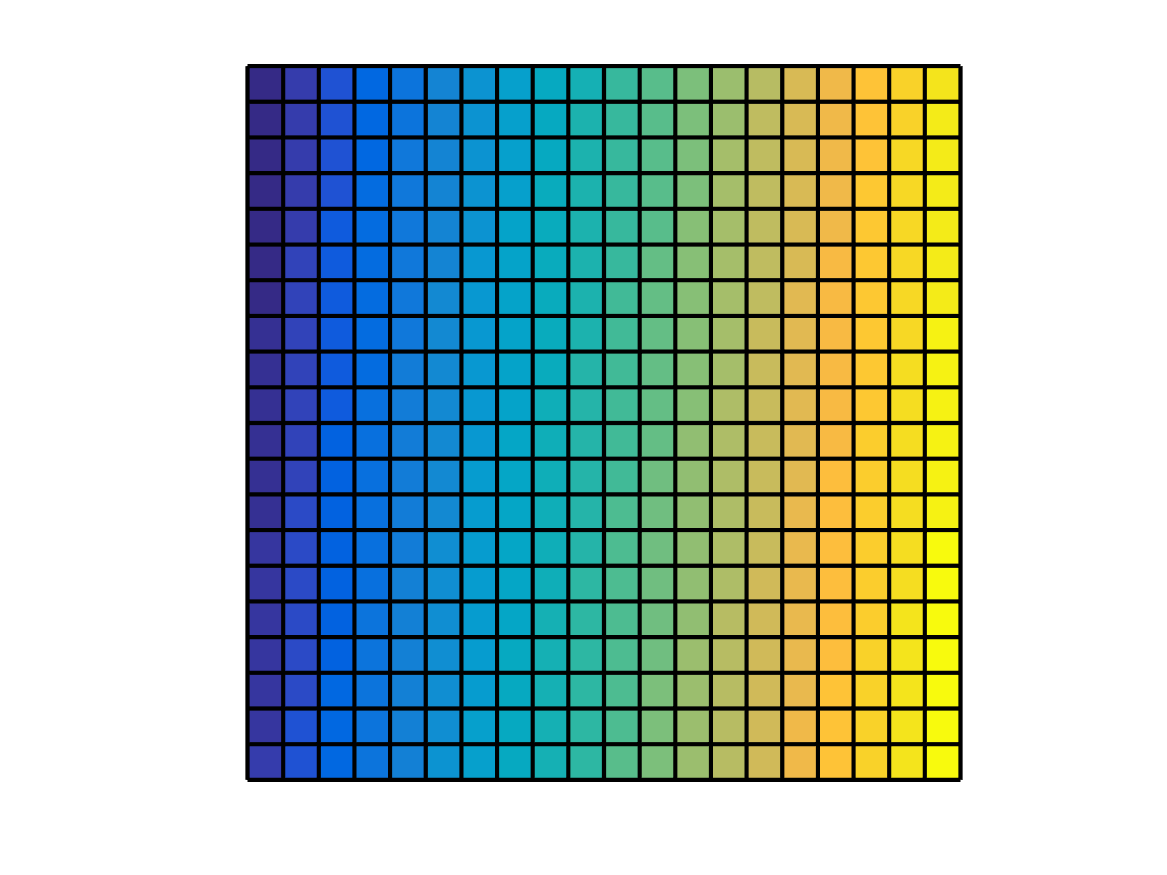}
\caption{\label{fig:distributions}Example of a distribution $b$ with linear trend (with slope parameter $\theta=0.5$ on a $20\times 20$ grid).}
\end{figure}

We generate $M = 10^3$ empirical distributions $\hat{a}_n$ (such that $n \hat{a}_n$ follows a multinomial distribution with parameter $a$) for different values of $n$ and grid size. In this way, we obtain $M$ realizations of $ \sqrt{n}(\WB_{p,\varepsilon}^p(\hat{a}_n,b)-\WB_{p,\varepsilon}^p(a,b))$, and we use a kernel density estimate (with a data-driven bandwidth) to compare the distribution of these realizations to the density of the Gaussian distribution $\langle G,\ualpha_{\varepsilon}^{a,b}-1/2(\ualpha_{\varepsilon}^{a,a}+\vbeta_{\varepsilon}^{a,a})\rangle$. The results are reported in Figure \ref{fig:H1one5} (grid $5\times 5$) and Figure \ref{fig:H1one20} (grid $20\times 20$).
It can be seen that the convergence of the empirical Sinkhorn loss to its asymptotic distribution ($n\to\infty$) is relatively fast. 
\newcommand{\sidecap}[1]{ {\begin{sideways}\parbox{0.2\textwidth}{\centering #1}\end{sideways}} }
\begin{figure}[ht!]
\centering
\begin{tabular}{cccc}
 & $n=10^2$ & $n=10^3$ & $n=10^4$ \\
\sidecap{$\varepsilon=1$} & \includegraphics[width=0.27 \textwidth,height=0.32\textwidth]{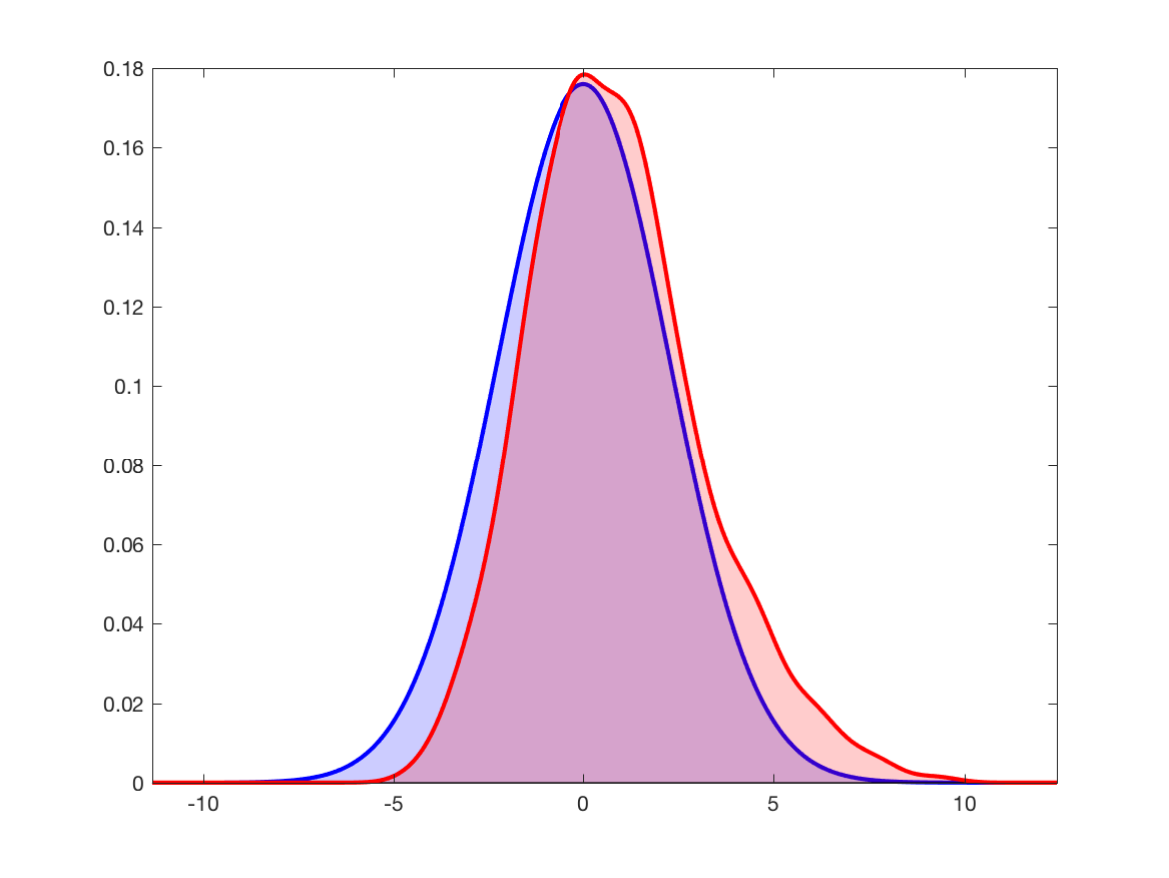} &
\includegraphics[width=0.27 \textwidth,height=0.32\textwidth]{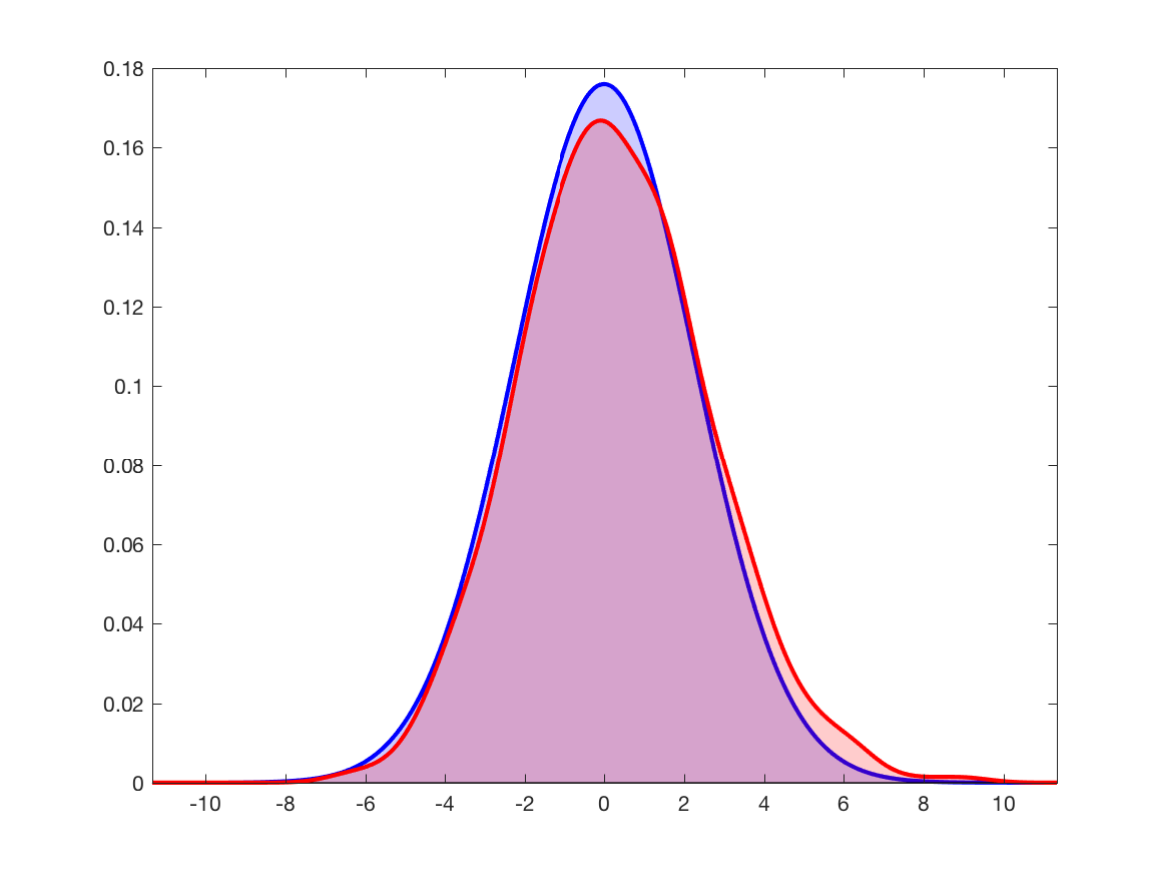} &
\includegraphics[width=0.27 \textwidth,height=0.32\textwidth]{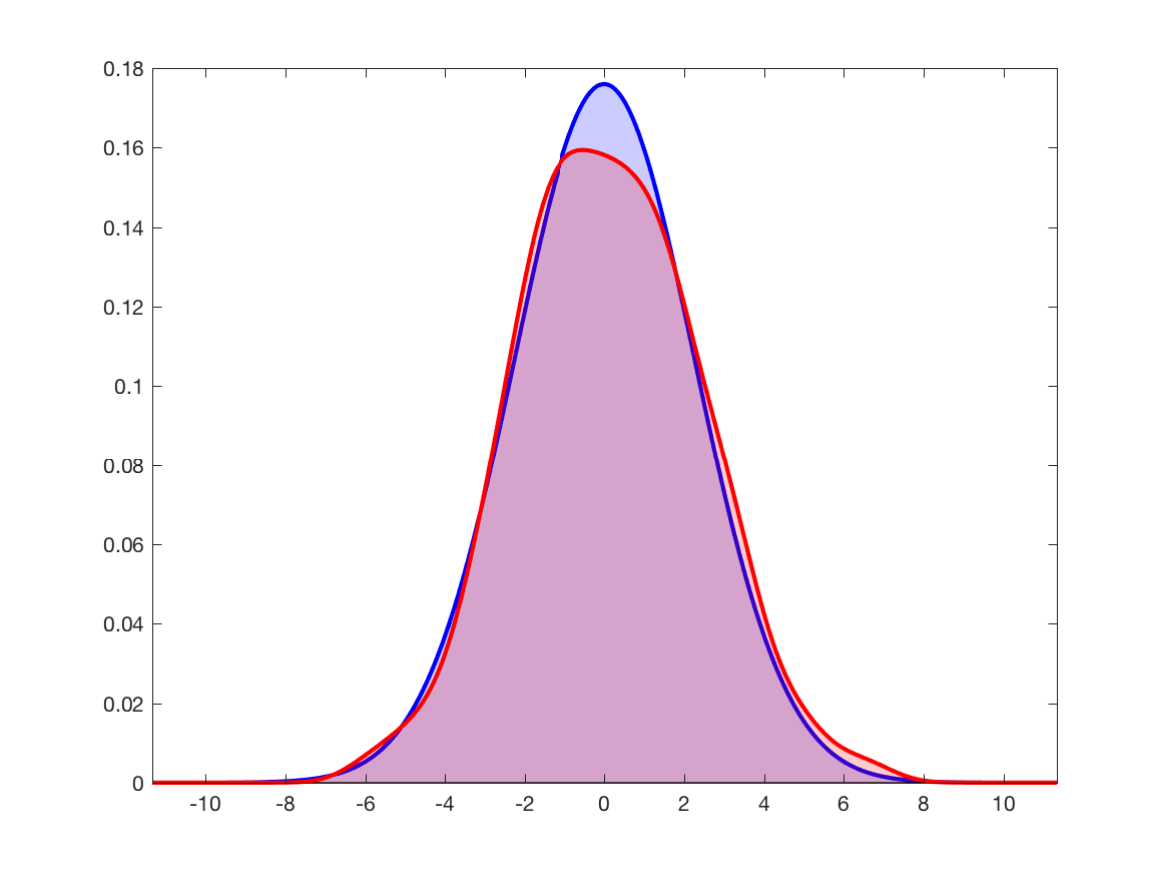} \\

\sidecap{$\varepsilon=10$} & \includegraphics[width=0.27 \textwidth,height=0.32\textwidth]{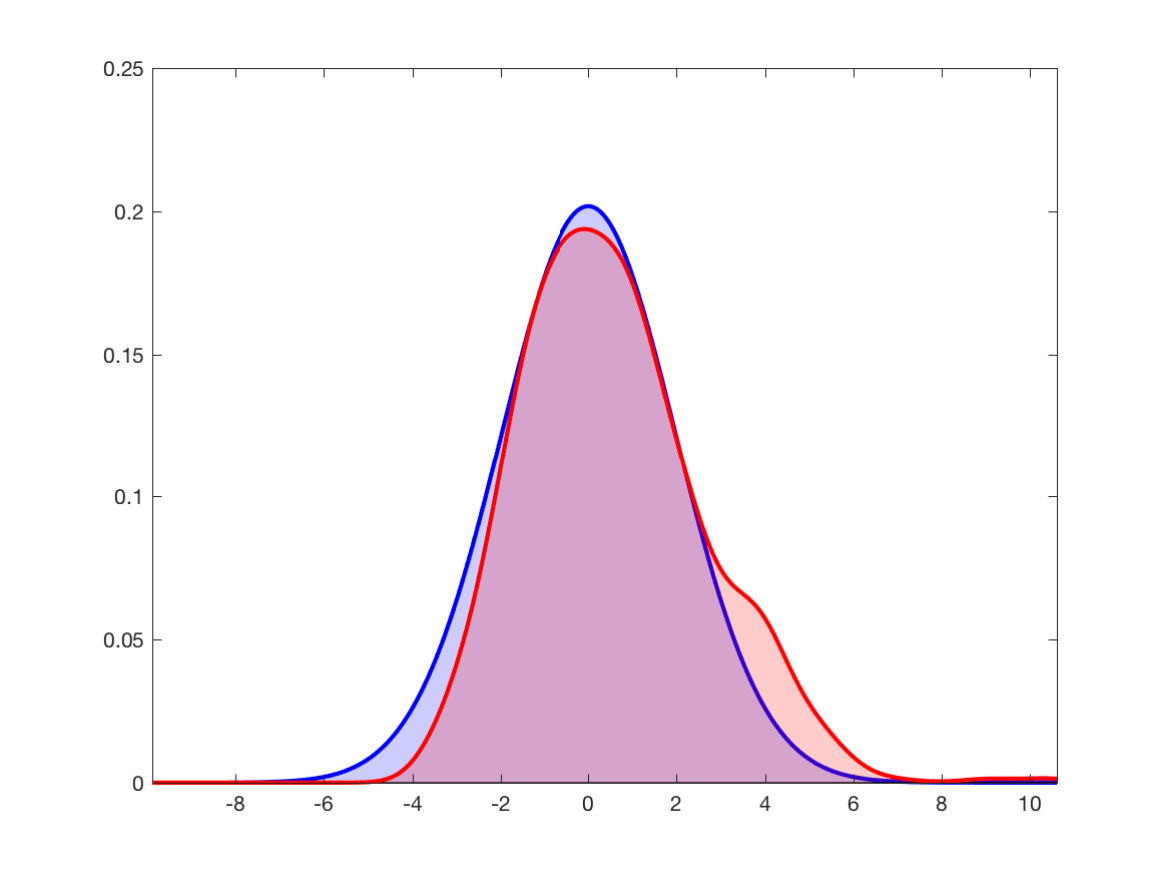} &
\includegraphics[width=0.27 \textwidth,height=0.32\textwidth]{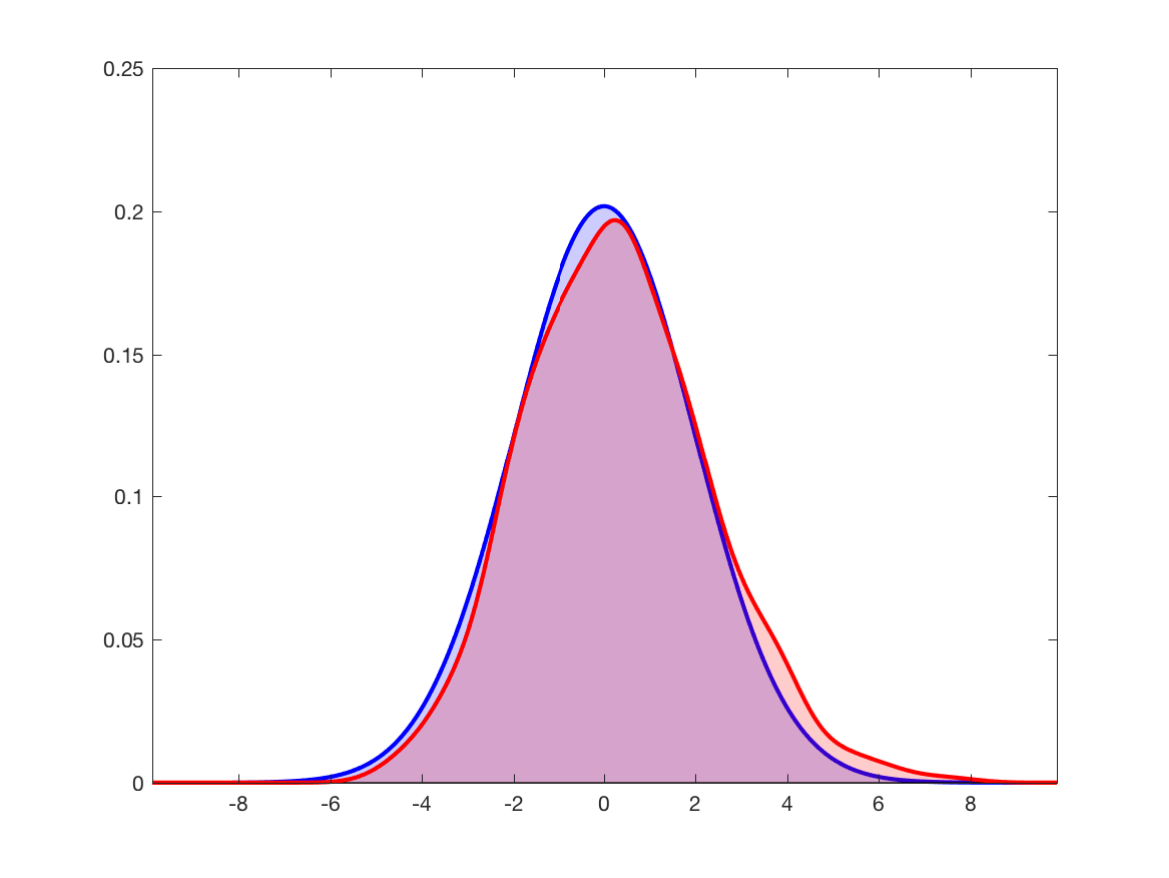} &
\includegraphics[width=0.257 \textwidth,height=0.32\textwidth]{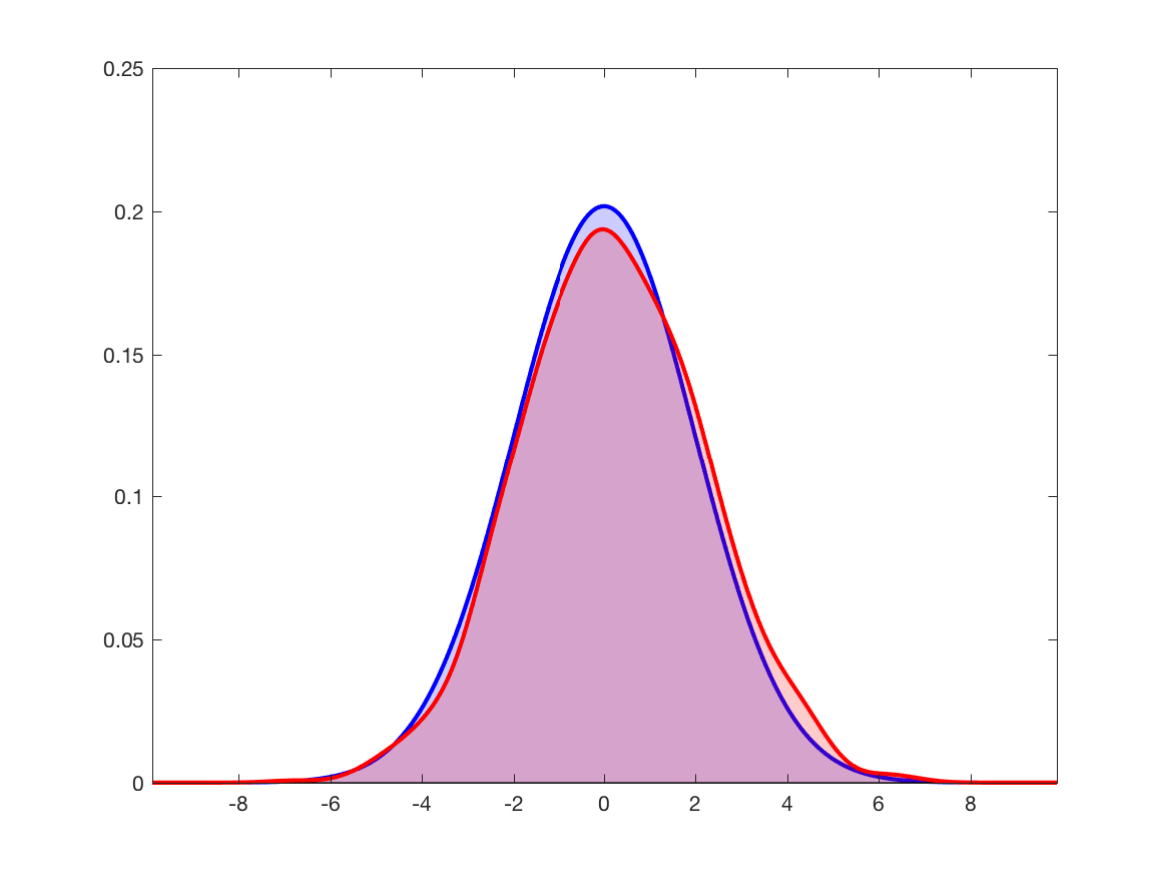} \\

\sidecap{$\varepsilon=100$} & \includegraphics[width=0.27 \textwidth,height=0.32\textwidth]{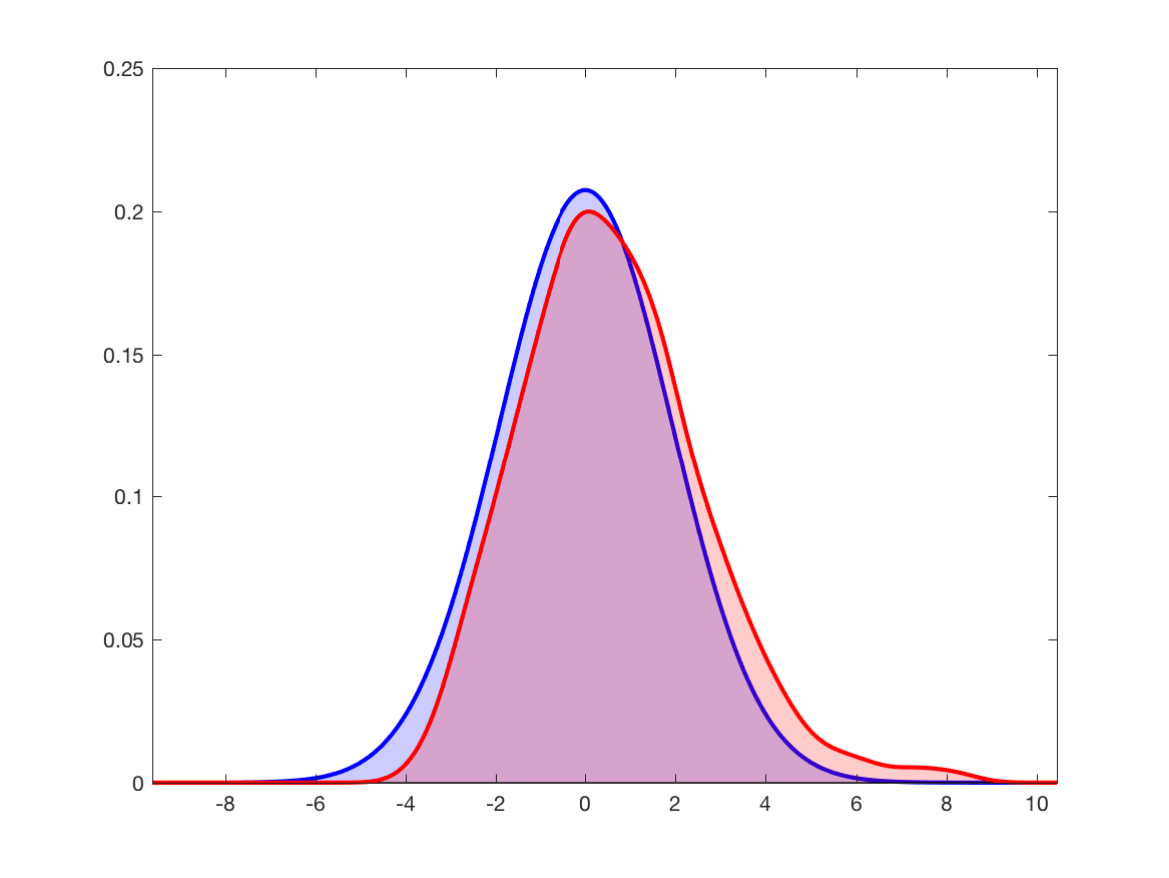} &
\includegraphics[width=0.27 \textwidth,height=0.32\textwidth]{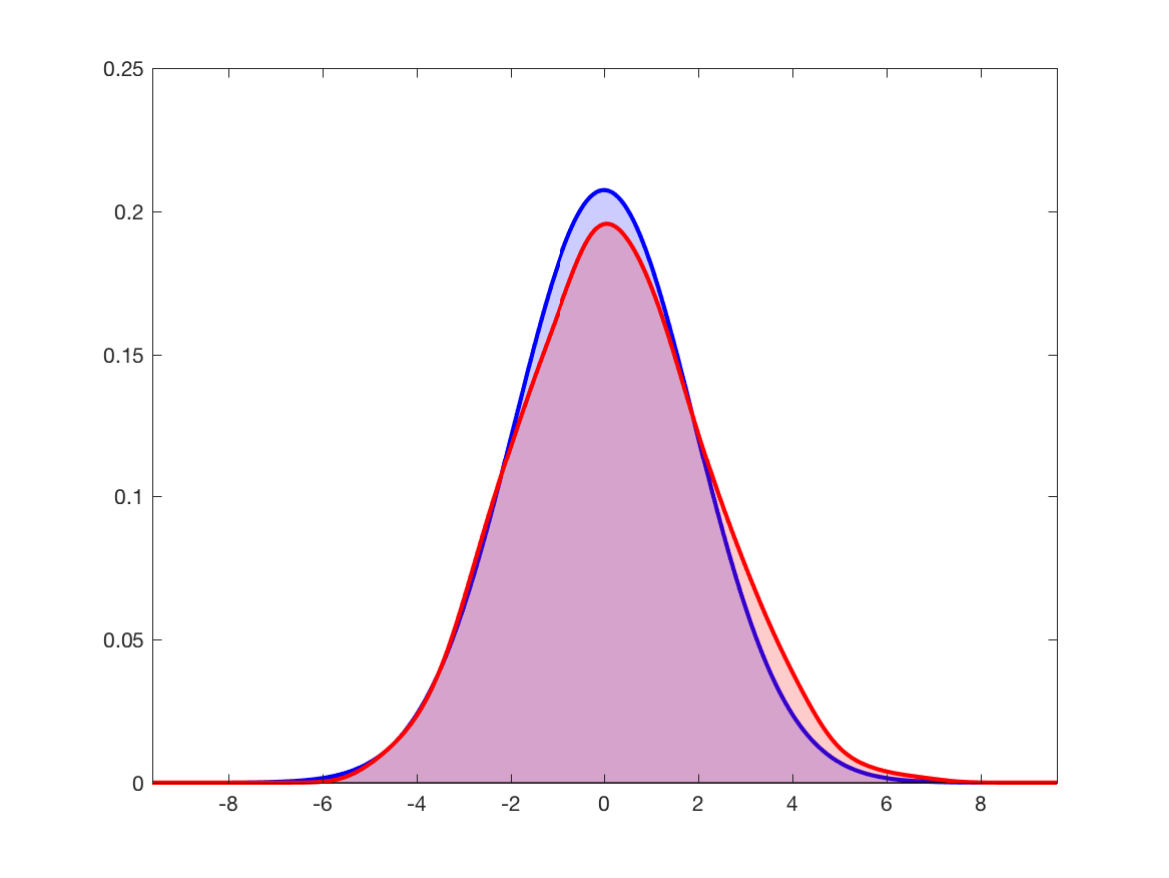} &
\includegraphics[width=0.27 \textwidth,height=0.32\textwidth]{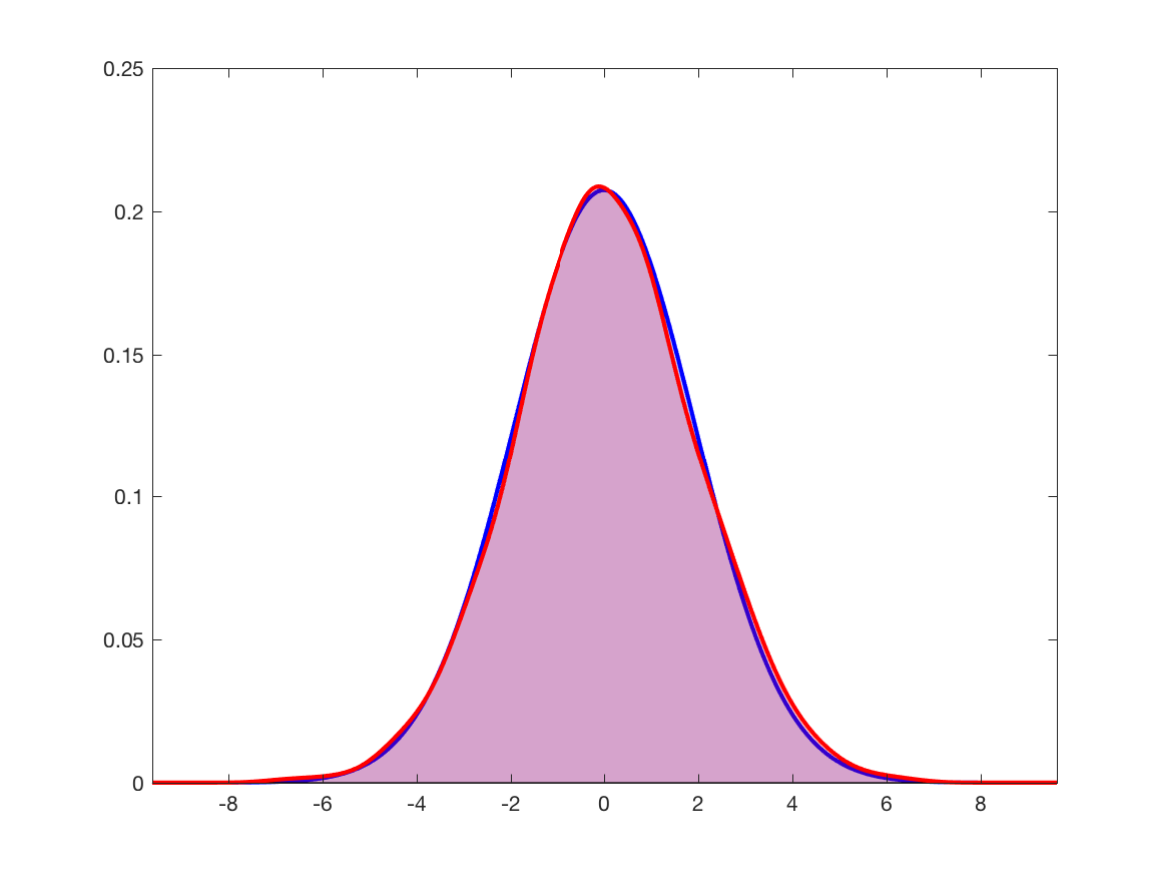}
\end{tabular}
\caption{\label{fig:H1one5} Case $a\neq b$ with one sample. Illustration of the convergence in distribution of  the empirical Sinkhorn loss for a $5\times 5$ grid, $\varepsilon=1,10,100$ and $n$ ranging from $10^2$ to $10^4$. Densities in red  (resp.   light  blue) represent the distribution of $\sqrt{n}(\WB_{p,\varepsilon}^p(\hat{a}_n,b)-\WB_{p,\varepsilon}^p(a,b))$ (resp.\ $\langle G,\ualpha_{\varepsilon}^{a,b}-1/2(\ualpha_{\varepsilon}^{a,a}+\vbeta_{\varepsilon}^{a,a})\rangle$).}
\end{figure}

\renewcommand{\sidecap}[1]{ {\begin{sideways}\parbox{0.2\textwidth}{\centering #1}\end{sideways}} }
\begin{figure}[ht!]
\centering
\begin{tabular}{cccc}
 & $n=10^2$ & $n=10^3$ & $n=10^4$ \\
\sidecap{$\varepsilon=10$} &
\includegraphics[width=0.25 \textwidth,height=0.29\textwidth]{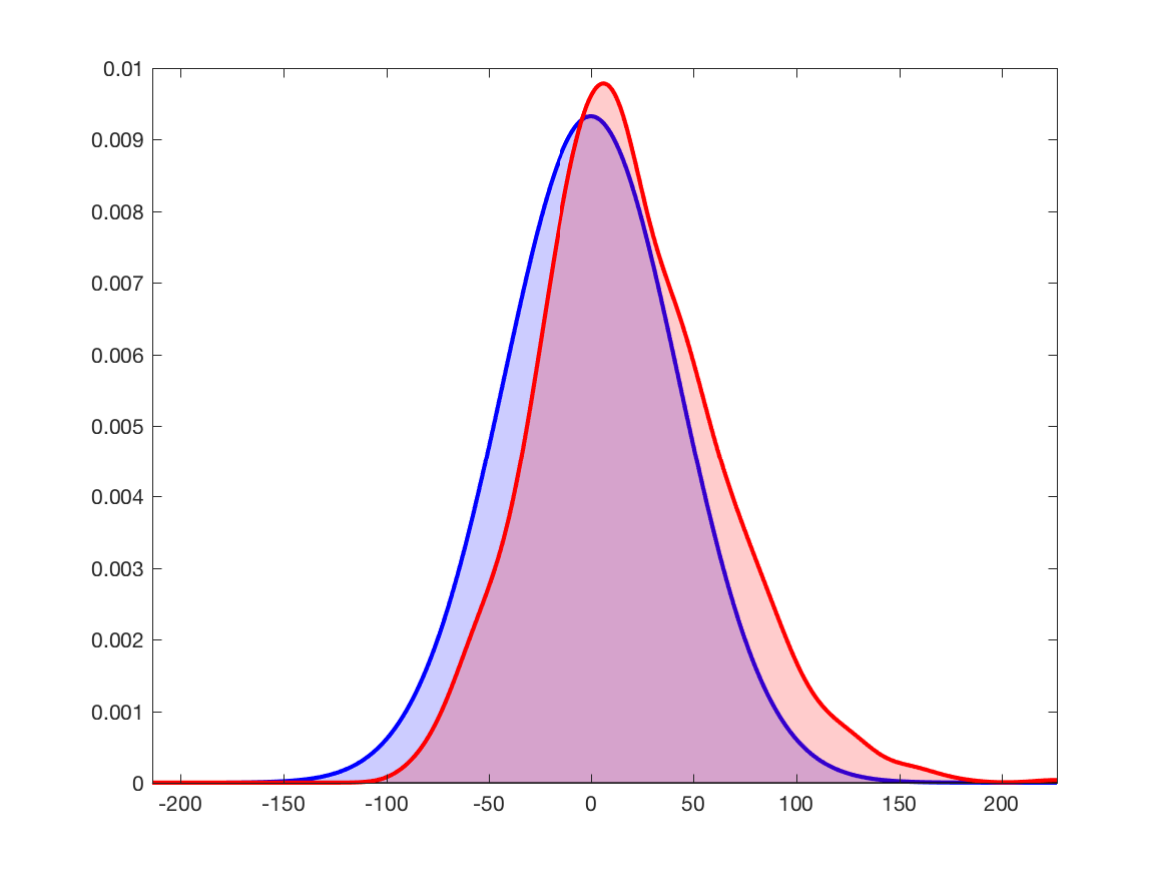} &
\includegraphics[width=0.25\textwidth,height=0.29\textwidth]{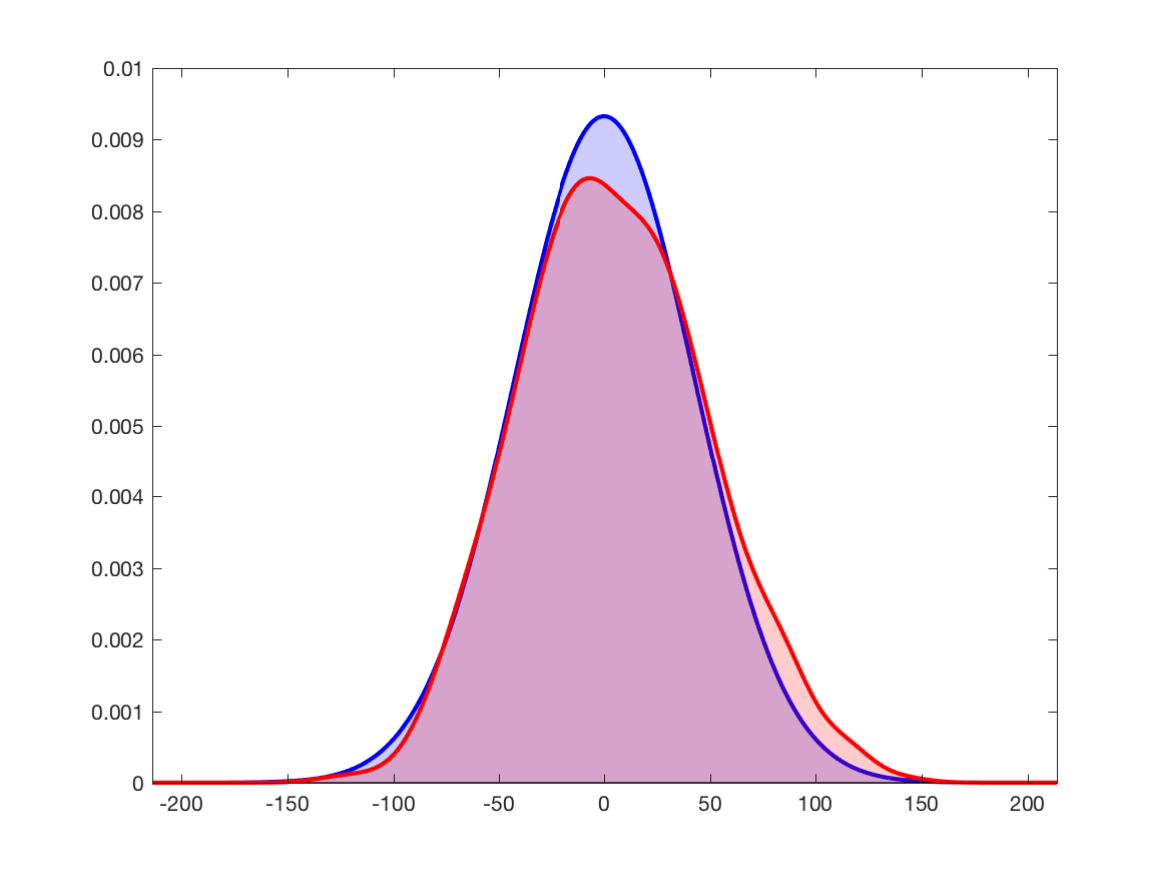} &
 \includegraphics[width=0.25 \textwidth,height=0.29\textwidth]{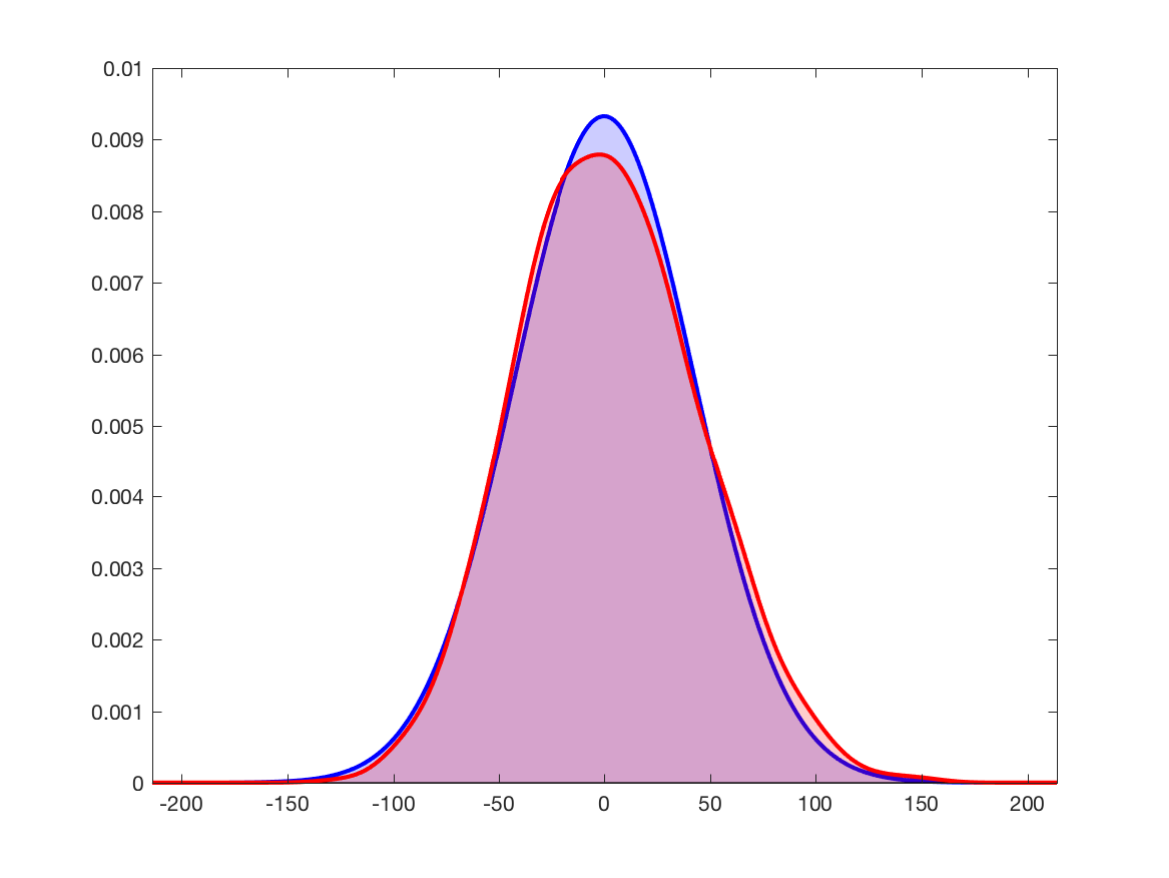} \\

\sidecap{$\varepsilon=100$} &
\includegraphics[width=0.25 \textwidth,height=0.29\textwidth]{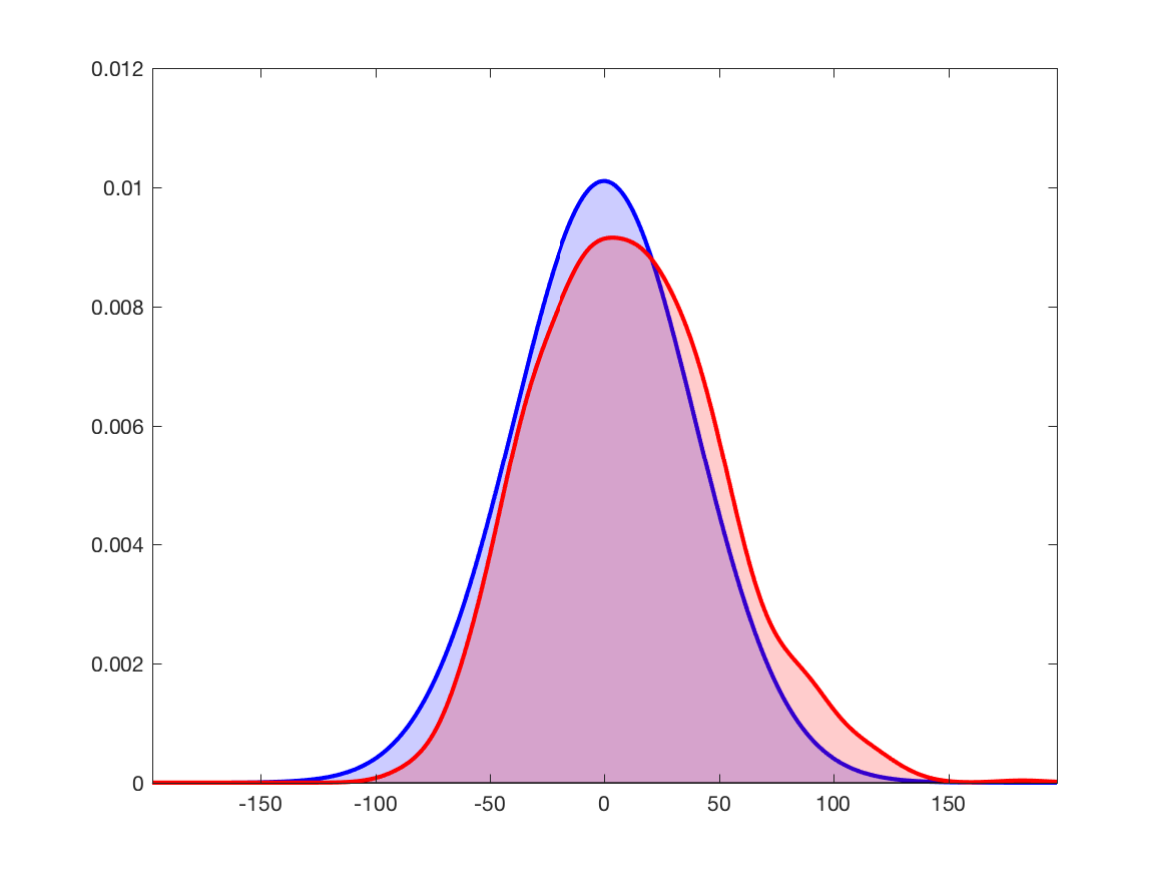} &
\includegraphics[width=0.25 \textwidth,height=0.29\textwidth]{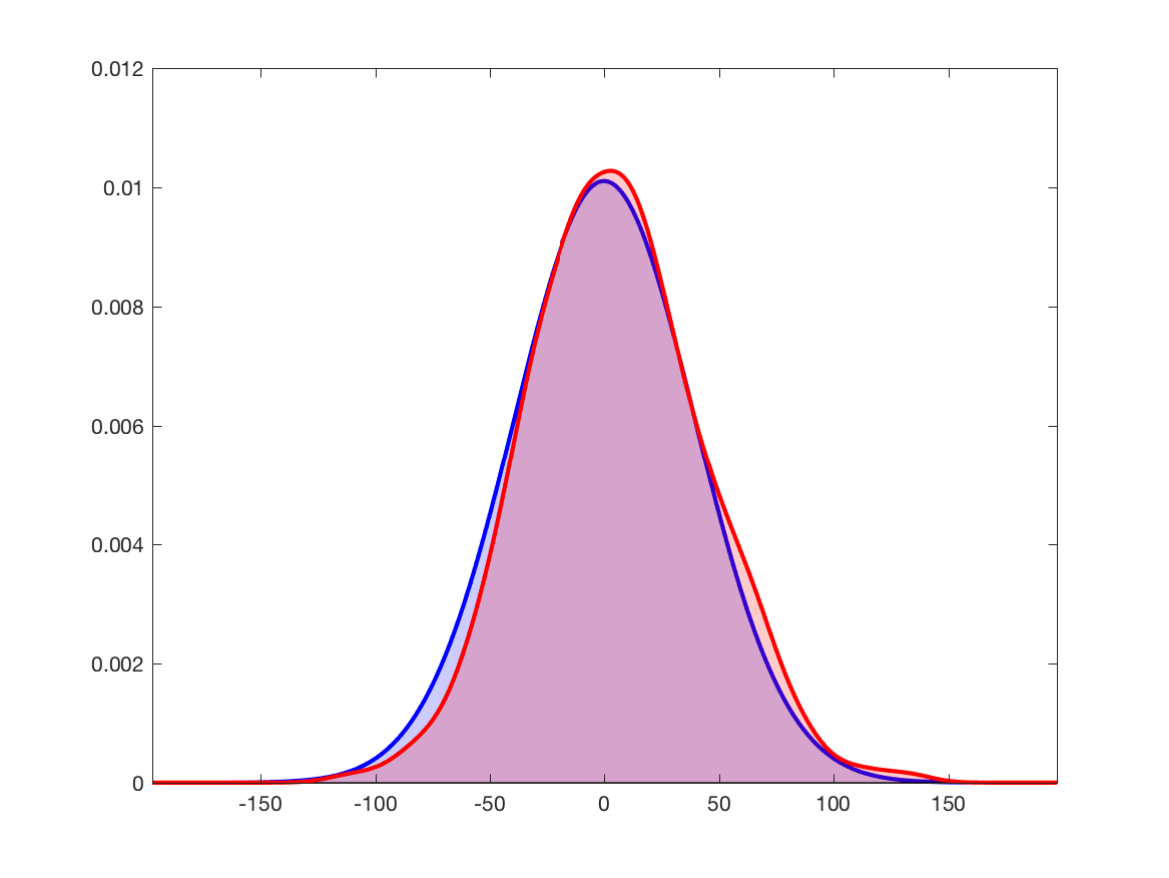} &
\includegraphics[width=0.25 \textwidth,height=0.29\textwidth]{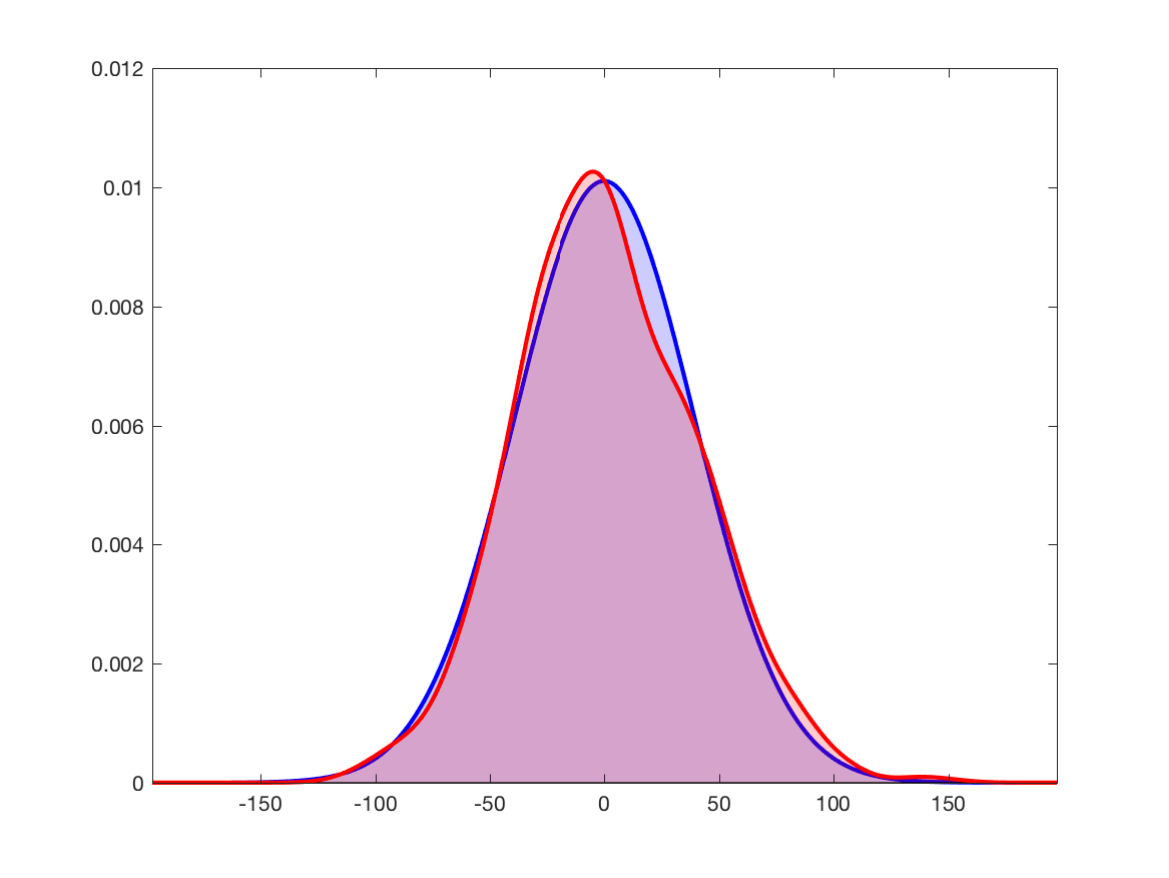} \\
\end{tabular}
\caption{\label{fig:H1one20} Case $a\neq b$ with one sample. Illustration of the convergence in distribution of  empirical Sinkhorn loss for a $20\times 20$ grid, $\varepsilon=10,100$ and $n$ ranging from $10^2$ to $10^4$. Densities in red  (resp.   light  blue) represent the distribution of $\sqrt{n}(\WB_{p,\varepsilon}^p(\hat{a}_n,b)-\WB_{p,\varepsilon}^p(a,b))$ (resp.\ $\langle G,\ualpha_{\varepsilon}^{a,b}-1/2(\ualpha_{\varepsilon}^{a,a}+\vbeta_{\varepsilon}^{a,a})\rangle$).}
\end{figure}

Let us now shed some light on the bootstrap procedure. The results on bootstrap experiments are reported in Figure \ref{fig:example_Boot_H1_one}. From the uniform distribution $a$,  we generate only one random distribution $\hat{a}_n$. The value of the realization $\sqrt{n}(\WB_{p,\varepsilon}^p(\hat{a}_n,b)-\WB_{p,\varepsilon}^p(a,b))$ is represented by the red vertical lines in Figure \ref{fig:example_Boot_H1_one}.  Besides, we generate from $\hat{a}_n$, a sequence of $M = 10^3$ bootstrap samples of random measures denoted by $\hat{a}_n^{\ast}$ (such that $n \hat{a}_n^{\ast}$ follows a multinomial distribution with parameter $\hat{a}_n$). We use again a kernel density estimate (with a data-driven bandwidth) to compare the distribution of $\sqrt{n}(\WB_{p,\varepsilon}^p(\hat{a}_n^{\ast},b)-\WB_{p,\varepsilon}^p(\hat{a}_n,b))$ to the distribution of  $ \sqrt{n}(\WB_{p,\varepsilon}^p(\hat{a}_n,b)-\WB_{p,\varepsilon}^p(a,b))$ displayed in Figure \ref{fig:H1one5} and Figure \ref{fig:H1one20}. The green vertical lines in Figure \ref{fig:example_Boot_H1_one} represent a confidence interval of level $95\%$.  The observation represented by the red vertical line is mostly located within this confidence interval, and the density estimated by bootstrap decently captures the shape of the non-asymptotic distribution of Sinkhorn losses.

\renewcommand{\sidecap}[1]{ {\begin{sideways}\parbox{0.3\textwidth}{\centering #1}\end{sideways}} }
\begin{figure}[ht!]
\centering
\begin{tabular}{cccc}
& $n=10^2$ & $n=10^3$ & $n=10^4$ \\
\sidecap{Grid $5\times 5$} & \includegraphics[width=0.25 \textwidth,height=0.29\textwidth]{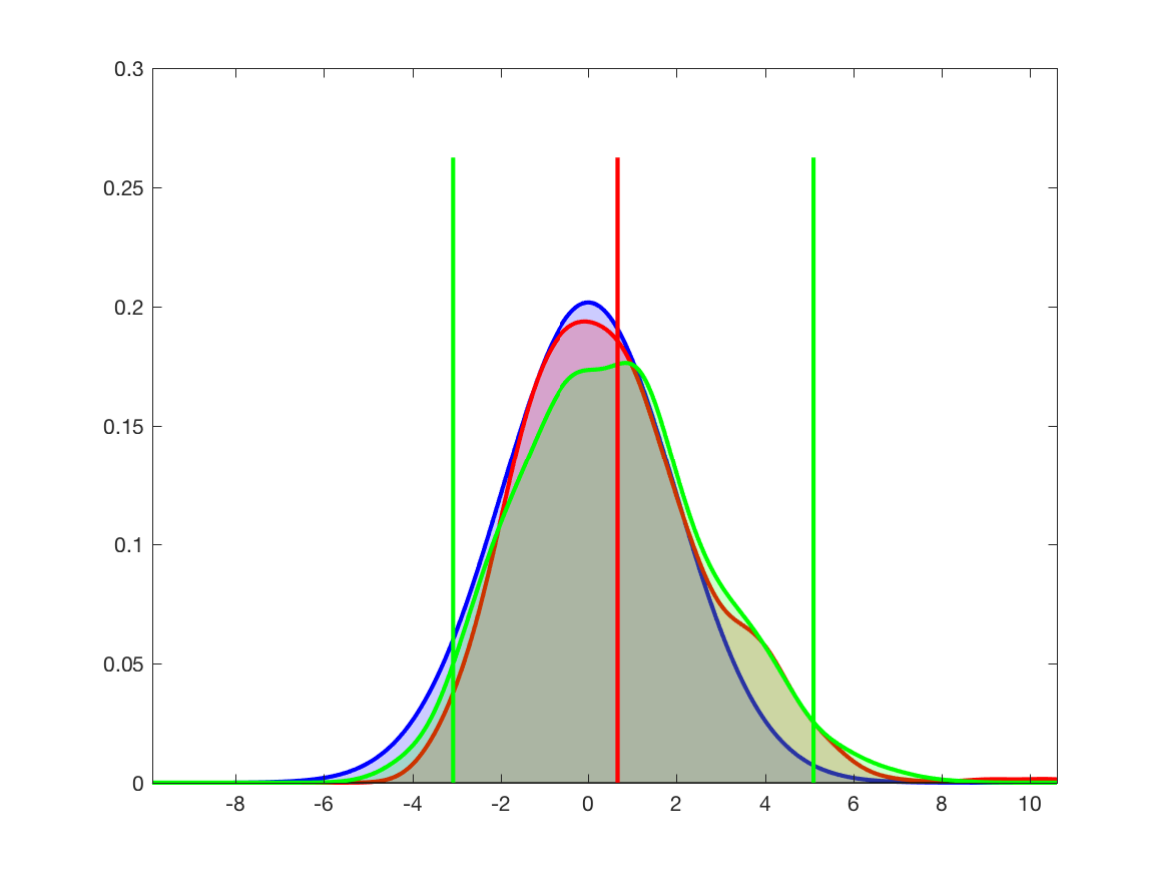} &
\includegraphics[width=0.25 \textwidth,height=0.29\textwidth]{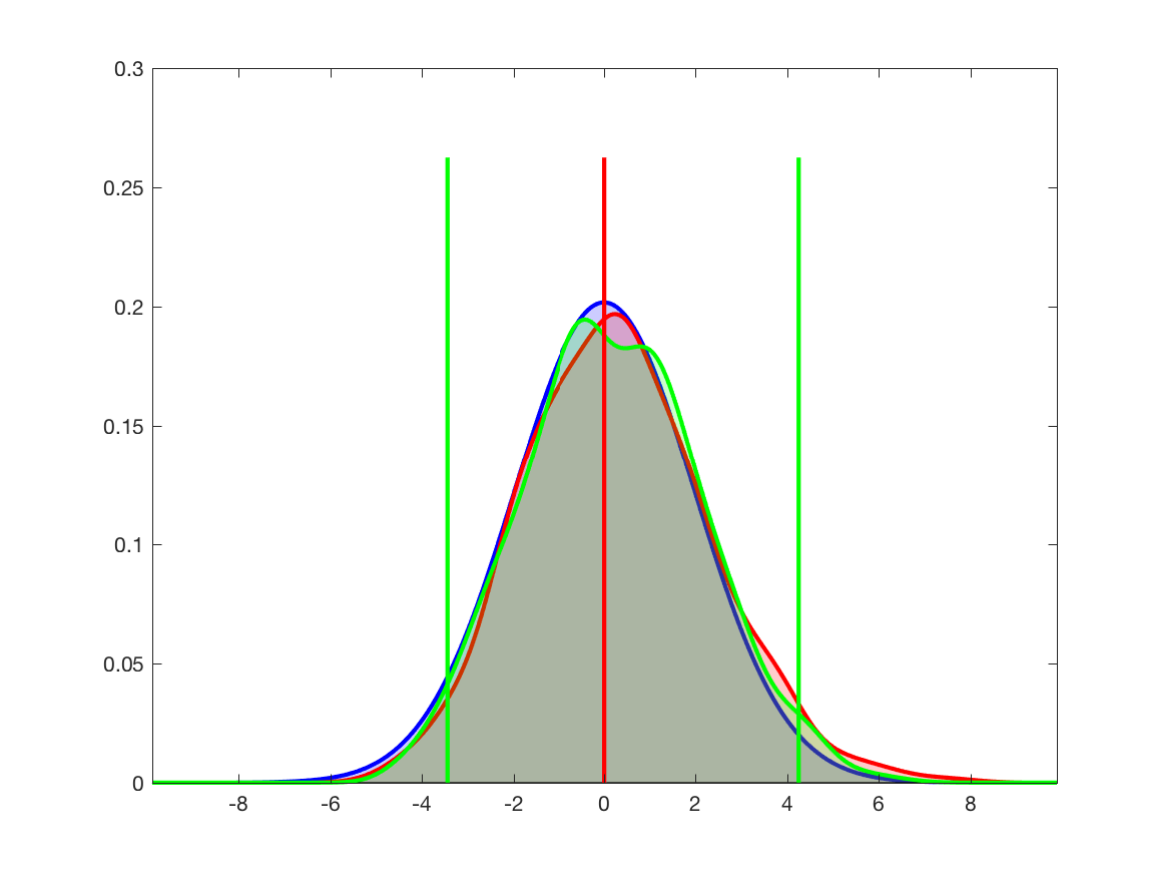} &
\includegraphics[width=0.25 \textwidth,height=0.29\textwidth]{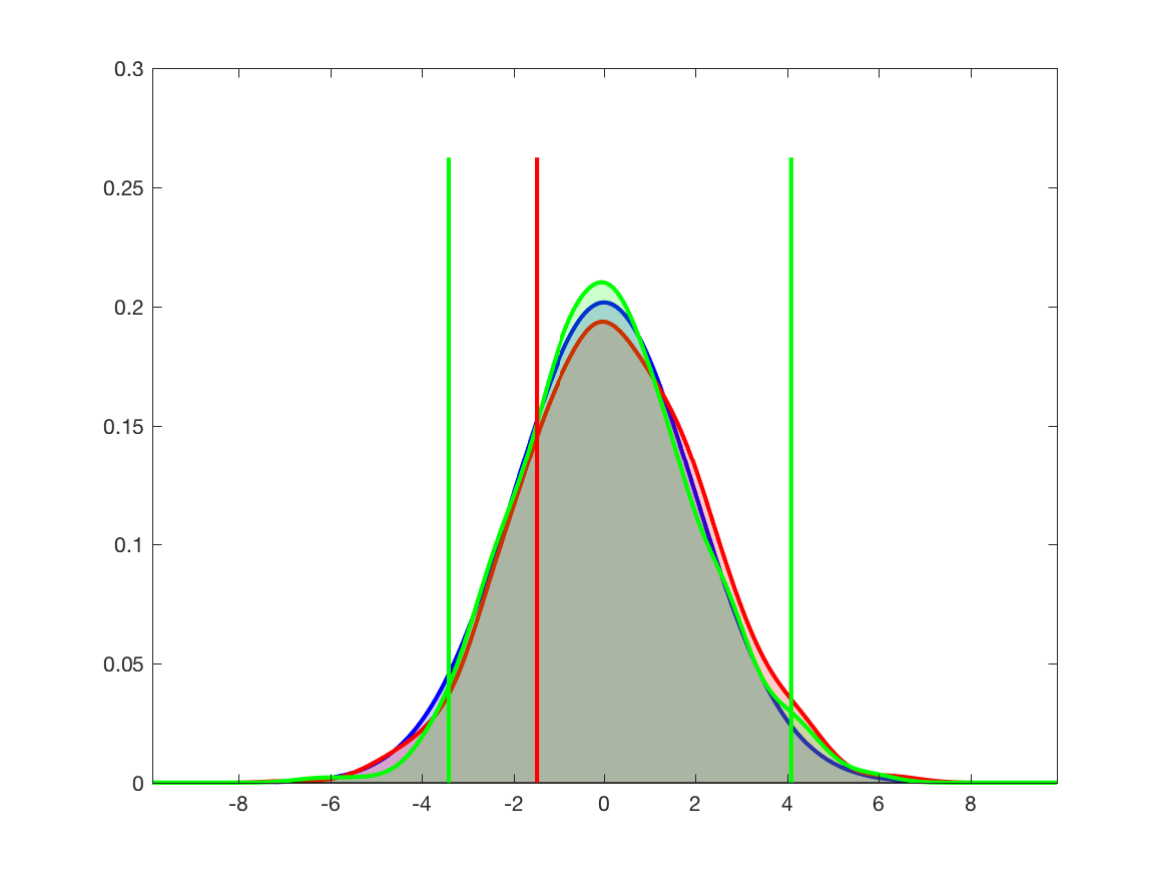} \\
& $n=10^2$ & $n=10^3$ & $n=10^4$ \\
\sidecap{Grid $20\times 20$} & \includegraphics[width=0.25 \textwidth,height=0.29\textwidth]{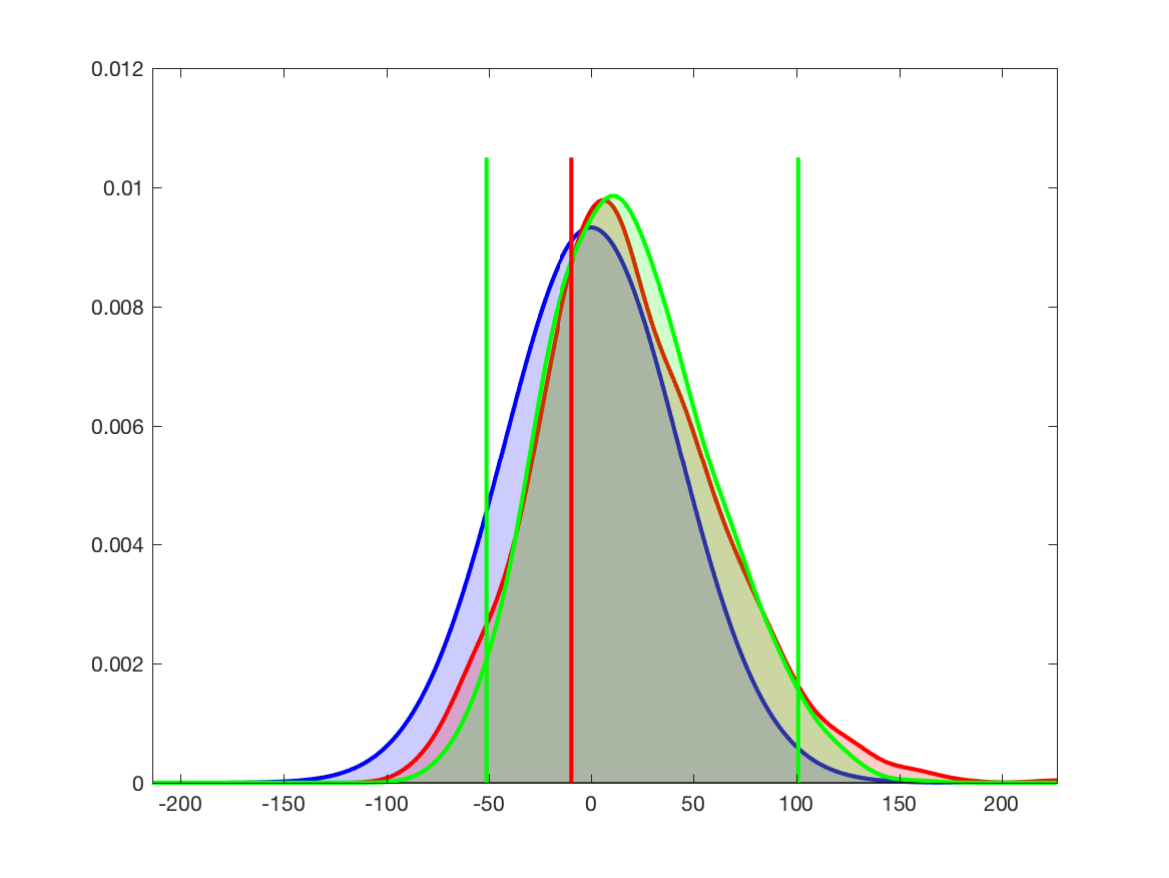} &
\includegraphics[width=0.25 \textwidth,height=0.29\textwidth]{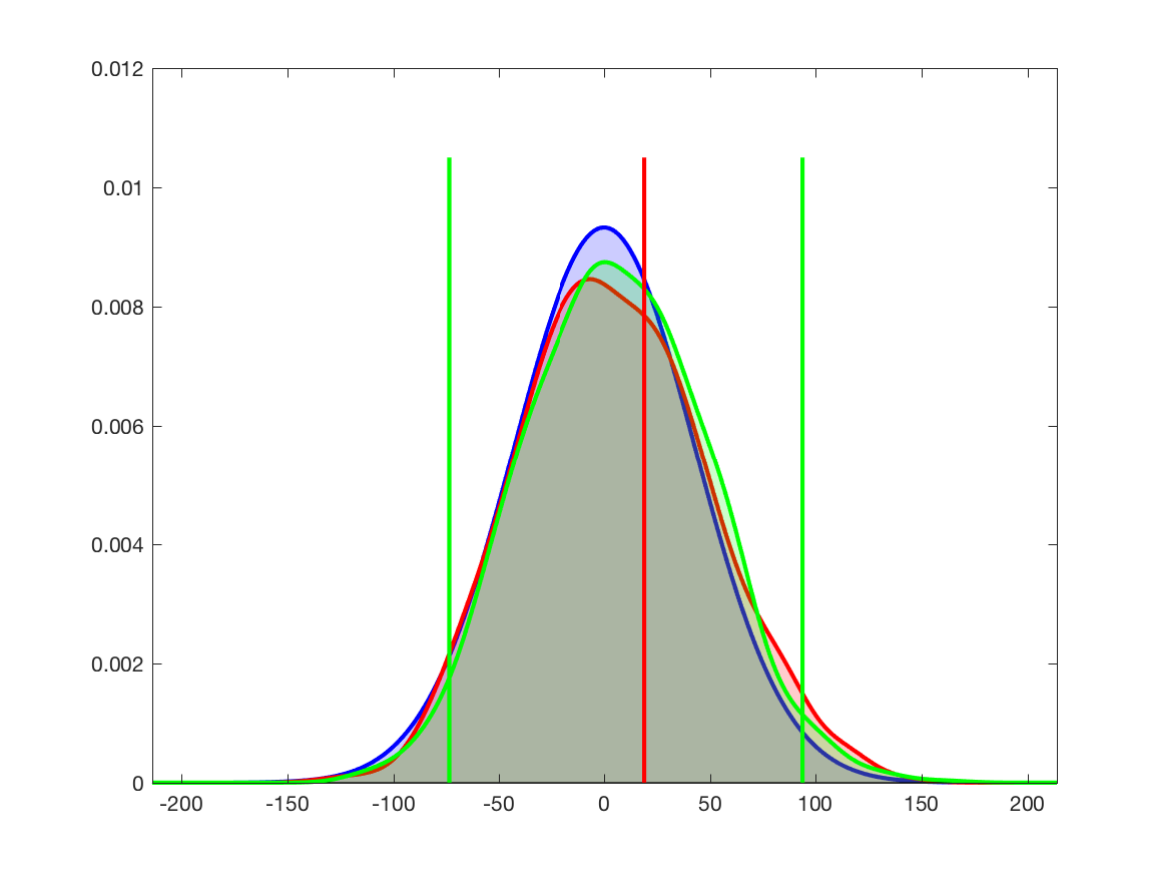} &
\includegraphics[width=0.25 \textwidth,height=0.29\textwidth]{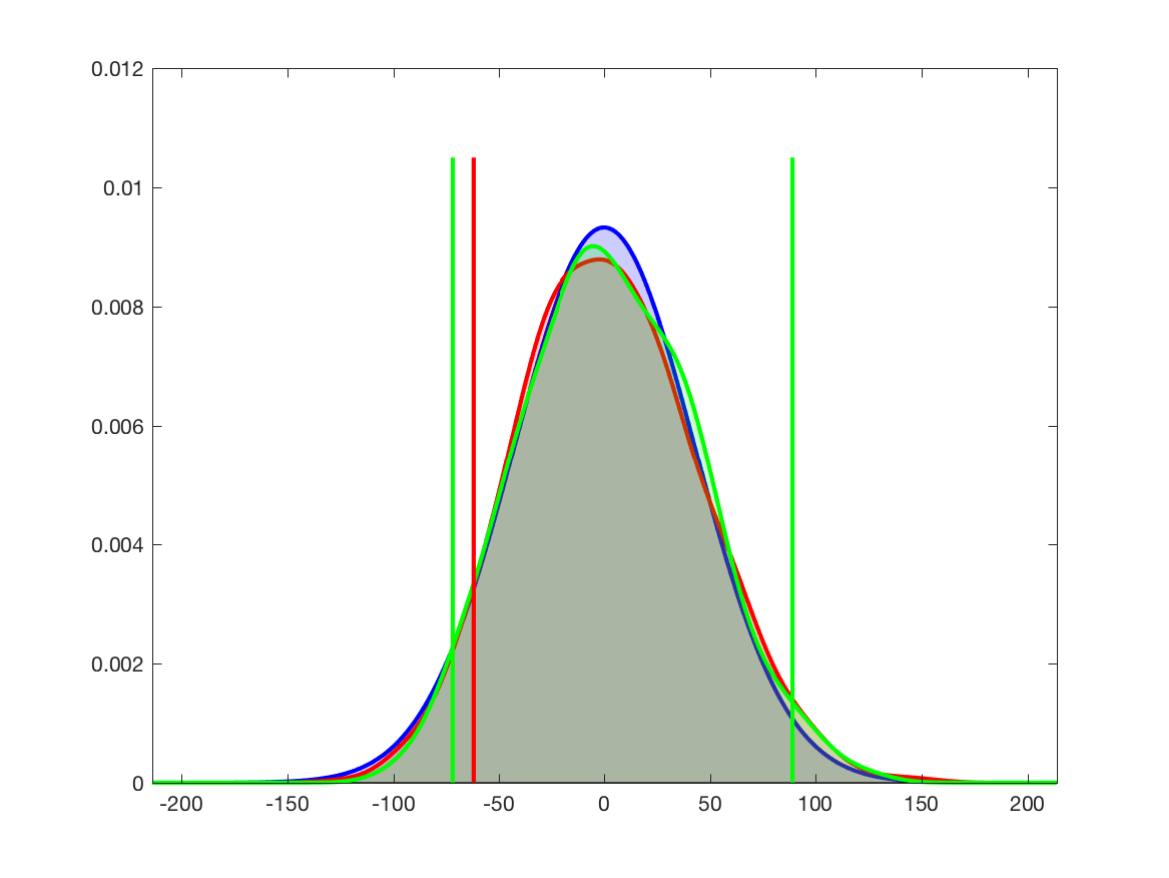}
\end{tabular}
\caption{\label{fig:example_Boot_H1_one}Case $a\neq b$ with one sample. Illustration of the bootstrap with $\varepsilon=10$, grids of size $5\times 5$ and $20\times 20$ to approximate the non-asymptotic distribution of  empirical Sinkhorn losses. Densities in red  (resp. light blue) represent the distribution of $\sqrt{n}(\WB_{p,\varepsilon}^p(\hat{a}_n,b)-\WB_{p,\varepsilon}^p(a,b))$ (resp.\ $\langle G,\ualpha_{\varepsilon}^{a,b}-1/2(\ualpha_{\varepsilon}^{a,a}+\vbeta_{\varepsilon}^{a,a})\rangle$). The green density represents the distribution of the random variable $\sqrt{n}(\WB_{p,\varepsilon}^p(\hat{a}^{\ast}_n,b)-\WB_{p,\varepsilon}^p(\hat{a}_n,b))$ in Theorem \ref{bootstrap}.}
\end{figure}

\subsubsection{Alternative $a\neq b$ - Two samples}
We consider the same  setting as before, excepting that data are now both sampled from distributions $a$ and $b$. Hence, we run $M=10^3$ experiments to obtain a kernel density estimation of the distribution of
$$
\rho_{n,m}(\WB_{p,\varepsilon}^p(\hat{a}_n,\hat{b}_m)-\WB_{p,\varepsilon}^p(a,b)),
$$
that is compared to the density of the Gaussian variable
$$\sqrt{\gamma} \langle G,\ualpha_{\varepsilon}^{a,b}-\frac{1}{2}(\ualpha_{\varepsilon}^{a,a}+ \vbeta_{\varepsilon}^{a,a})\rangle
+\sqrt{1-\gamma} \langle H,\vbeta_{\varepsilon}^{a,b}-\frac{1}{2}(\ualpha_{\varepsilon}^{b,b}+ \vbeta_{\varepsilon}^{b,b})\rangle,$$
for different values of $n$ and $m$.
The results are reported in Figure \ref{fig:H1two}. The convergence does not seem as good as in the one sample case, this must be due to the randomness coming from both $\hat{a}_n$ and $\hat{b}_m$.
\renewcommand{\sidecap}[1]{ {\begin{sideways}\parbox{0.2\textwidth}{\centering #1}\end{sideways}} }
\begin{figure}[ht!]
\centering
\begin{tabular}{cccc}
& \hspace{0.2cm}   $n=m=10^3$ & $n=m=10^4$ & $n=m=10^5$\\
\\
\sidecap{$\varepsilon=10$}& \hspace{0.2cm} \includegraphics[width=0.25 \textwidth,height=0.3\textwidth]{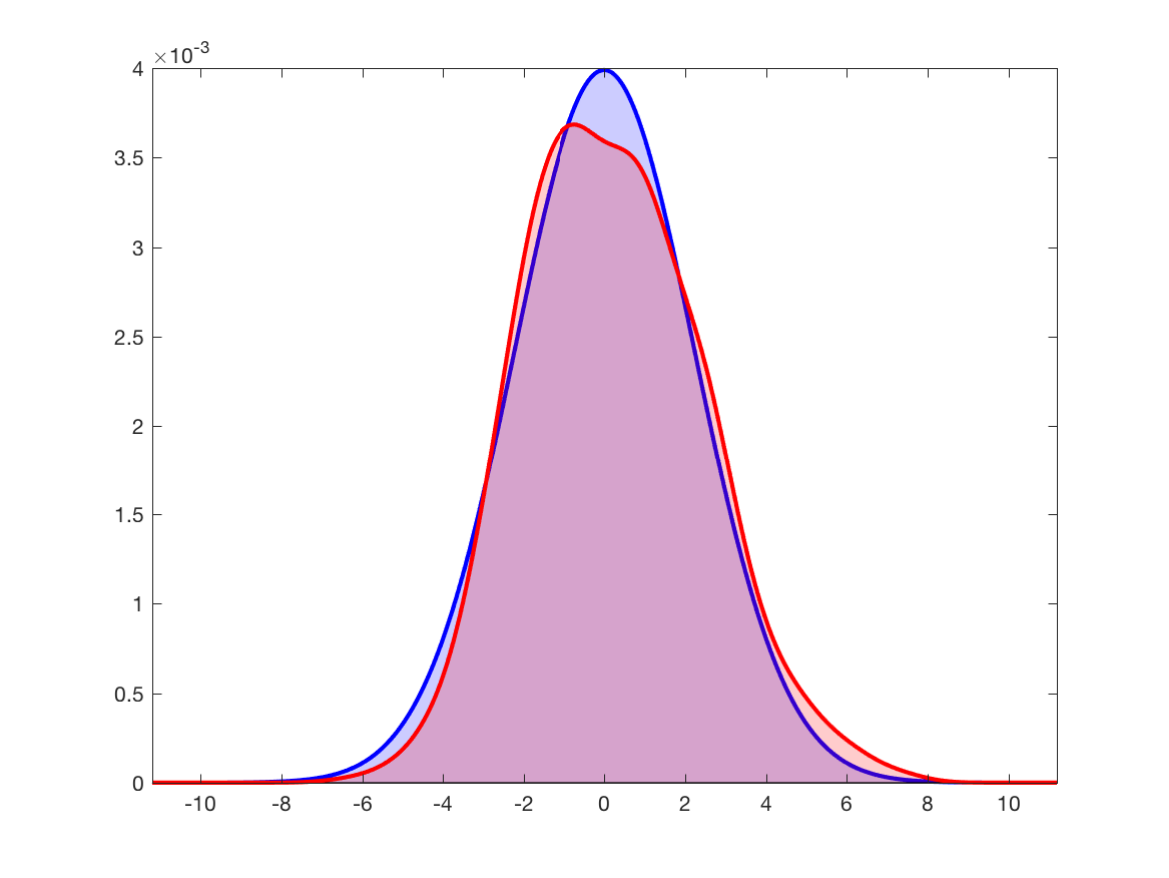} &
\includegraphics[width=0.25 \textwidth,height=0.3\textwidth]{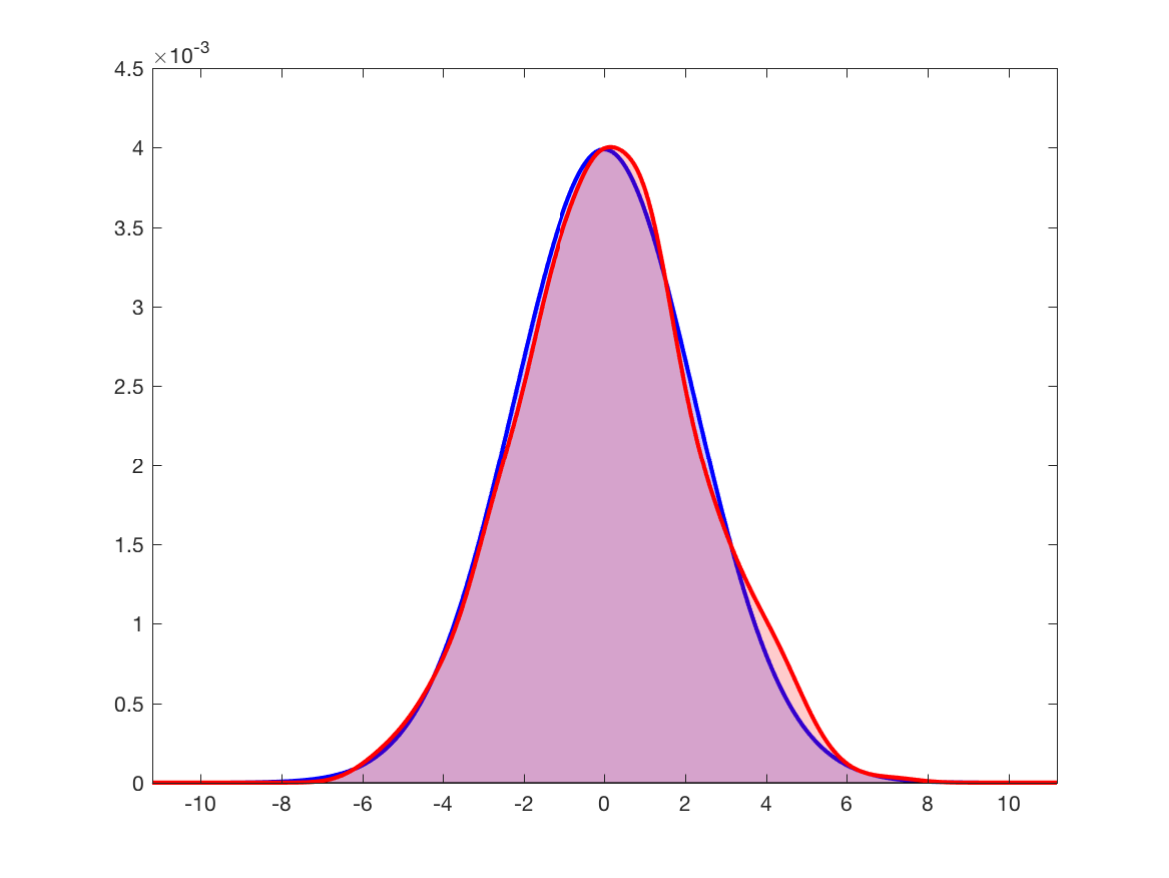} &
\includegraphics[width=0.25 \textwidth,height=0.3\textwidth]{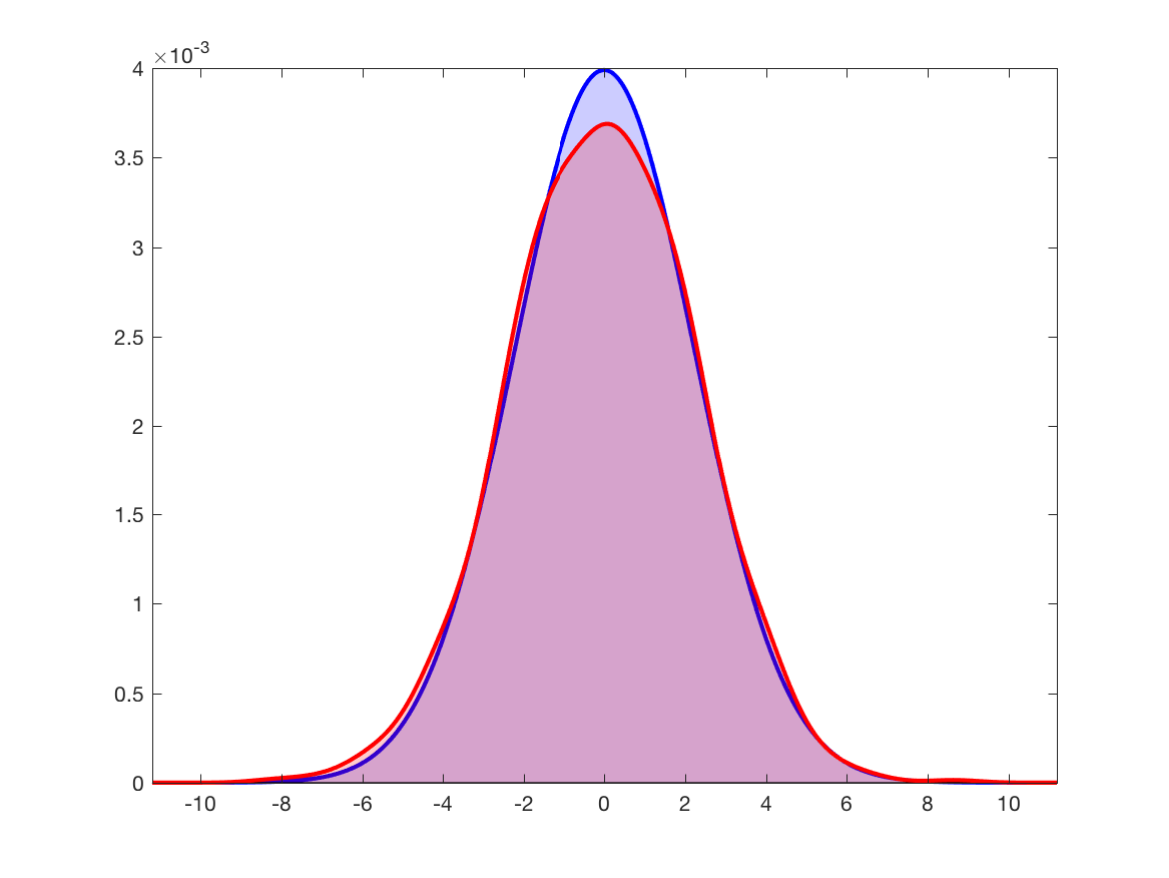} \\
\sidecap{$\varepsilon=100$}& \hspace{0.2cm} \includegraphics[width=0.25 \textwidth,height=0.3\textwidth]{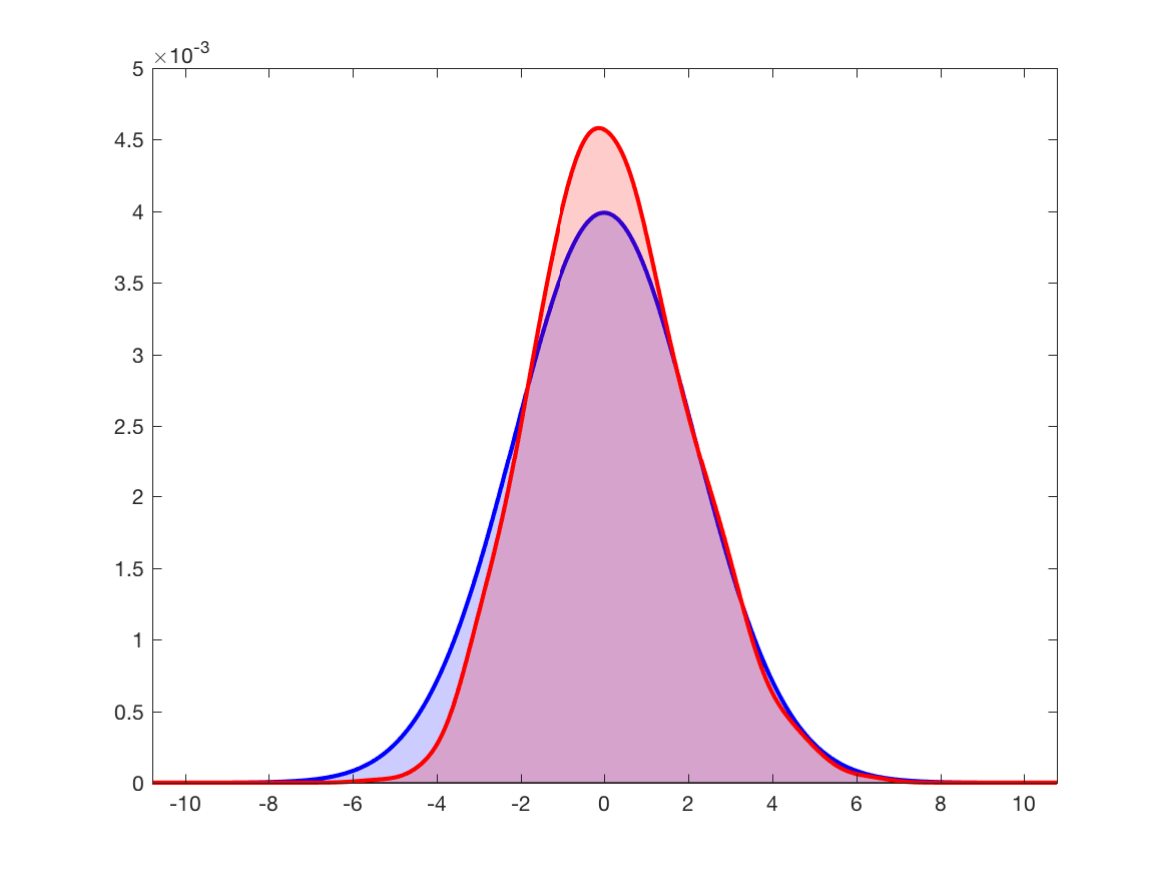} &
\includegraphics[width=0.25 \textwidth,height=0.3\textwidth]{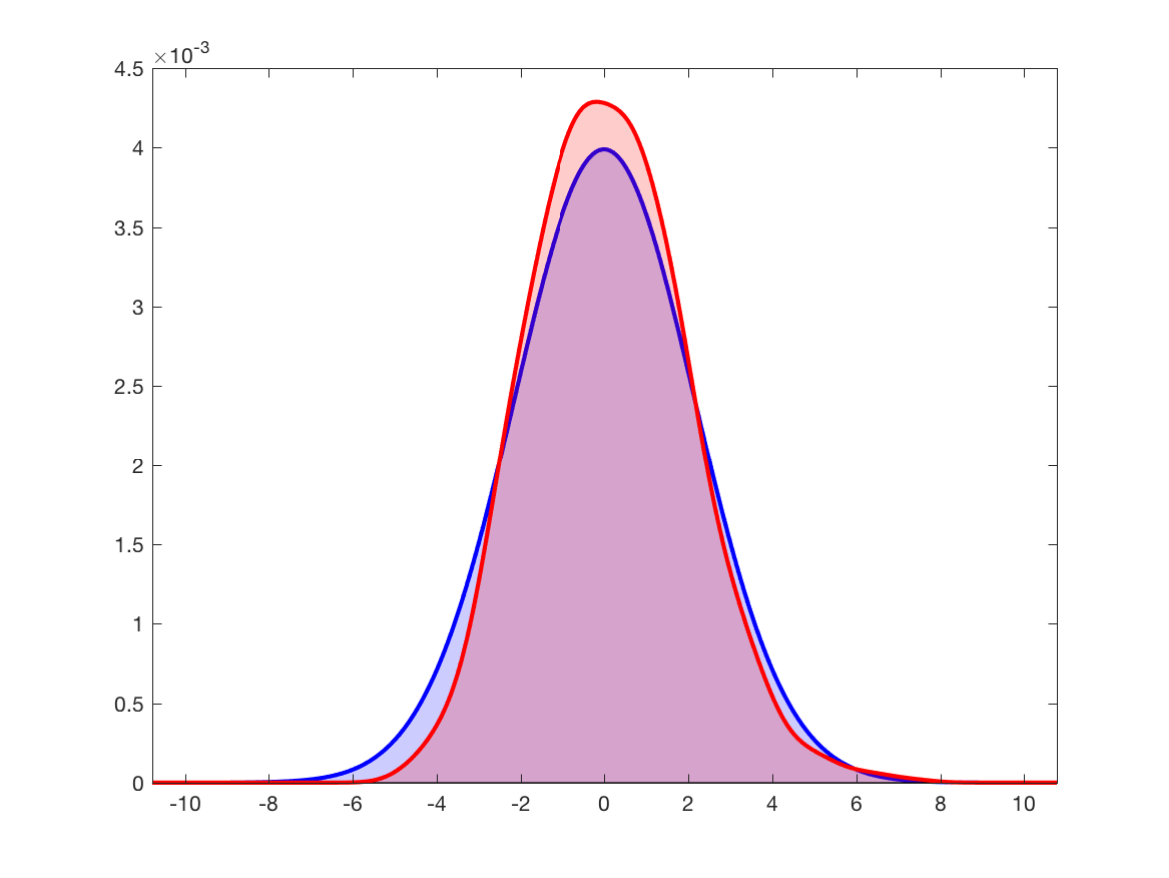} &
\includegraphics[width=0.25 \textwidth,height=0.3\textwidth]{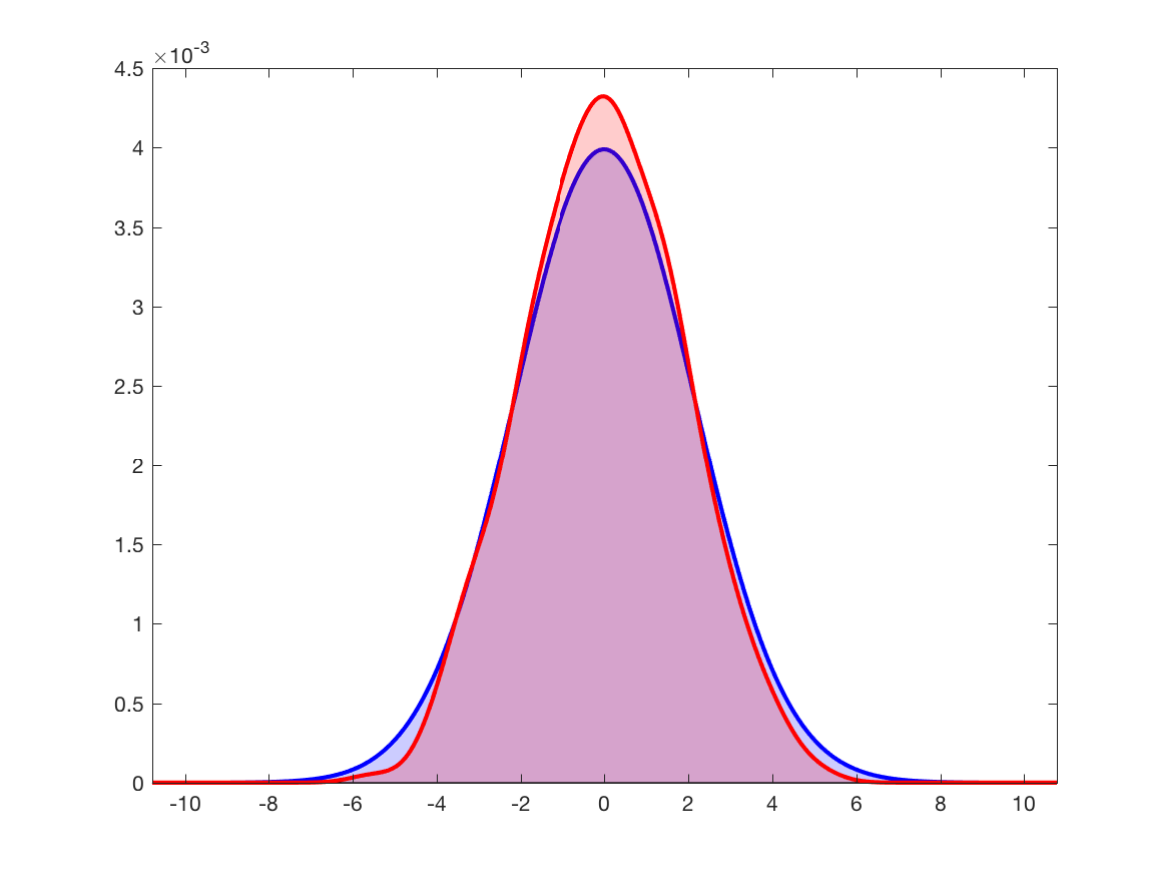} \end{tabular}
\caption{\label{fig:H1two} Case $a\neq b$ with two samples. Illustration of the convergence in distribution of  empirical Sinkhorn loss  for a $5\times 5$ grid, for $\varepsilon=10,100$, $n=m$ and $n$ ranging from $10^3$ to $10^5$. Densities in red  (resp.\ blue) represent the distribution of $\rho_{n,m}(\WB_{p,\varepsilon}^p(\hat{a}_n,\hat{b}_m)-\WB_{p,\varepsilon}^p(a,b))$ (resp.\ $\sqrt{\gamma}  \langle G,\ualpha_{\varepsilon}^{a,b}-\frac{1}{2}(\ualpha_{\varepsilon}^{a,a}+ \vbeta_{\varepsilon}^{a,a})\rangle
+ \sqrt{1-\gamma} \langle H,\vbeta_{\varepsilon}^{a,b}-\frac{1}{2}(\ualpha_{\varepsilon}^{b,b}+ \vbeta_{\varepsilon}^{b,b})\rangle$ with $\gamma = 1/2$).}
\end{figure}

We also report in Figure \ref{fig:example_Boot_H1_two} results on the consistency of the bootstrap procedure   under the hypothesis $H_1$ with two samples. From the distributions $a$ and $b$,  we generate two random distributions $\hat{a}_n$ and $\hat{b}_m$. The value of the realization  $\rho_{n,m}(\WB_{p,\varepsilon}^p(\hat{a}_n,\hat{b}_m)-\WB_{p,\varepsilon}^p(a,b))$ is represented by the red vertical lines in Figure \ref{fig:example_Boot_H1_two}. Then, we generate from $\hat{a}_n$ and $\hat{b}_m$, two sequences of $M = 10^3$ bootstrap samples of random measures denoted by $\hat{a}_n^{\ast}$ and $\hat{b}_m^{\ast}$. We use again a kernel density estimate (with a data-driven bandwith) to compare the green distribution of $\rho_{n,m}(\WB_{p,\varepsilon}^p(\hat{a}^{\ast}_n,\hat{b}_m^{\ast})-\WB_{p,\varepsilon}^p(\hat{a}_n,\hat{b}_m))$ to the red distribution of  $\rho_{n,m}(\WB_{p,\varepsilon}^p(\hat{a}_n,\hat{b}_m)-\WB_{p,\varepsilon}^p(a,b))$ displayed in Figure \ref{fig:example_Boot_H1_two}. The green vertical lines in Figure \ref{fig:example_Boot_H1_two} represent a  confidence interval of level $95\%$. We can draw the same conclusion as in the one sample case. All these experiments thus perfectly illustrate the Theorem \ref{Th_CLT_loss_H1}. 

\renewcommand{\sidecap}[1]{ {\begin{sideways}\parbox{0.3\textwidth}{\centering #1}\end{sideways}} }
\begin{figure}[ht!]
\centering
\begin{tabular}{cccc}
& $n=m=10^3$ & $n=m=10^4$ & $n=m=10^5$ \\
\sidecap{$\varepsilon=10$}& \includegraphics[width=0.25 \textwidth,height=0.32\textwidth]{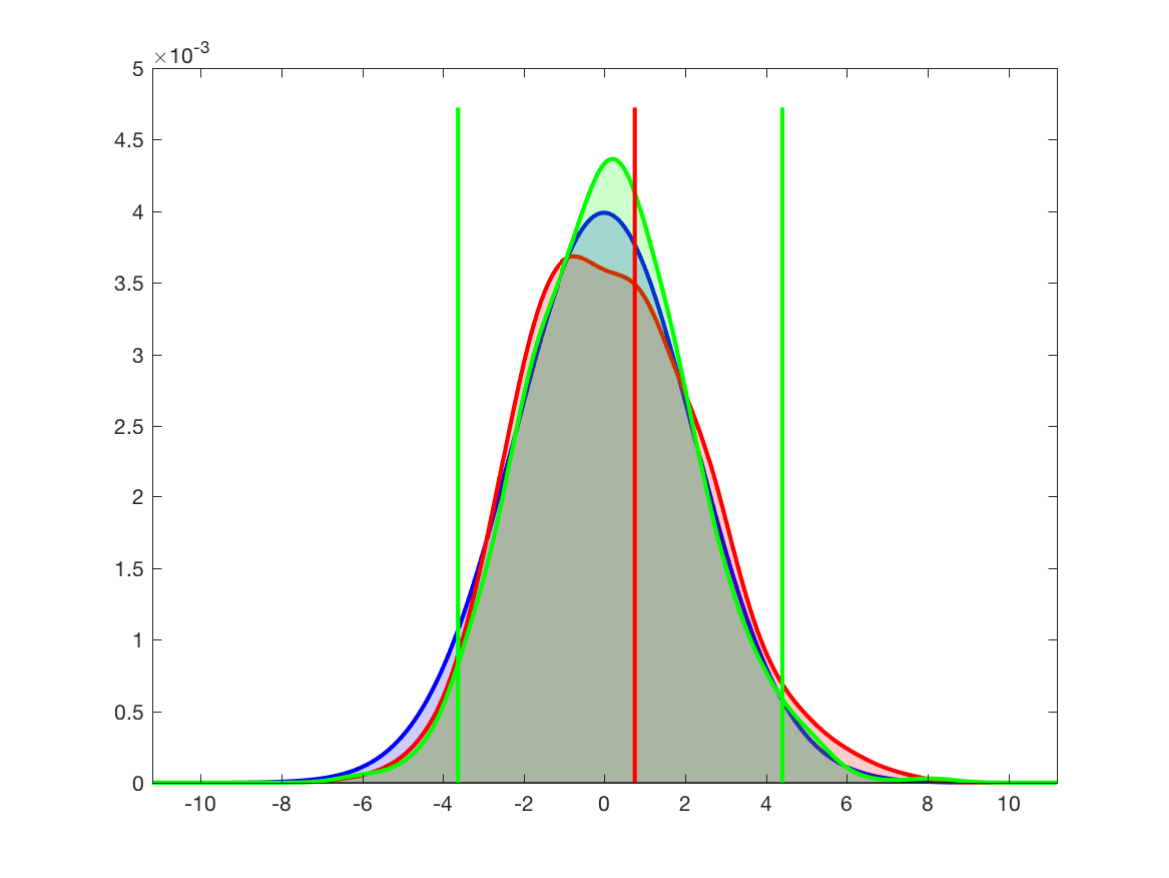} &
\includegraphics[width=0.25 \textwidth,height=0.32\textwidth]{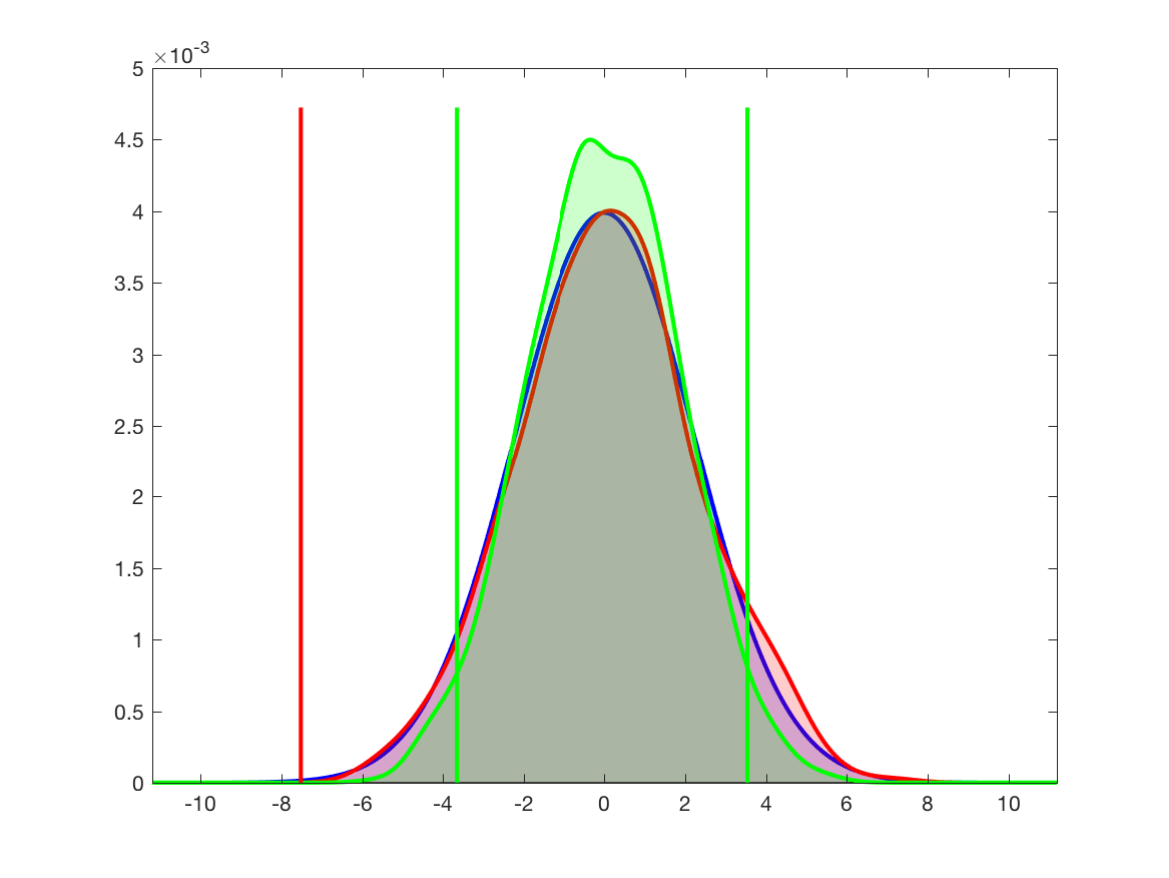} &
\includegraphics[width=0.25 \textwidth,height=0.32\textwidth]{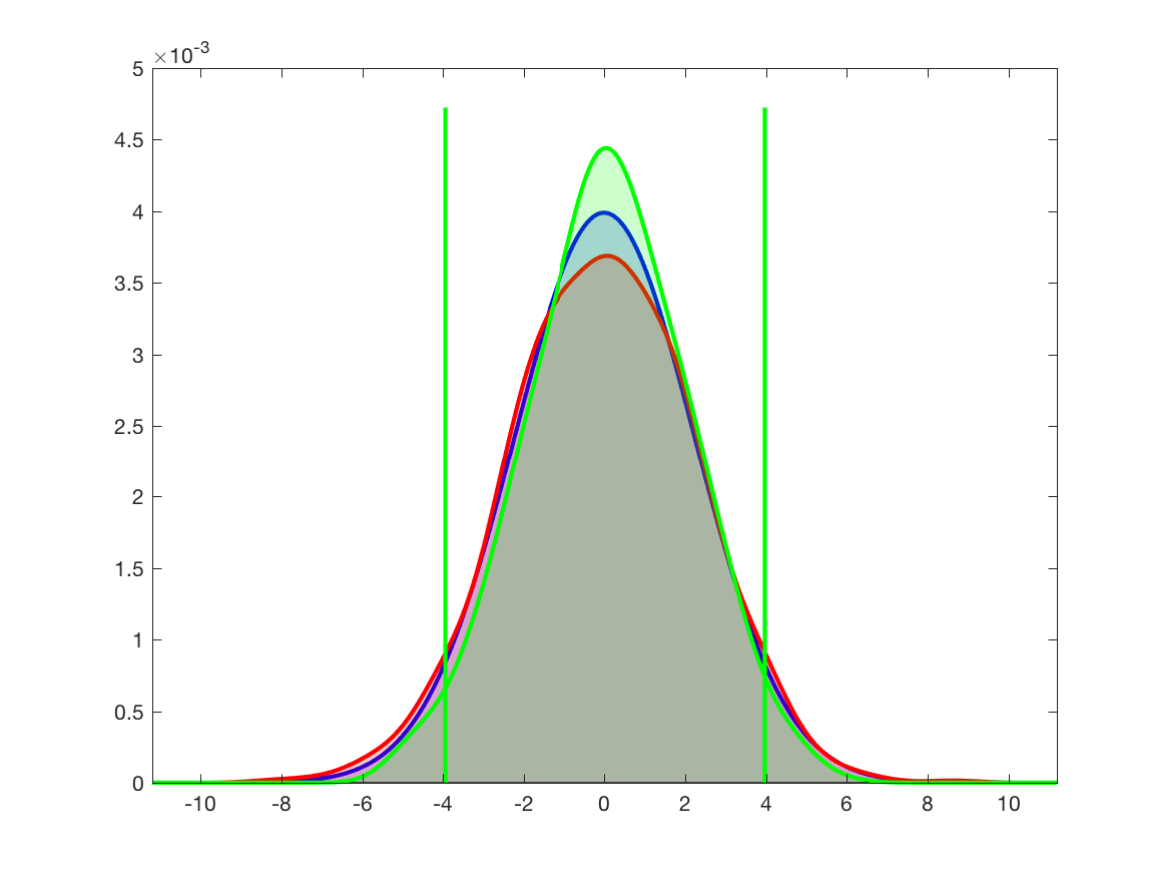} \\
\sidecap{$\varepsilon=100$}& \includegraphics[width=0.25 \textwidth,height=0.32\textwidth]{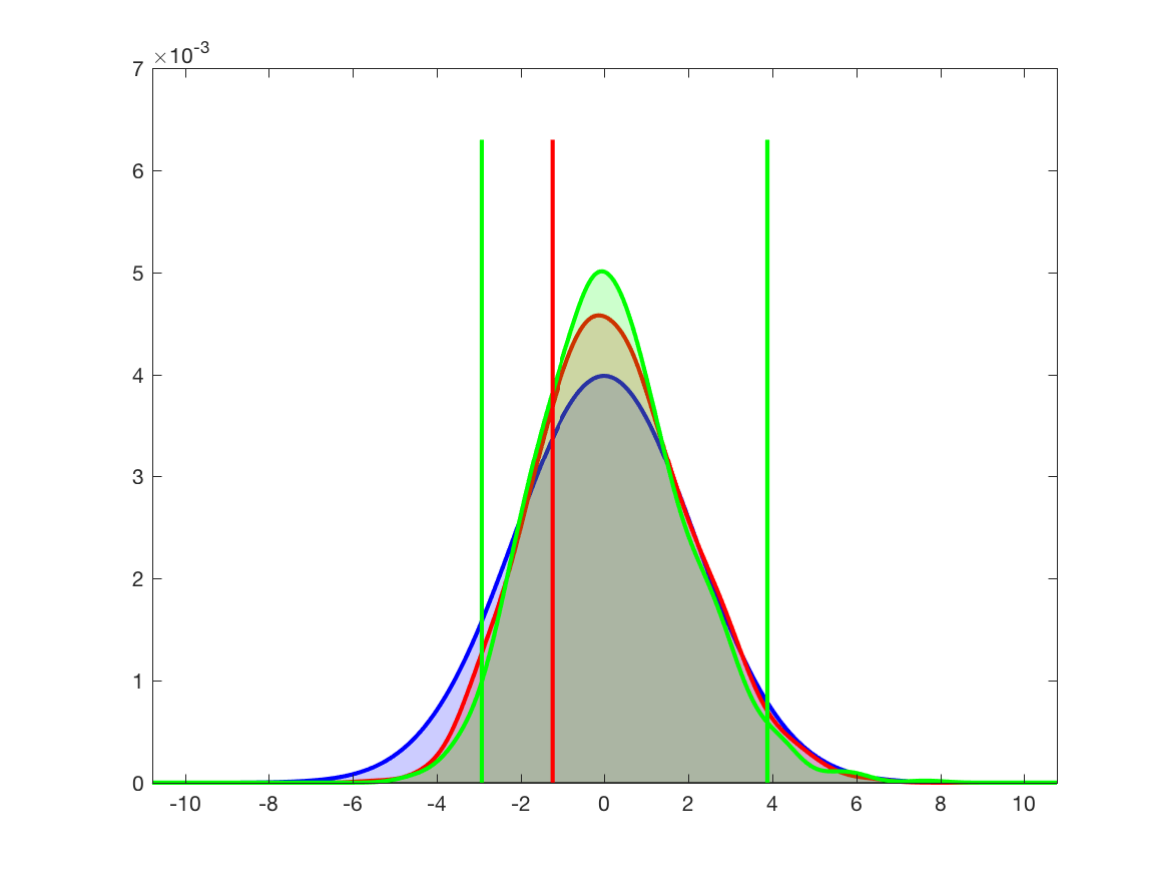} &
\includegraphics[width=0.25 \textwidth,height=0.32\textwidth]{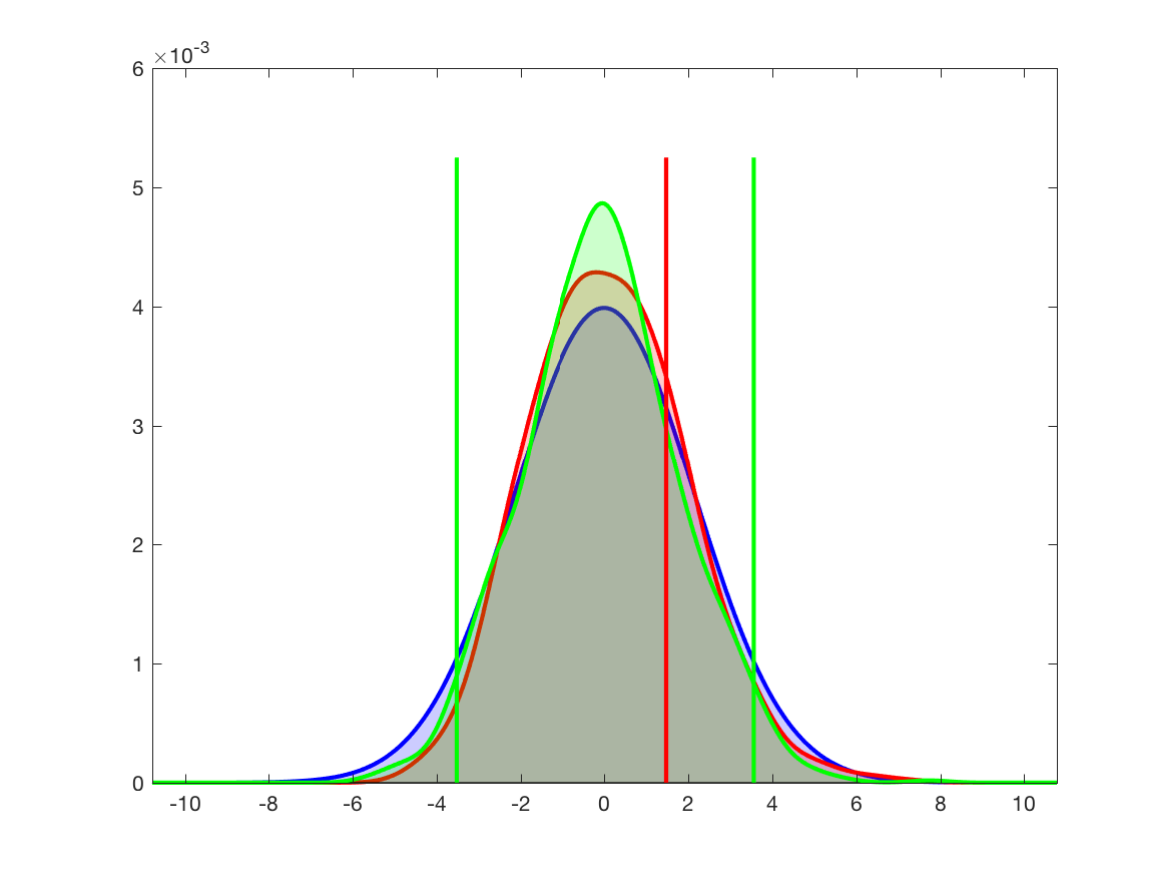} &
\includegraphics[width=0.25 \textwidth,height=0.32\textwidth]{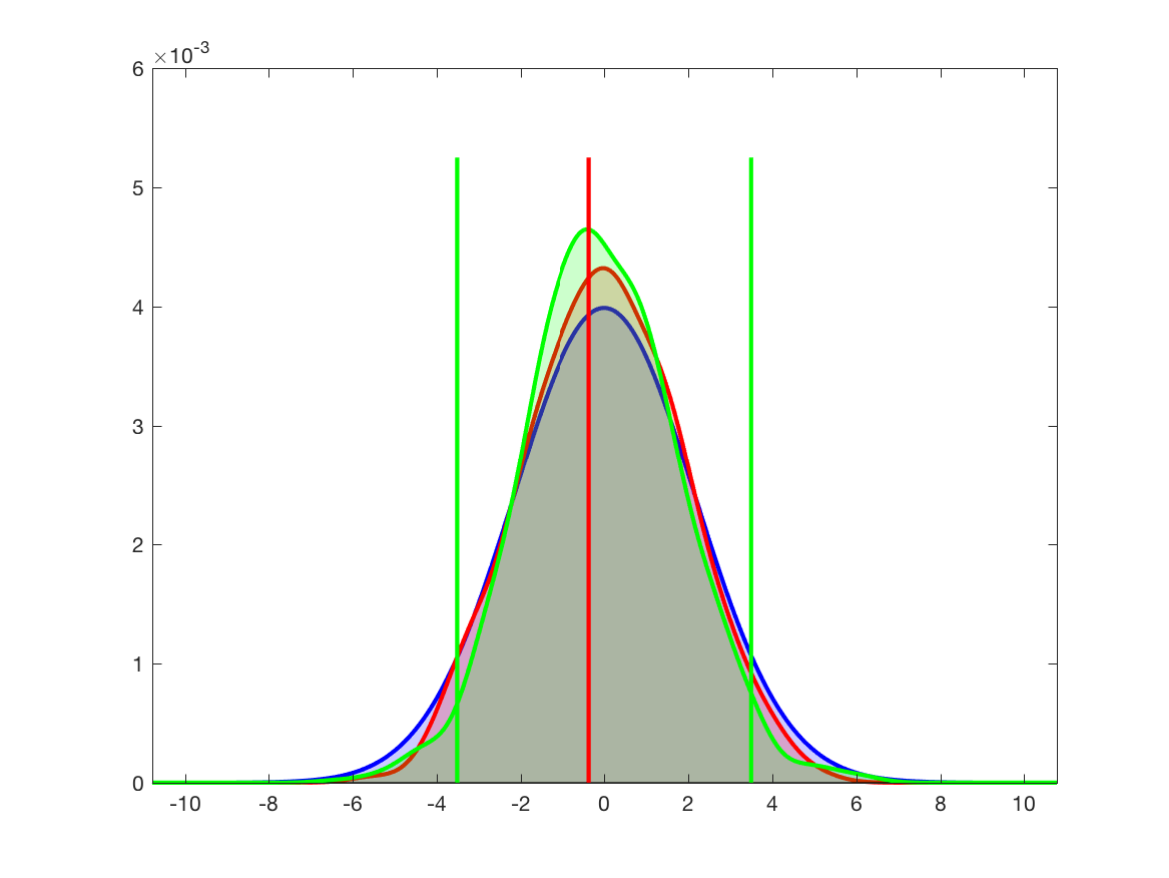}
\end{tabular}
\caption{\label{fig:example_Boot_H1_two}Case $a\neq b$ with two samples. Illustration of the bootstrap  with $\varepsilon=10$ for the grid of size $5\times 5$ and $\varepsilon=100$ for the grid $20\times 20$ to approximate the non-asymptotic distribution of empirical Sinkhorn divergences. Densities in red  (resp.\ blue) represent the distribution of $\rho_{n,m}(\WB_{p,\varepsilon}^p(\hat{a}_n,\hat{b}_m)-\WB_{p,\varepsilon}^p(a,b))$ (resp.\ $\sqrt{\gamma}  \langle G,\ualpha_{\varepsilon}^{a,b}-\frac{1}{2}(\ualpha_{\varepsilon}^{a,a}+ \vbeta_{\varepsilon}^{a,a})\rangle
+ \sqrt{1-\gamma} \langle H,\vbeta_{\varepsilon}^{a,b}-\frac{1}{2}(\ualpha_{\varepsilon}^{b,b}+ \vbeta_{\varepsilon}^{b,b})\rangle$). The green density is the distribution of the random variable $\rho_{n,m}(\WB_{p,\varepsilon}^p(\hat{a}^{\ast}_n,\hat{b}_m^{\ast})-\WB_{p,\varepsilon}^p(\hat{a}_n,\hat{b}_m))$ in Theorem \ref{bootstrap}.} 
\end{figure}

\subsubsection{Hypothesis $a=b$ - One sample.}

As in the previous cases, we consider $a$ to be the uniform distribution on a square grid. We recall that the distributional limit in the right hand side of \eqref{null_one}  is the following mixture of random variables with chi-squared distribution of degree $1$
$$
\frac{1}{2}\sum_{i=1}^N \lambda_i\chi_{\scriptscriptstyle{i}}^2(1)\ \mbox{for} \ \lambda_1,\ldots,\lambda_N \ \mbox{the eigenvalues of} \ \Sigma(a)^{1/2}\partial_{\scriptscriptstyle{11}}^2\WB_{p,\varepsilon}^p(a,a) \Sigma(a)^{1/2}.
$$
It appears to be difficult to compute the density of this distributional limit or to draw samples from it, since computing the Hessian matrix $\partial_{\scriptscriptstyle{11}}^2\WB_{p,\varepsilon}^p(a,a)$ is a delicate task. We thus leave this problem open for future work, and  only rely on the non-asymptotic distribution of $n \WB_{p,\varepsilon}^p(\hat{a}_n,a)$. This justifies the use of the bootstrap procedure described in Section \ref{sec:bootstrap}. We  display the bootstrap statistic in Figure \ref{fig:H0one}. The shape of the non-asymptotic density of $n\WB_{p,\varepsilon}^p(\hat{a}_n,a)$ (red curves in Figure \ref{fig:H0one}) looks chi-squared distributed. In particular, it only takes  positive values. The bootstrap distribution in green also recovers the most significant mass location of the red density.
\renewcommand{\sidecap}[1]{ {\begin{sideways}\parbox{0.2\textwidth}{\centering #1}\end{sideways}} }
\begin{figure}[ht!]
\centering
\begin{tabular}{cccc|ccc}
& $n=10^3$ & $n=10^4$ & $n=10^5$ & \hspace{0.2cm}  $n=10^3$ & $n=10^4$ & $n=10^5$ \\
\sidecap{$\varepsilon=1$} & \includegraphics[width=0.15\textwidth,height=0.22\textwidth]{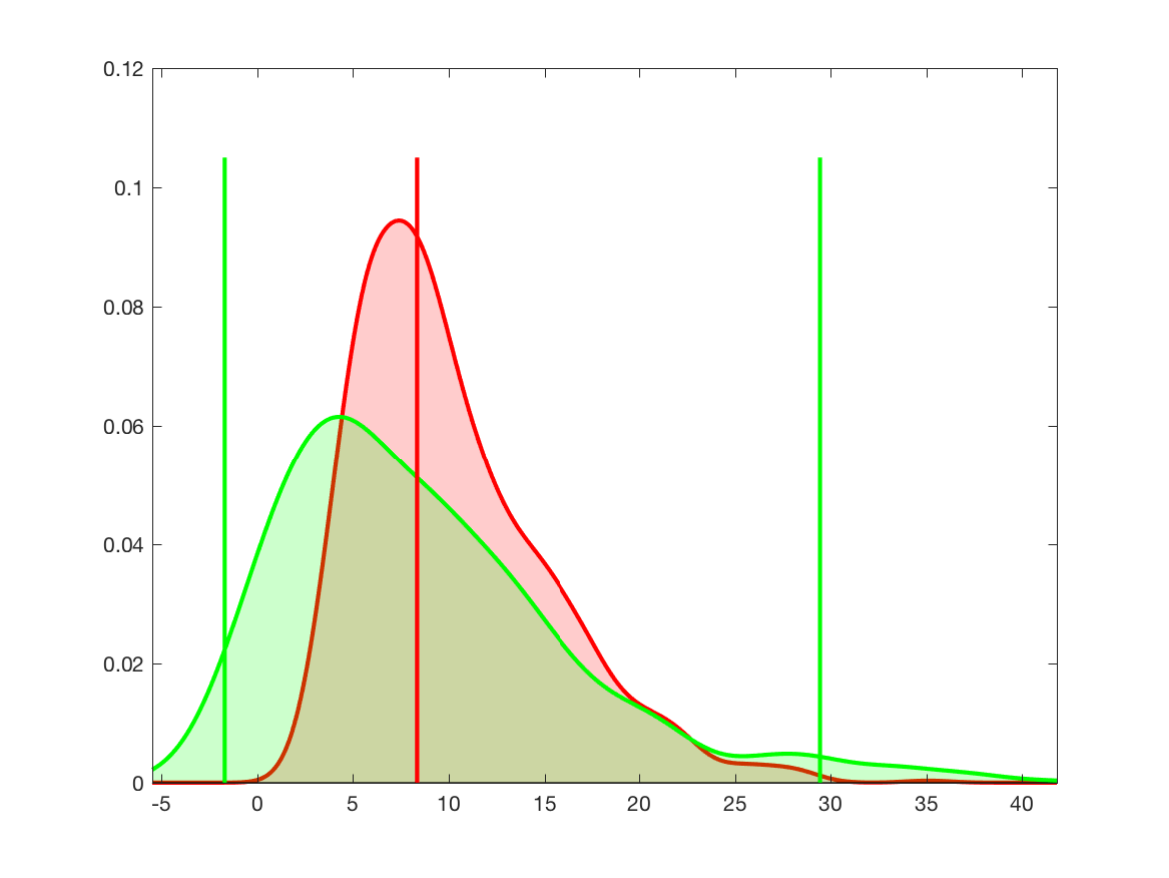} &
\includegraphics[width=0.15\textwidth,height=0.22\textwidth]{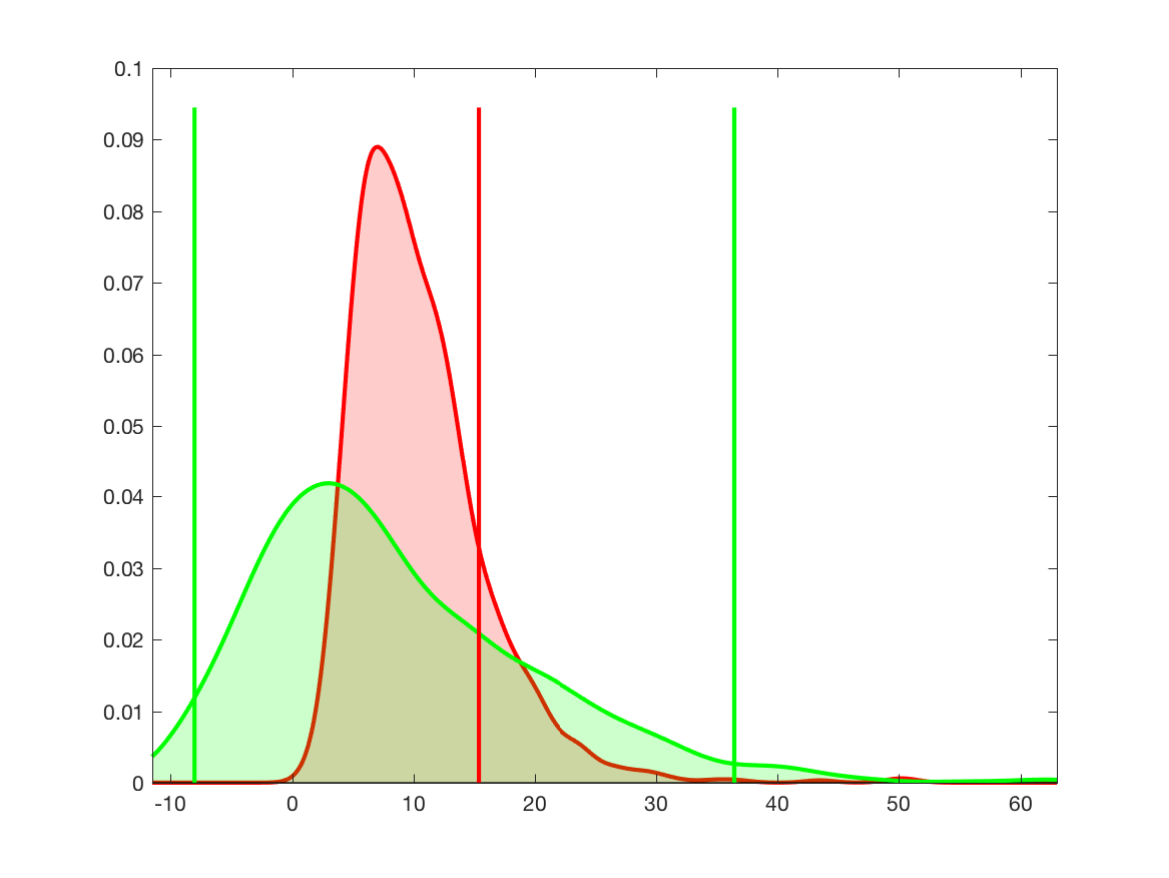} &
\includegraphics[width=0.15\textwidth,height=0.22\textwidth]{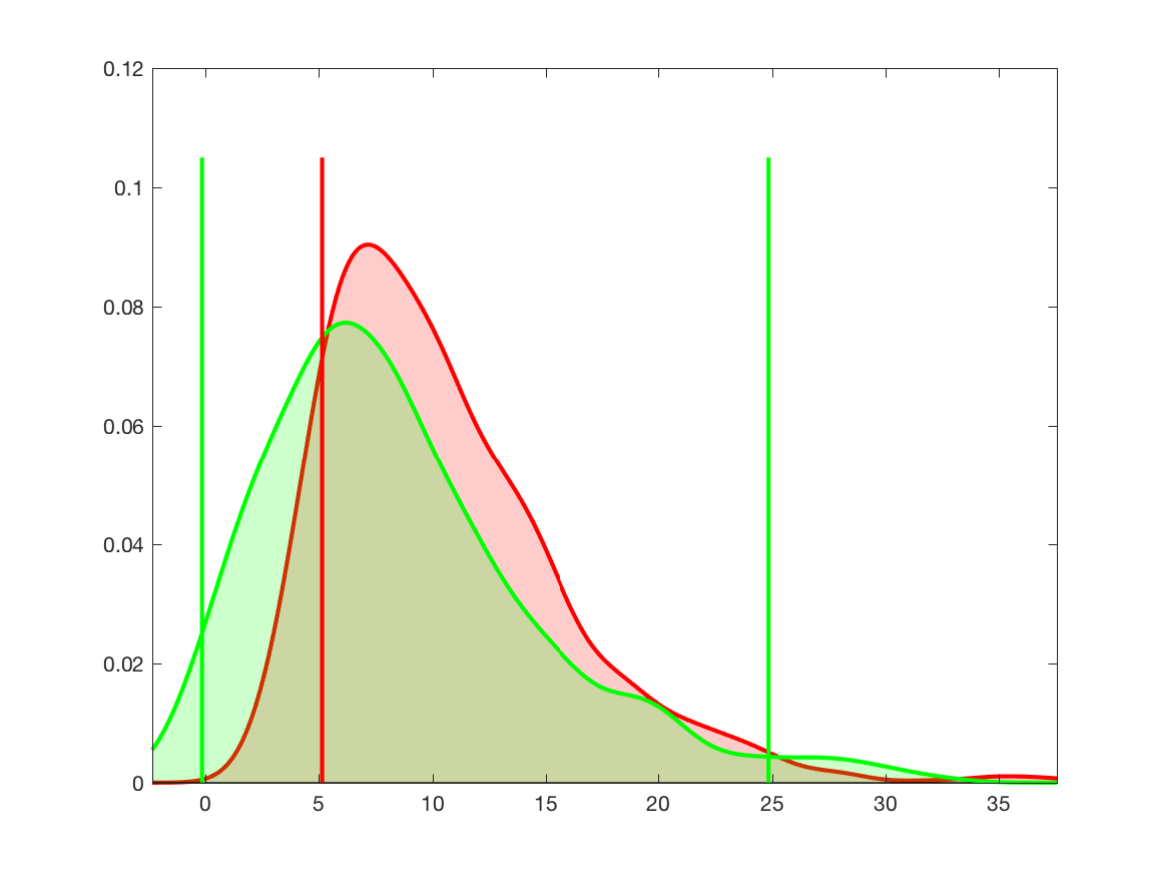} &
\hspace{0.2cm} \includegraphics[width=0.15\textwidth,height=0.22\textwidth]{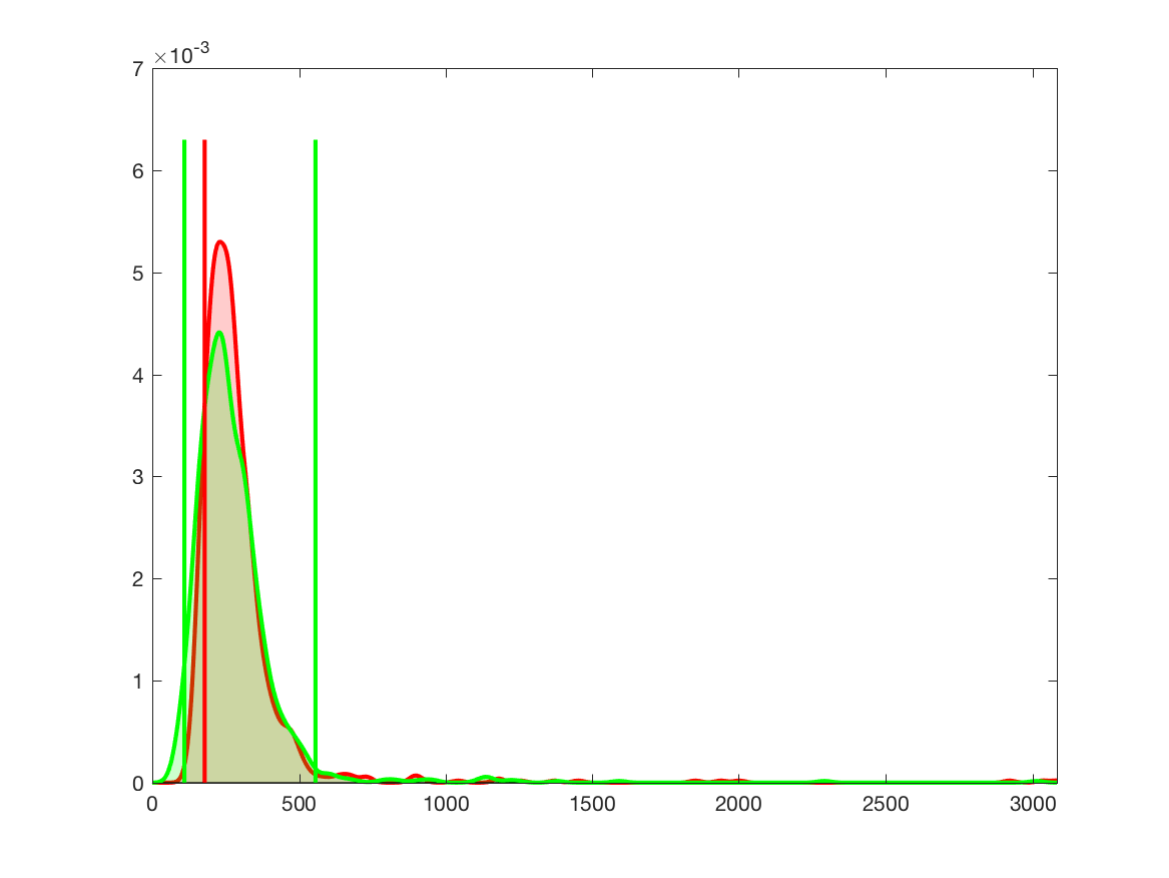} &
\includegraphics[width=0.15\textwidth,height=0.22\textwidth]{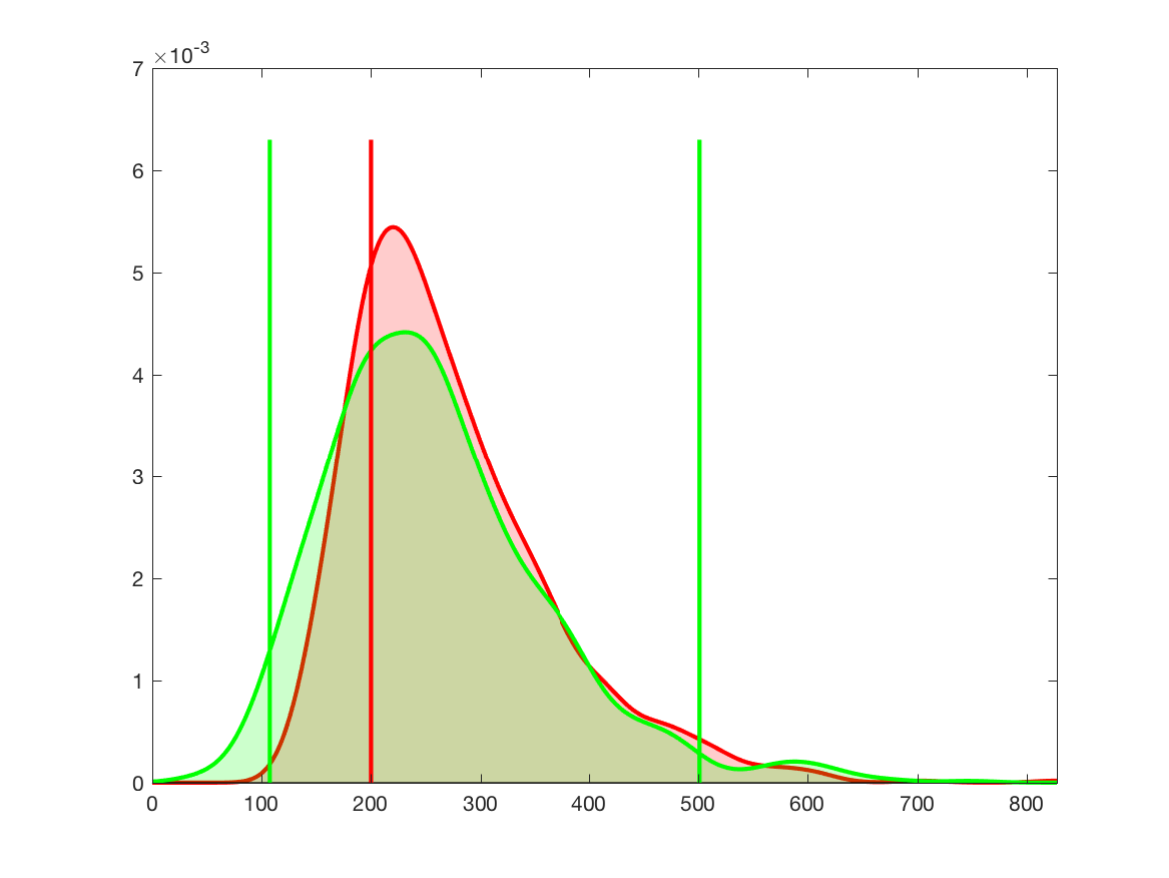} &
\includegraphics[width=0.15\textwidth,height=0.22\textwidth]{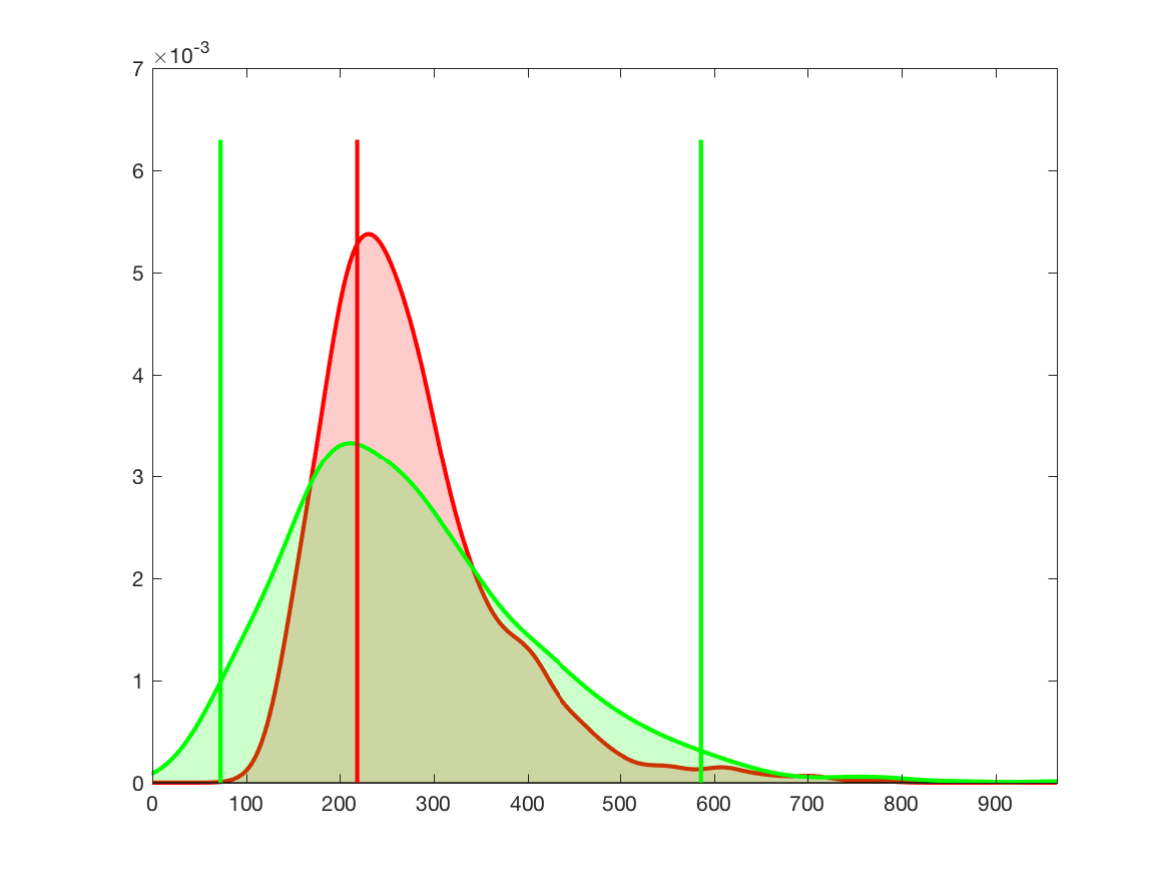}\\

\sidecap{$\varepsilon=10$} & \includegraphics[width=0.15\textwidth,height=0.22\textwidth]{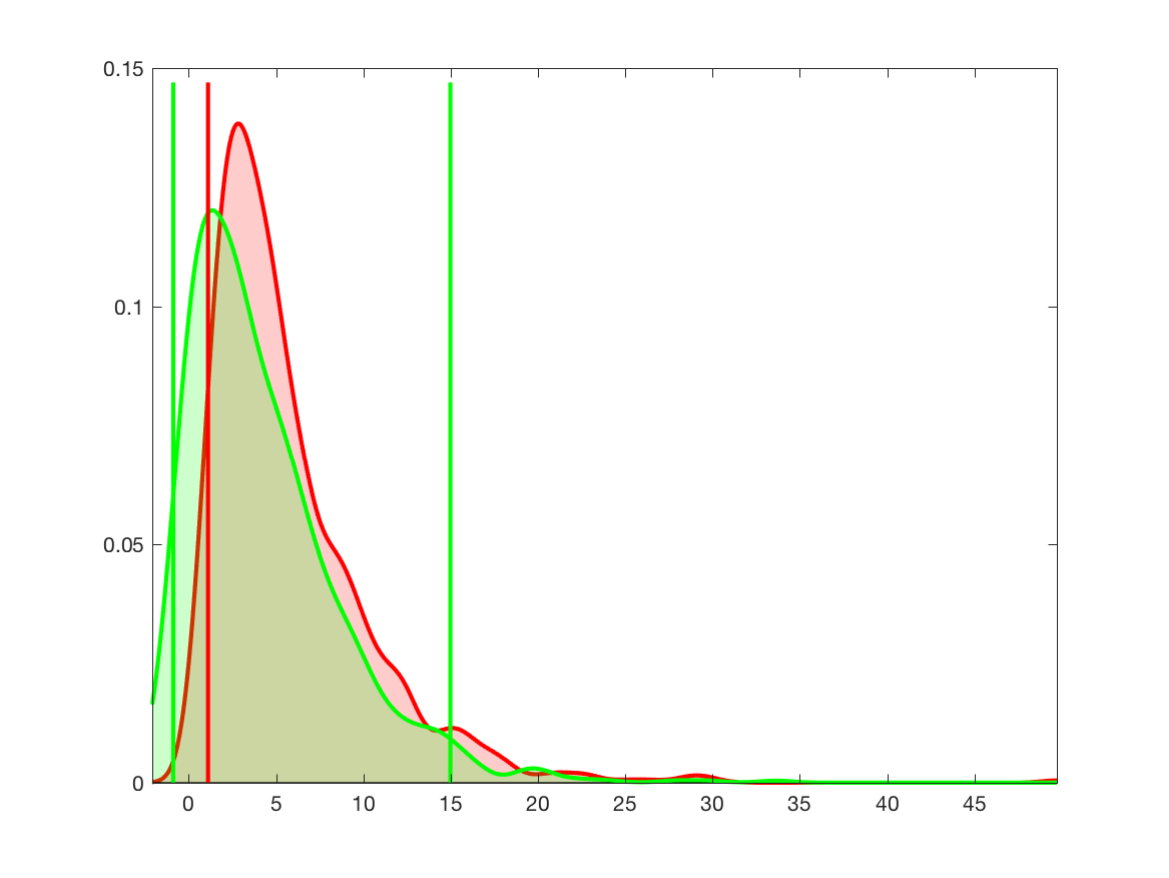} &
\includegraphics[width=0.15\textwidth,height=0.22\textwidth]{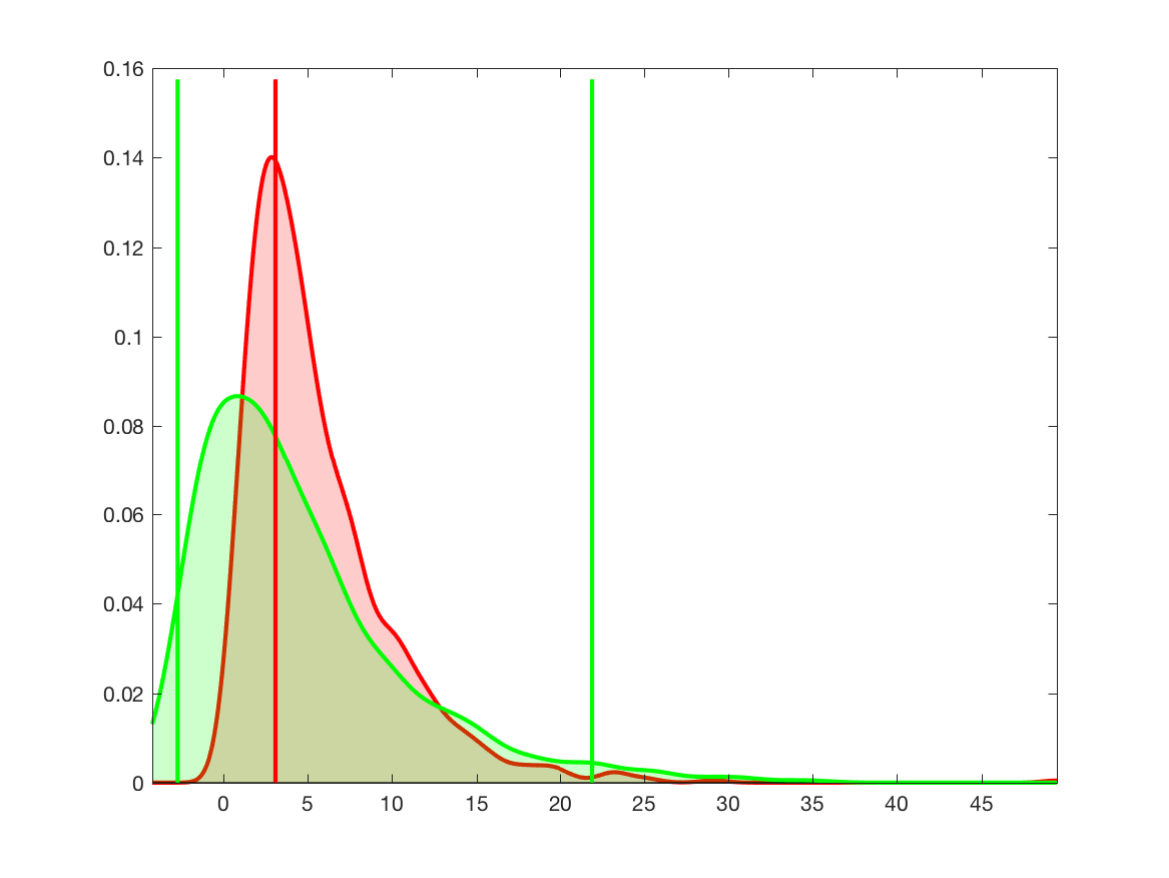} &
\includegraphics[width=0.15\textwidth,height=0.22\textwidth]{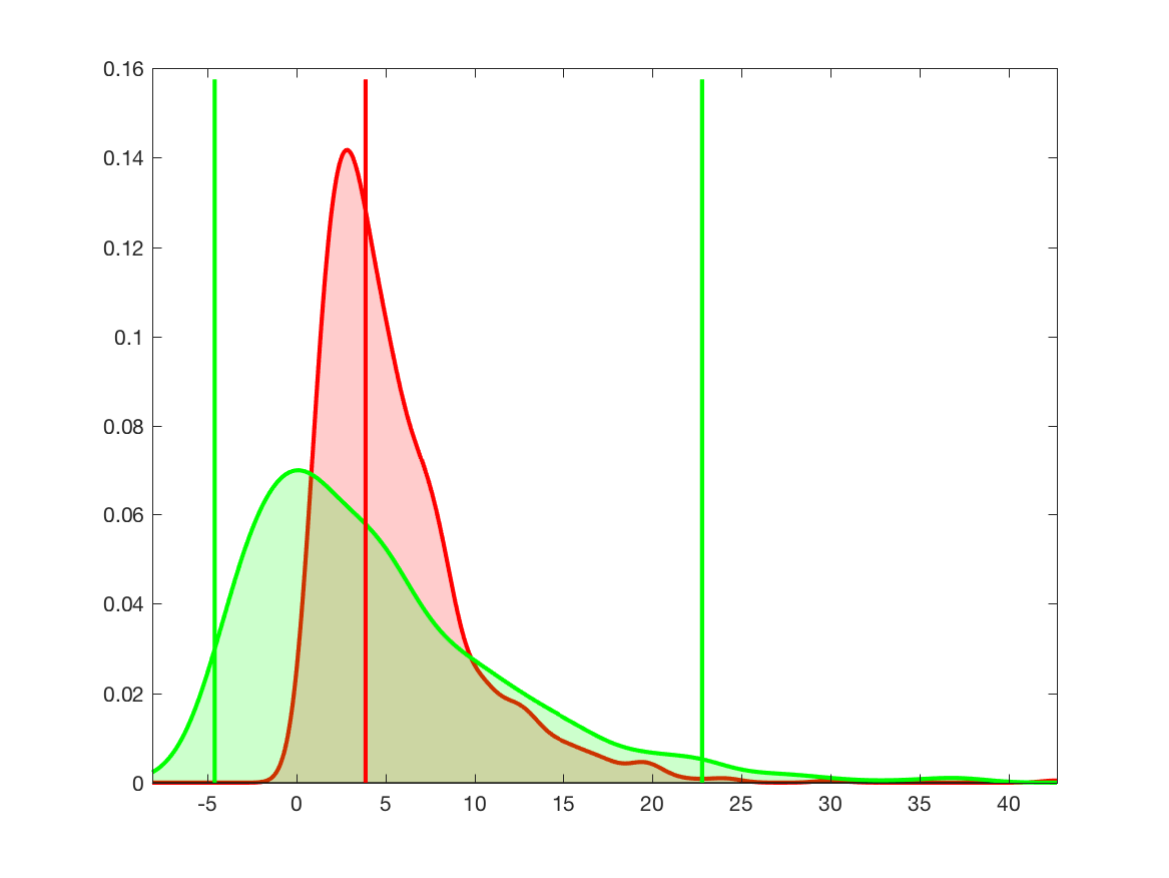} &
\hspace{0.2cm}
 \includegraphics[width=0.15\textwidth,height=0.22\textwidth]{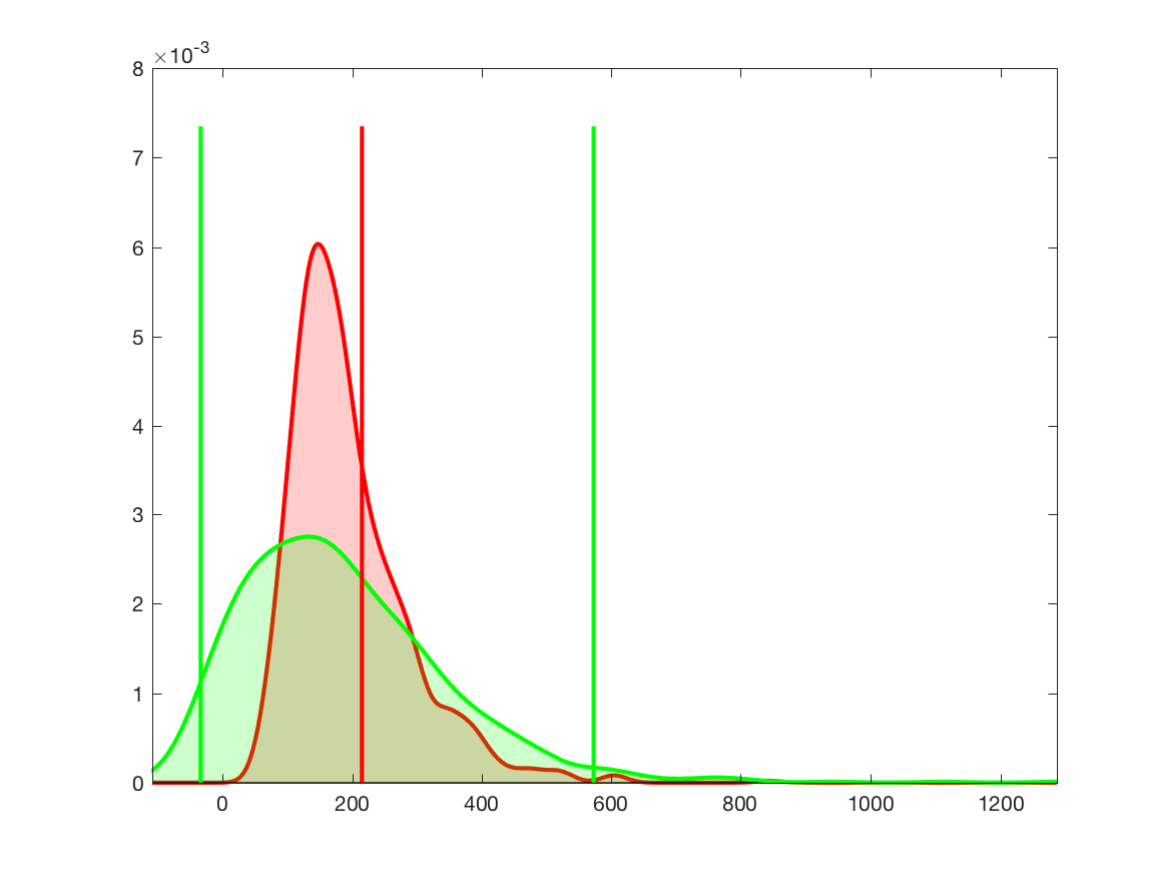} &
\includegraphics[width=0.15\textwidth,height=0.22\textwidth]{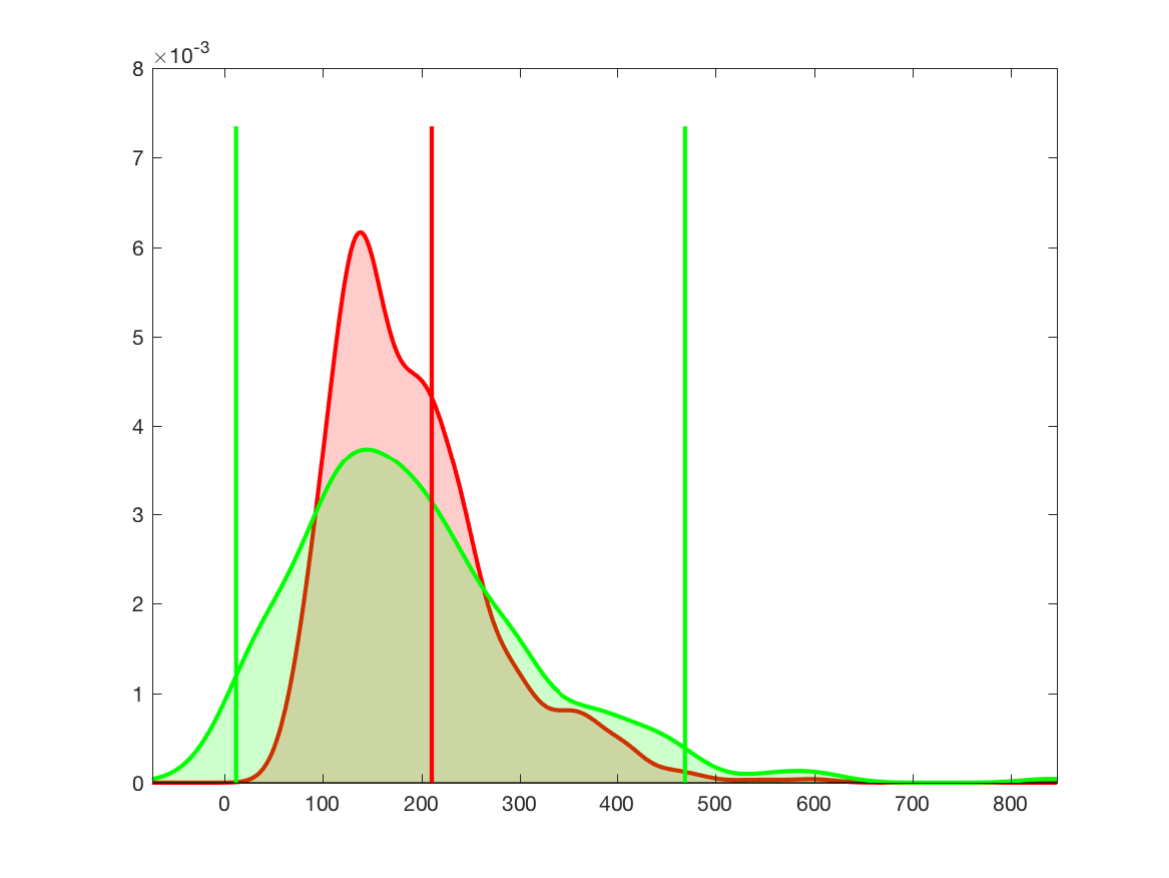} &
\includegraphics[width=0.15\textwidth,height=0.22\textwidth]{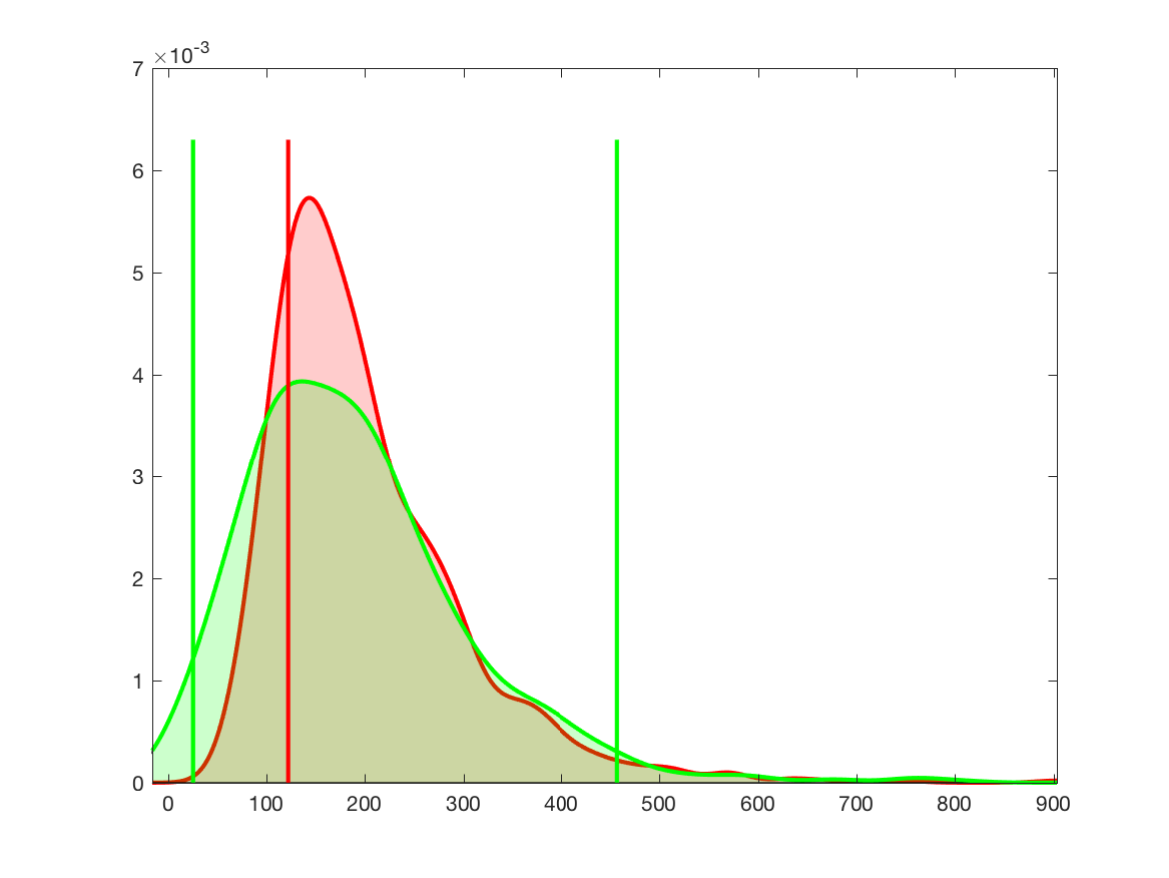}\\

\sidecap{$\varepsilon=100$} &  \includegraphics[width=0.15\textwidth,height=0.22\textwidth]{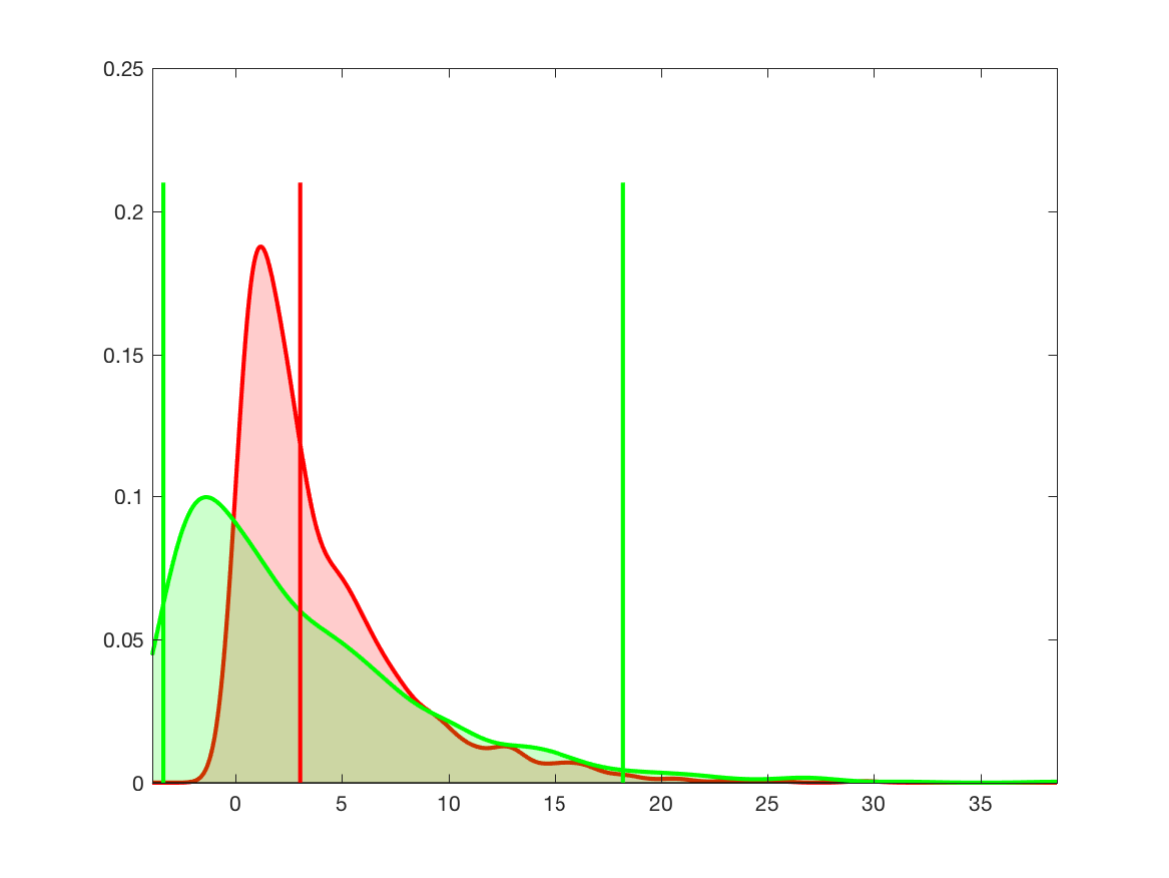} &
\includegraphics[width=0.15\textwidth,height=0.22\textwidth]{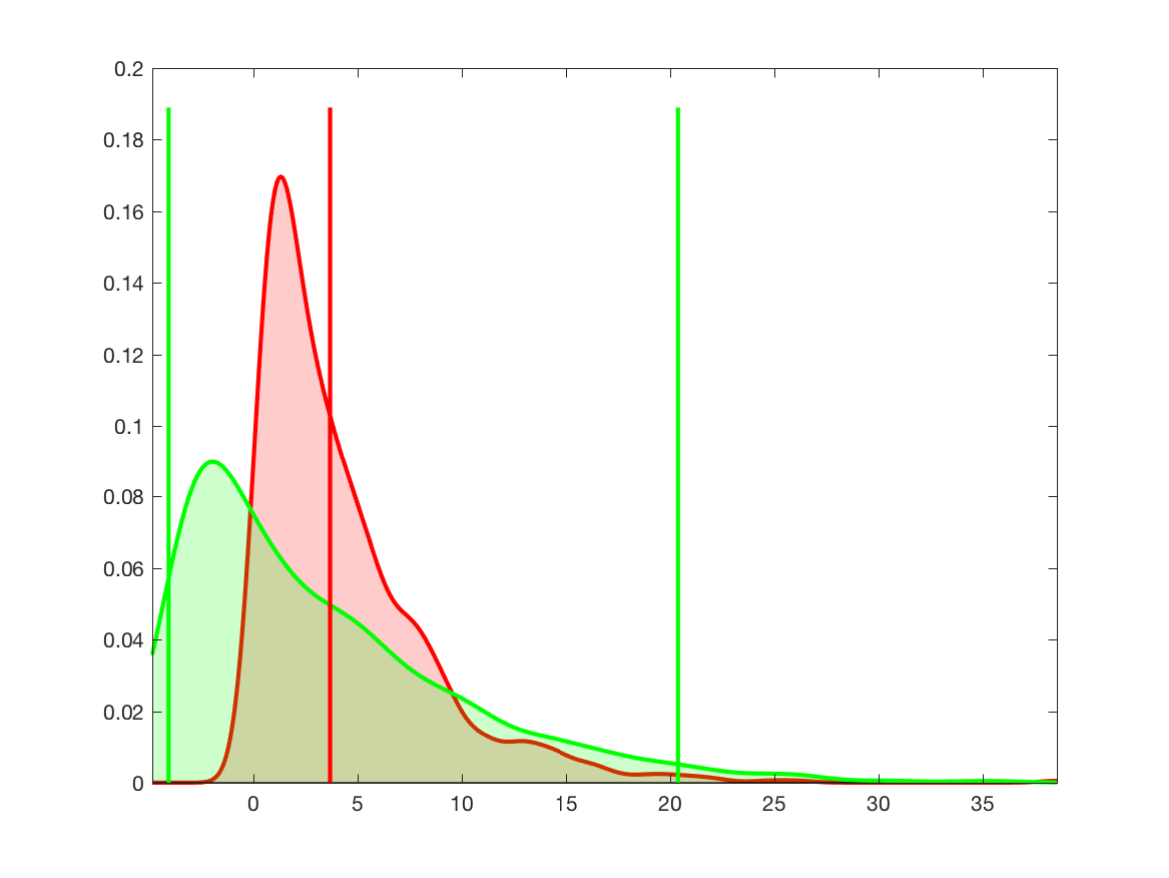} &
\includegraphics[width=0.15\textwidth,height=0.22\textwidth]{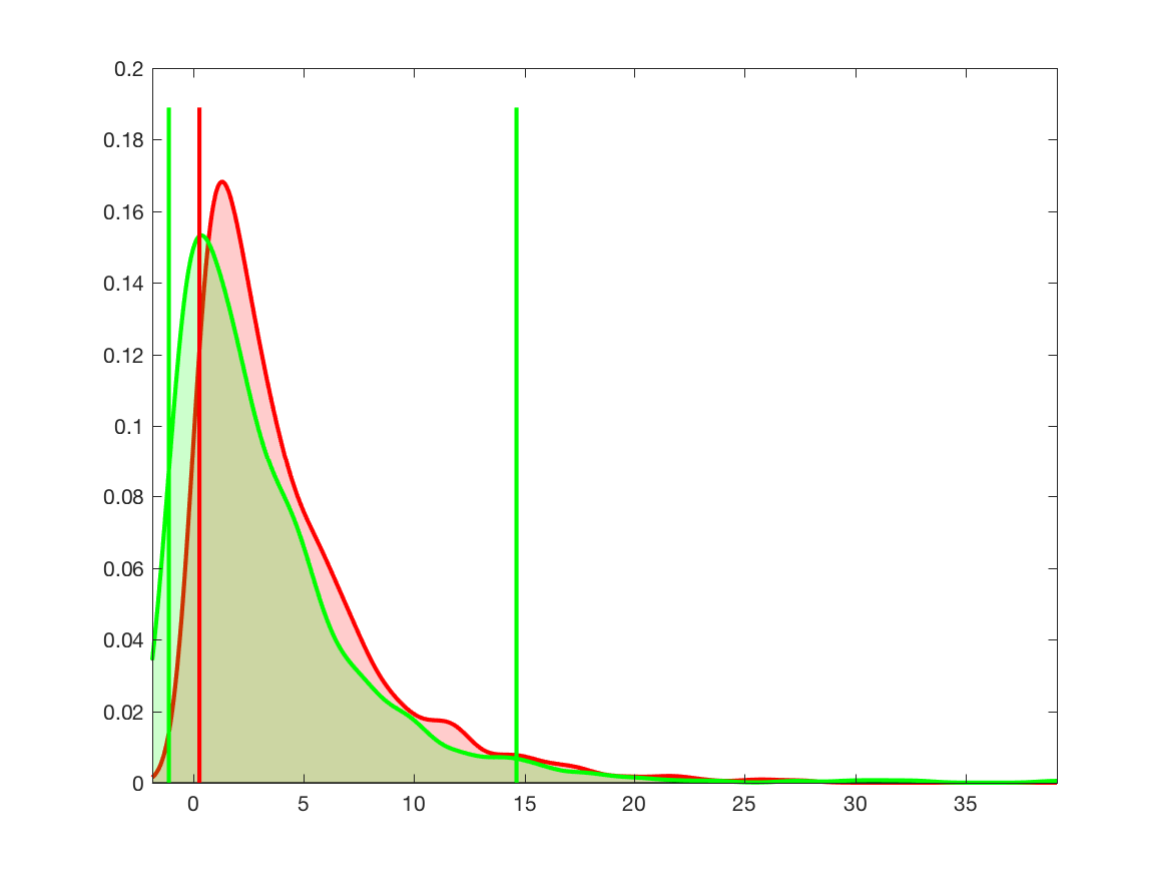} &
\includegraphics[width=0.15\textwidth,height=0.22\textwidth]{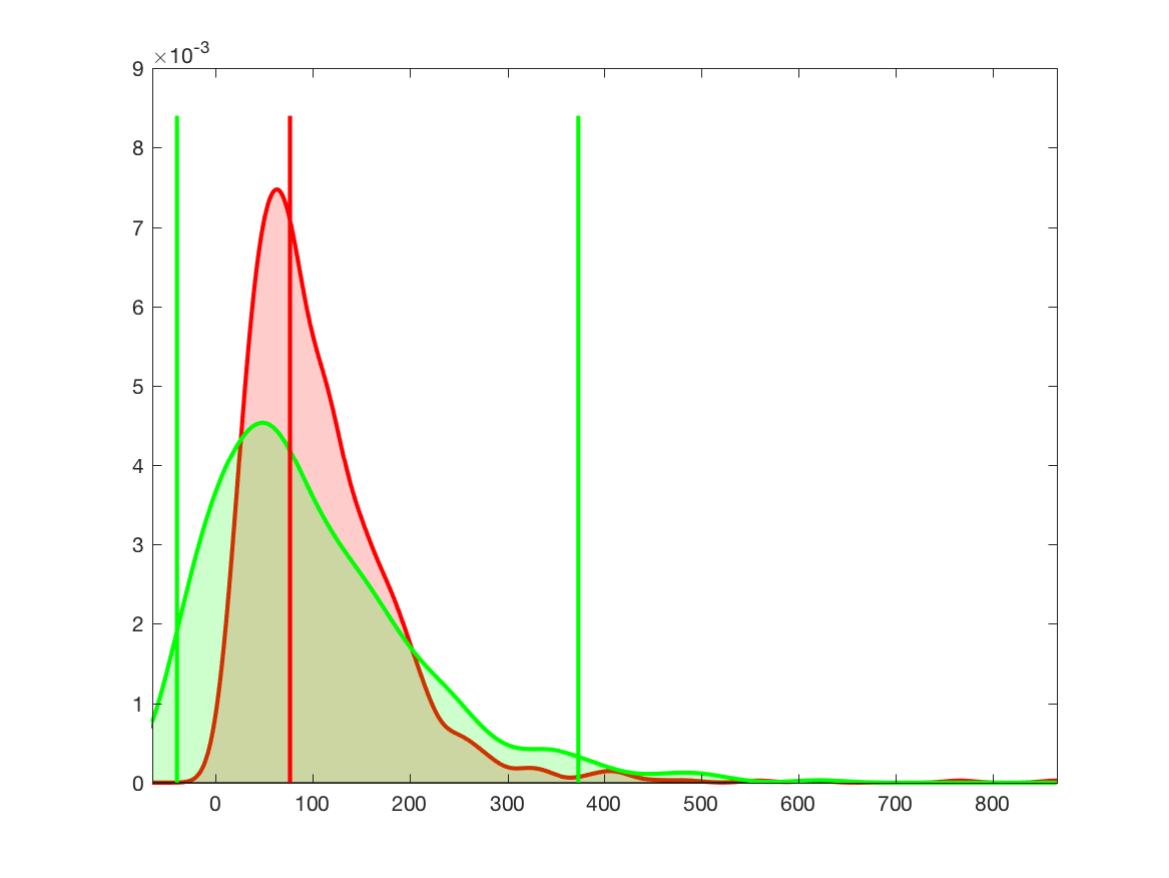} &
\includegraphics[width=0.15\textwidth,height=0.22\textwidth]{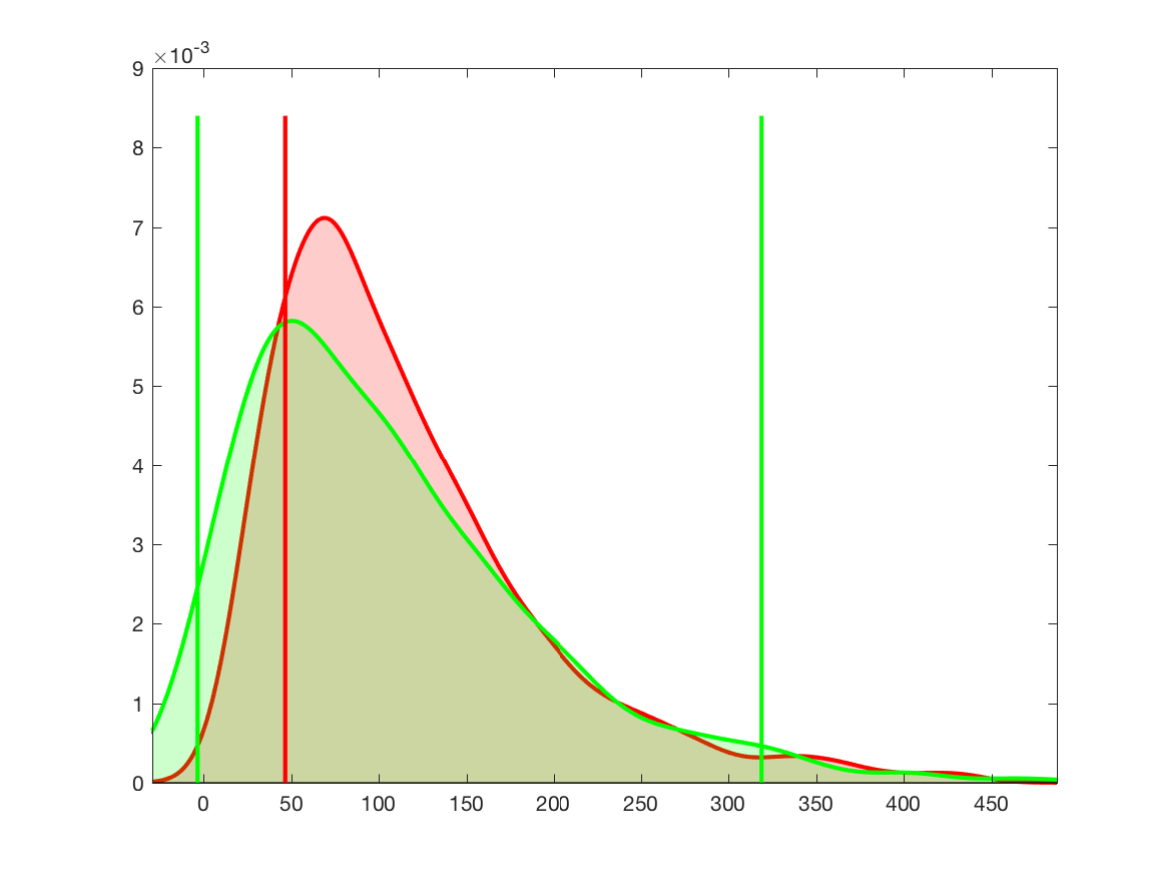} &
\includegraphics[width=0.15\textwidth,height=0.22\textwidth]{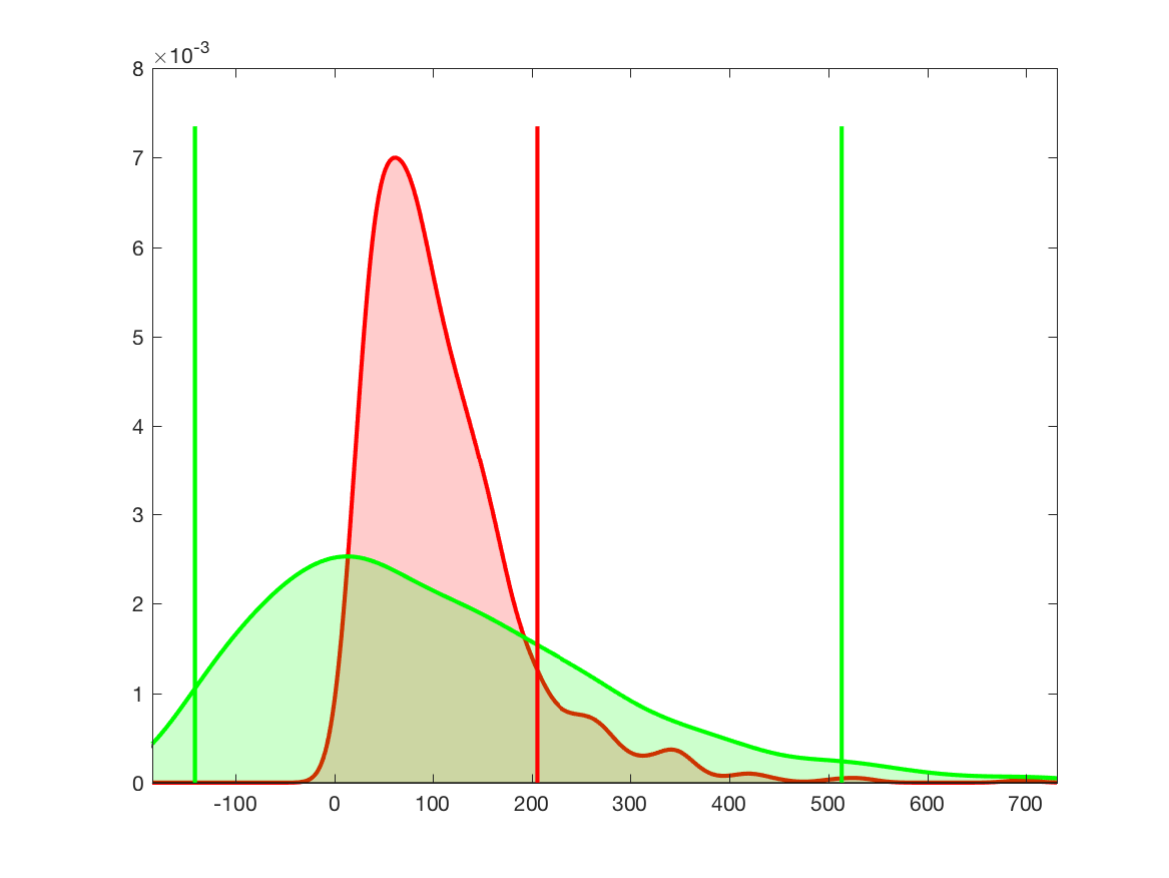}\\
&\multicolumn{3}{c}{Grid $5\times 5$}&\multicolumn{3}{c}{Grid $20\times 20$}
\end{tabular}
\caption{\label{fig:H0one}Case $a=b$ with one sample. Illustration of the bootstrap with $\varepsilon=1, 10, 100$ and two grids of size $5\times 5$ (left) and $20\times 20$ (right) to approximate the non-asymptotic distribution of the empirical Sinkhorn loss. Densities in red represent the distribution of $n\WB_{p,\varepsilon}^p(\hat{a}_n,a)$. The green density represents the distribution of the random variable $n(\WB_{p,\varepsilon}^p(\hat{a}_n,a)-\WB_{p,\varepsilon}^p(\hat{a}^{\ast}_n,a)-\langle\ualpha^{\hat{a}_n,a},\hat{a}_n^{\ast}-\hat{a}_n\rangle)$ in \eqref{eq:boot_null_loss_one}.}
\end{figure}

\subsubsection{Hypothesis $a=b$ - Two samples.}

We still consider $a=b$ to be the uniform distribution on a square grid and we sample two measures from $a$ denoted $\hat{a}_n,\hat{b}_m$. We  compute the non-asymptotic distribution of $(nm/(m+n)) \WB_{p,\varepsilon}^p(\hat{a}_n,\hat{b}_m)$ which, from Theorem \ref{Th_CLT_loss_H0_two}, must converge to 
$\frac{1}{2}\sum_{i=1}^N\tilde{\lambda}_i\chi^2_{\scriptscriptstyle{i}}(1)$  with $ \{\tilde{\lambda}_i\}_i$ the eigenvalues of
$$ \diag(\sqrt{\gamma}\Sigma(a)^{1/2},\sqrt{1-\gamma}\Sigma(a)^{1/2})\nabla^2\WB_{p,\varepsilon}^p(a,a) \diag(\sqrt{\gamma}\Sigma(a)^{1/2},\sqrt{1-\gamma}\Sigma(a)^{1/2}).$$
The results are displayed in red in Figure \ref{fig:H0two}, together  with the bootstrap distribution (in green) $\rho_{n,m}^2(\WB_{p,\varepsilon}^p(\hat{a}^{\ast}_n,\hat{b}_m^{\ast})-\WB_{p,\varepsilon}^p(\hat{a}_n,\hat{b}_m)-\langle\ualpha^{\hat{a}_n,\hat{b}_m},\hat{a}_n^{\ast}-\hat{a}_n\rangle-\langle\vbeta^{\hat{a}_n,\hat{b}_m},\hat{b}_m^{\ast}-\hat{b}_m\rangle)$. We obtain similar results  to the one sample case.

\renewcommand{\sidecap}[1]{ {\begin{sideways}\parbox{0.2\textwidth}{\centering #1}\end{sideways}} }
\begin{figure}[ht!]
\centering
\begin{tabular}{cccc|cc}
& $n=10^3$ & $n=10^4$ & $n=10^5$ & \hspace{0.2cm}  %$n=10^3$ &
 $n=10^4$ & $n=10^5$ \\
\sidecap{$\varepsilon=1$} & \includegraphics[width=0.15\textwidth,height=0.22\textwidth]{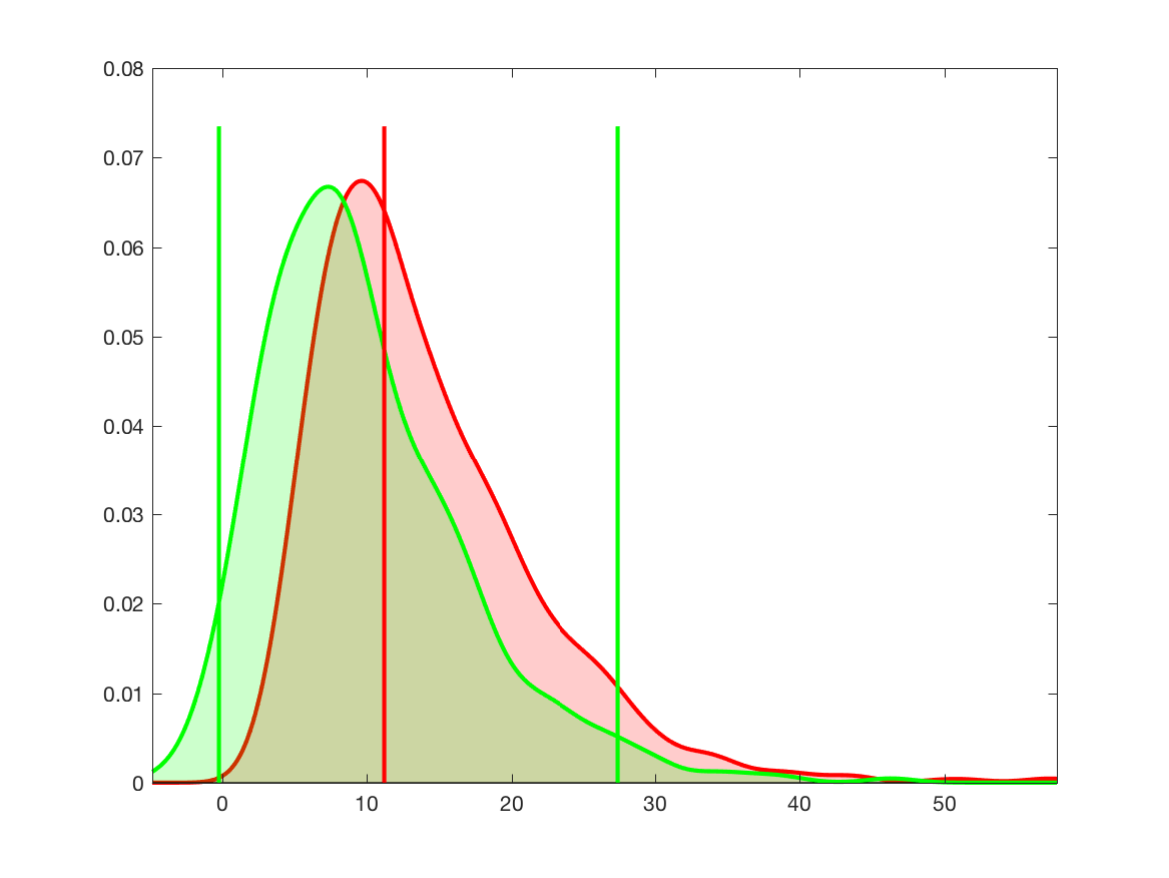} &
\includegraphics[width=0.15\textwidth,height=0.22\textwidth]{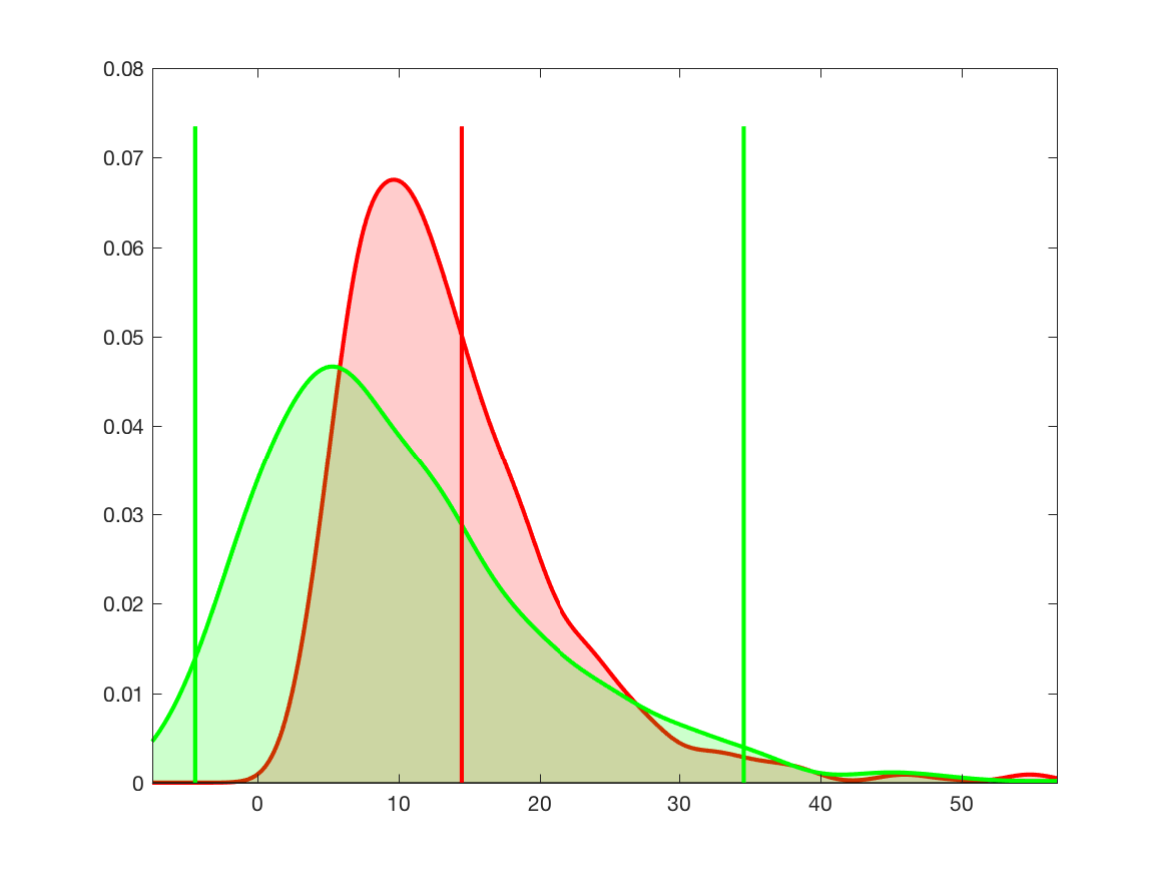} &
\includegraphics[width=0.15\textwidth,height=0.22\textwidth]{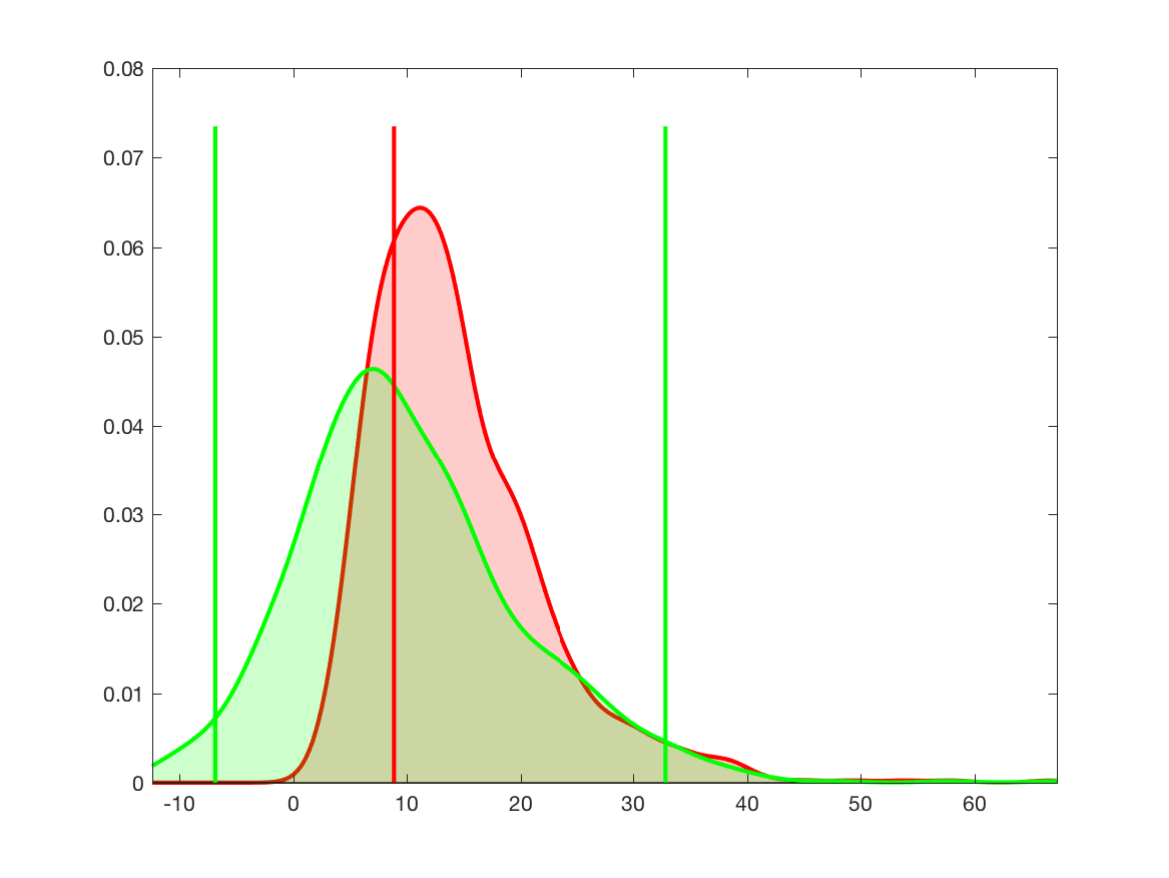} &
\hspace{0.2cm}  %\includegraphics[width=0.15\textwidth,height=0.22\textwidth]{../CodeTCL/FiguresLoss/format_eps/Grille20x20/Boot_two_samples_H0_lambda1_Ne3_RE.pdf} &
\includegraphics[width=0.15\textwidth,height=0.22\textwidth]{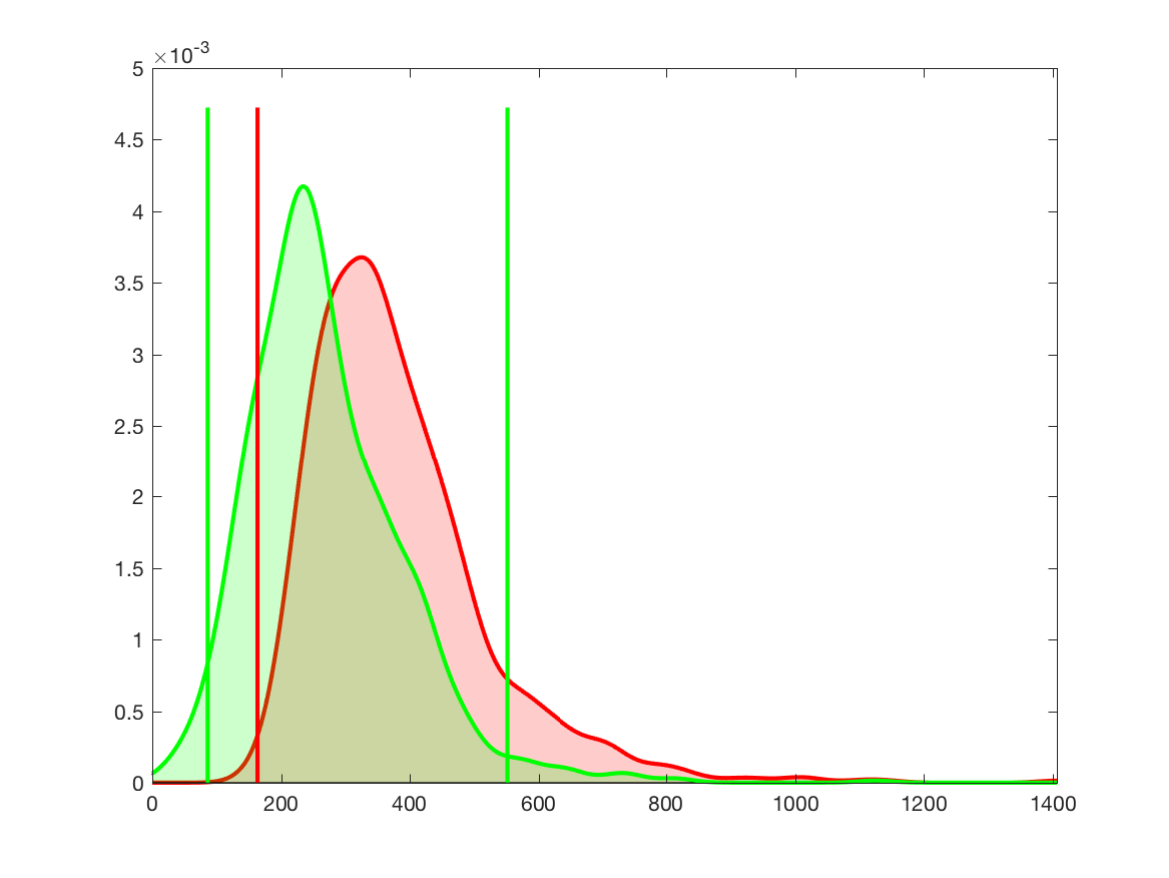} &
\includegraphics[width=0.15\textwidth,height=0.22\textwidth]{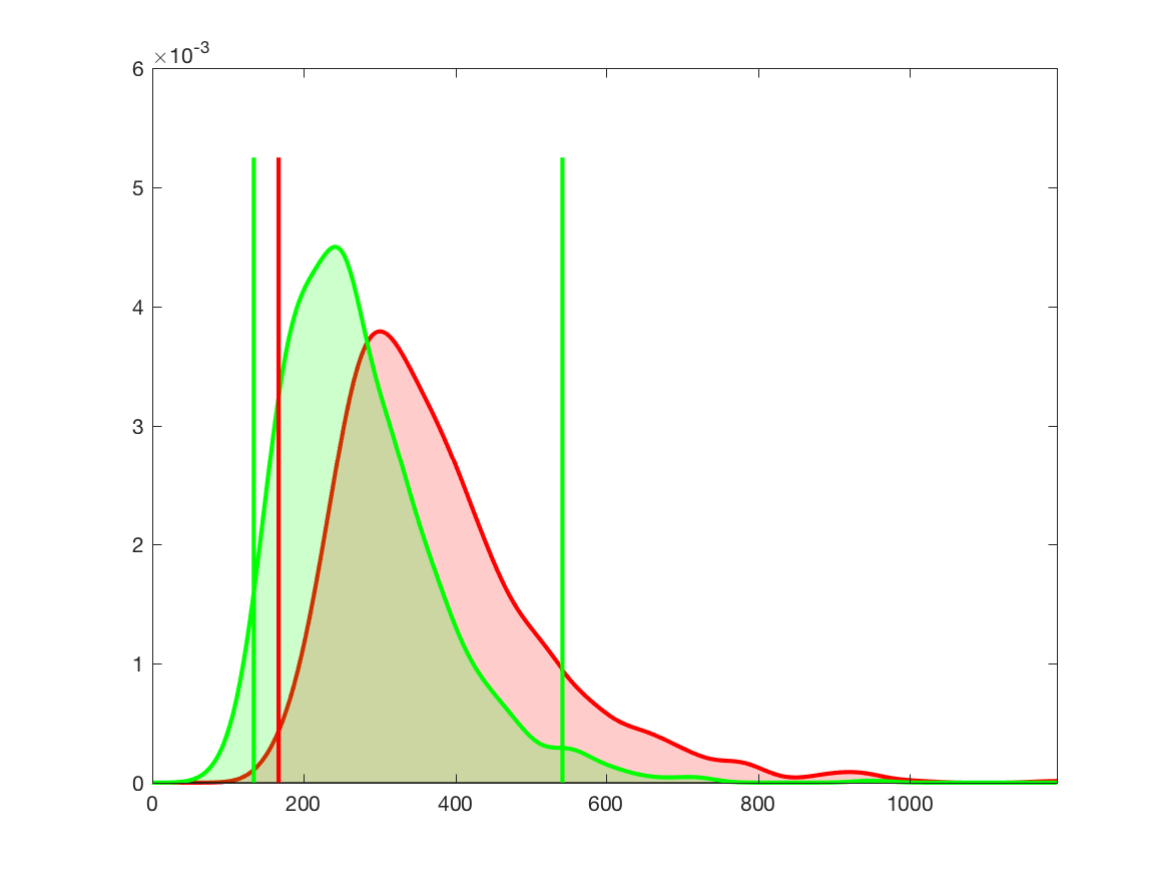}\\

\sidecap{$\varepsilon=10$} & \includegraphics[width=0.15\textwidth,height=0.22\textwidth]{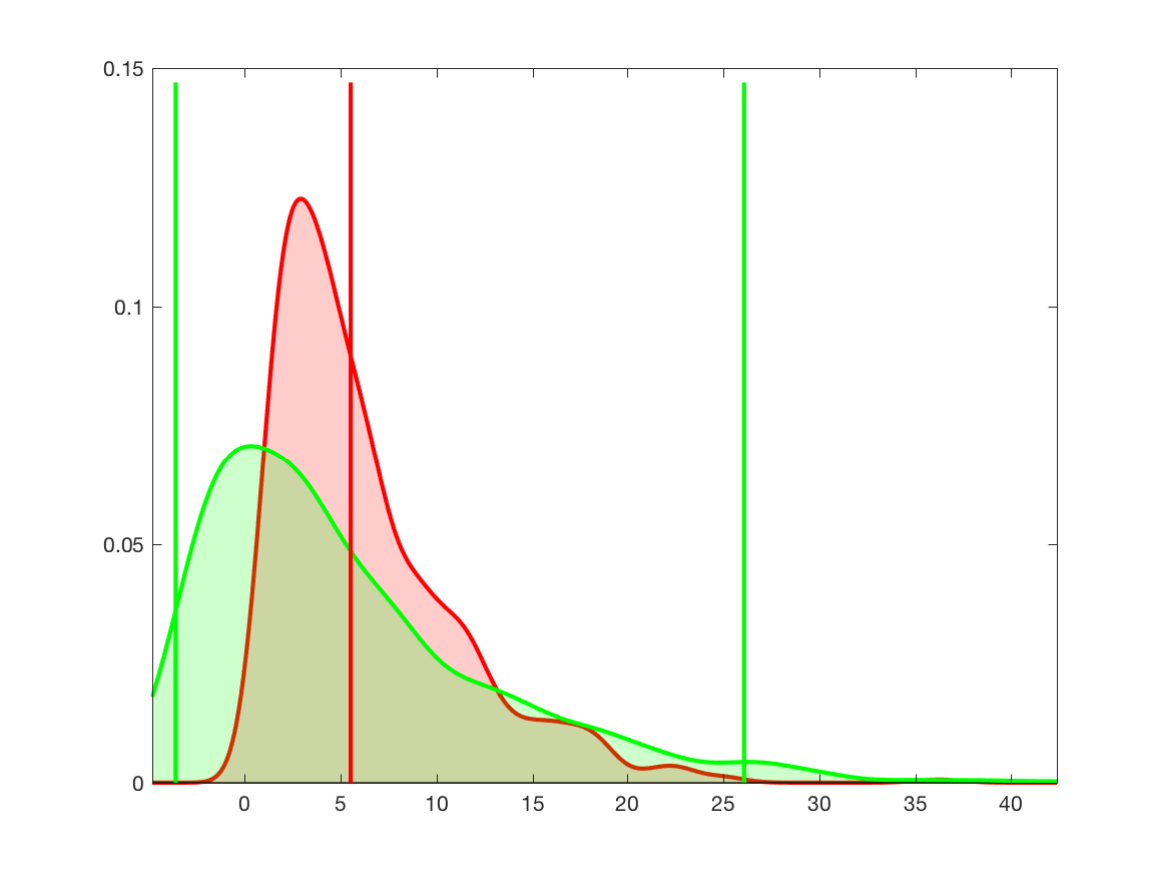} &
\includegraphics[width=0.15\textwidth,height=0.22\textwidth]{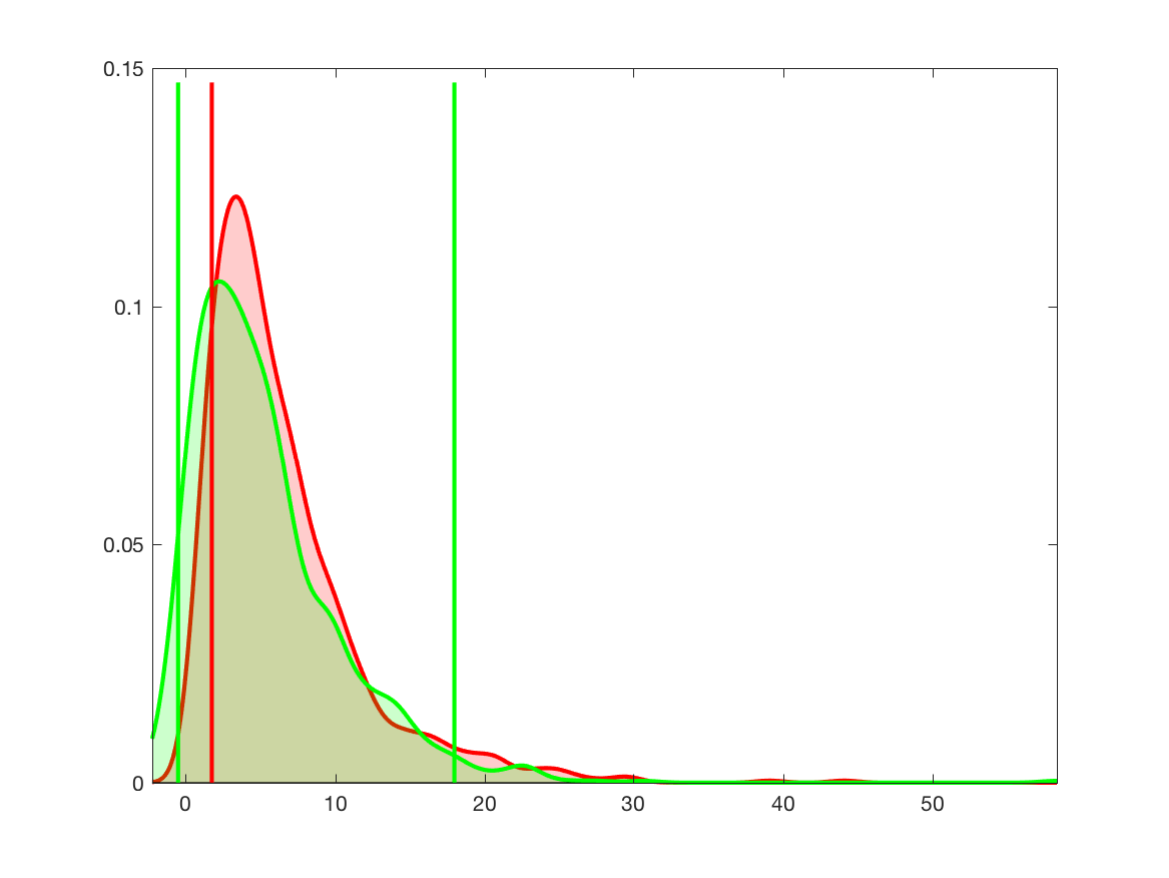} &
\includegraphics[width=0.15\textwidth,height=0.22\textwidth]{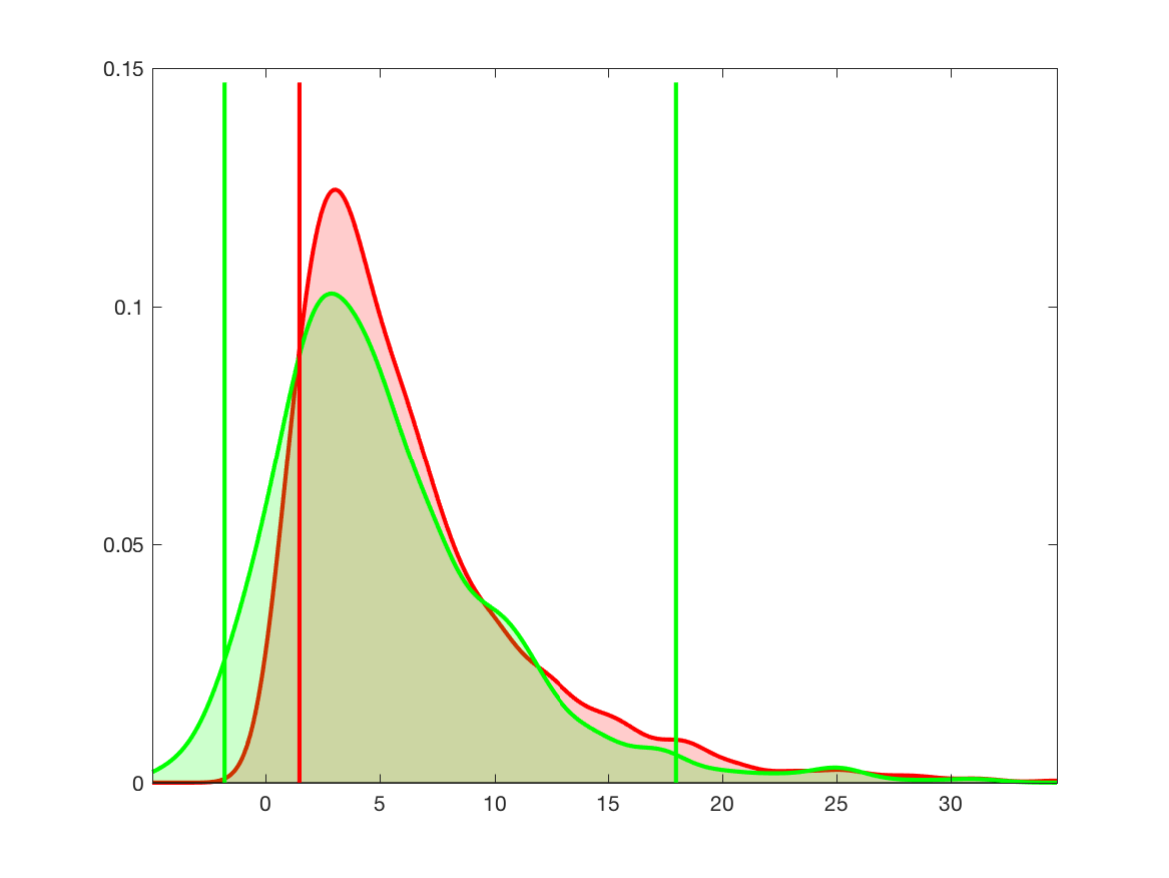} &
\hspace{0.2cm}  %\includegraphics[width=0.15\textwidth,height=0.22\textwidth]{../CodeTCL/FiguresLoss/format_eps/Grille20x20/Boot_two_samples_H0_lambda10_Ne3_RE.pdf} &
\includegraphics[width=0.15\textwidth,height=0.22\textwidth]{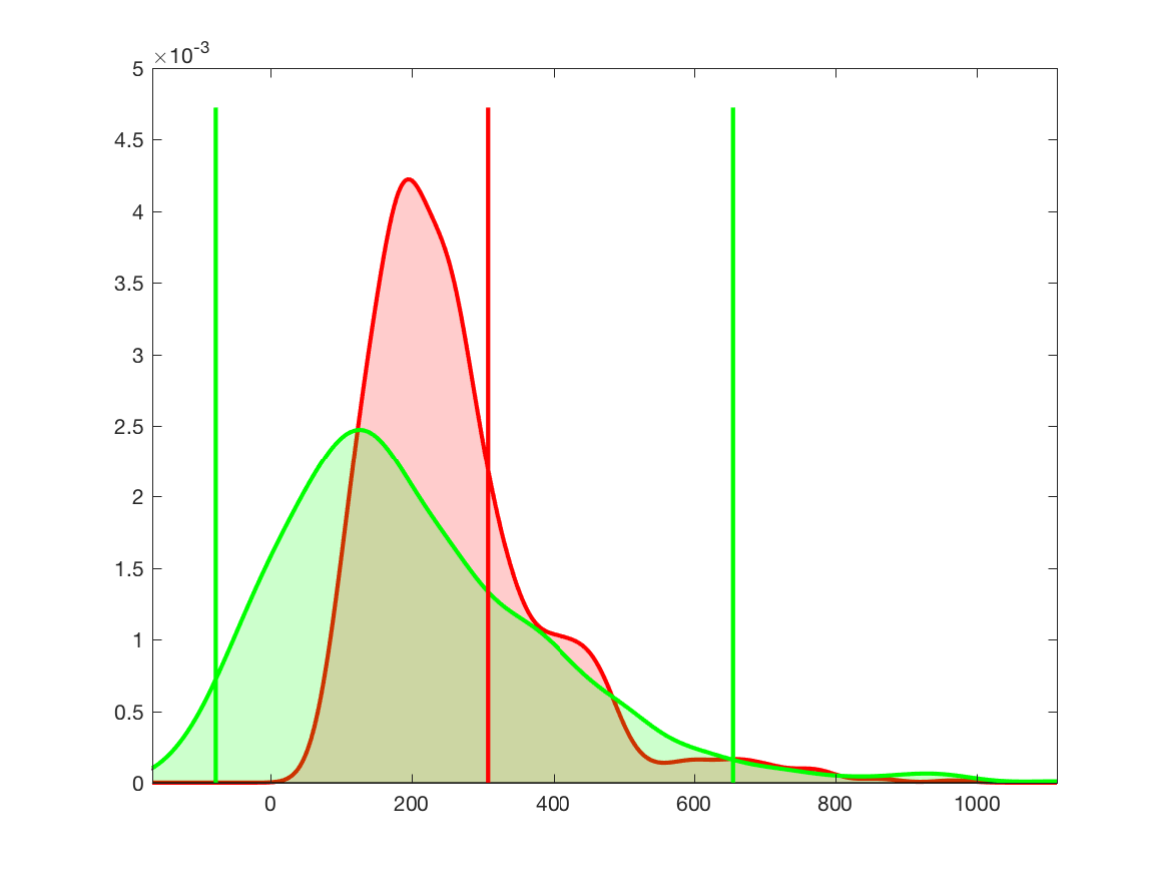} &
\includegraphics[width=0.15\textwidth,height=0.22\textwidth]{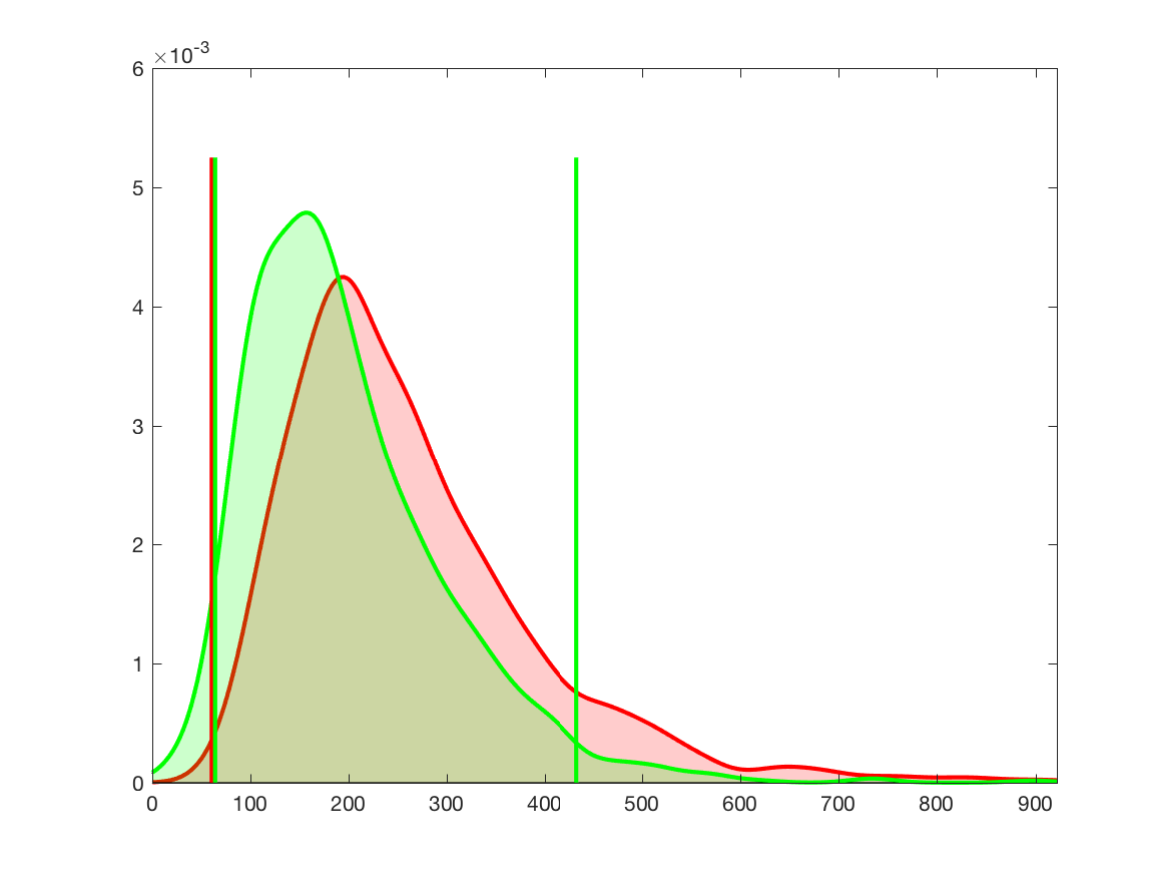}\\

\sidecap{$\varepsilon=100$} & \includegraphics[width=0.15\textwidth,height=0.22\textwidth]{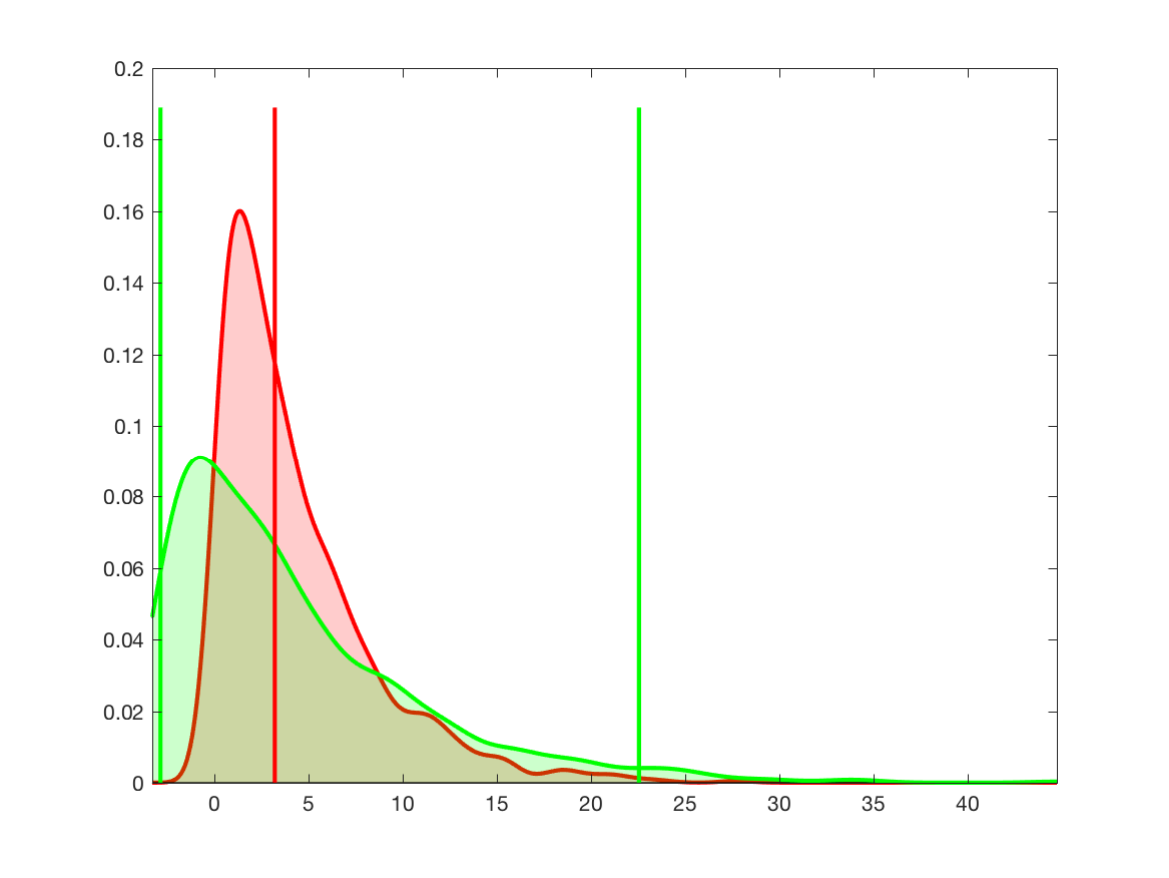} &
\includegraphics[width=0.15\textwidth,height=0.22\textwidth]{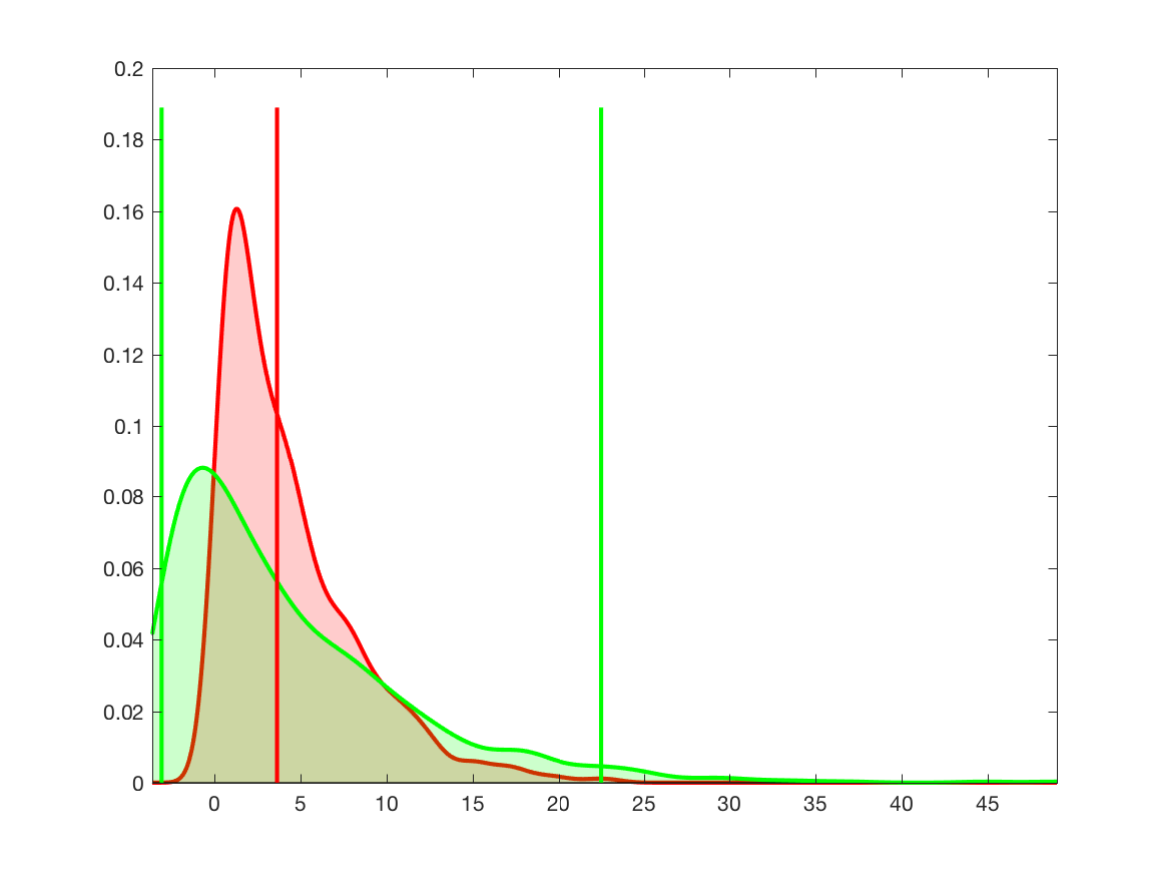} &
\includegraphics[width=0.15\textwidth,height=0.22\textwidth]{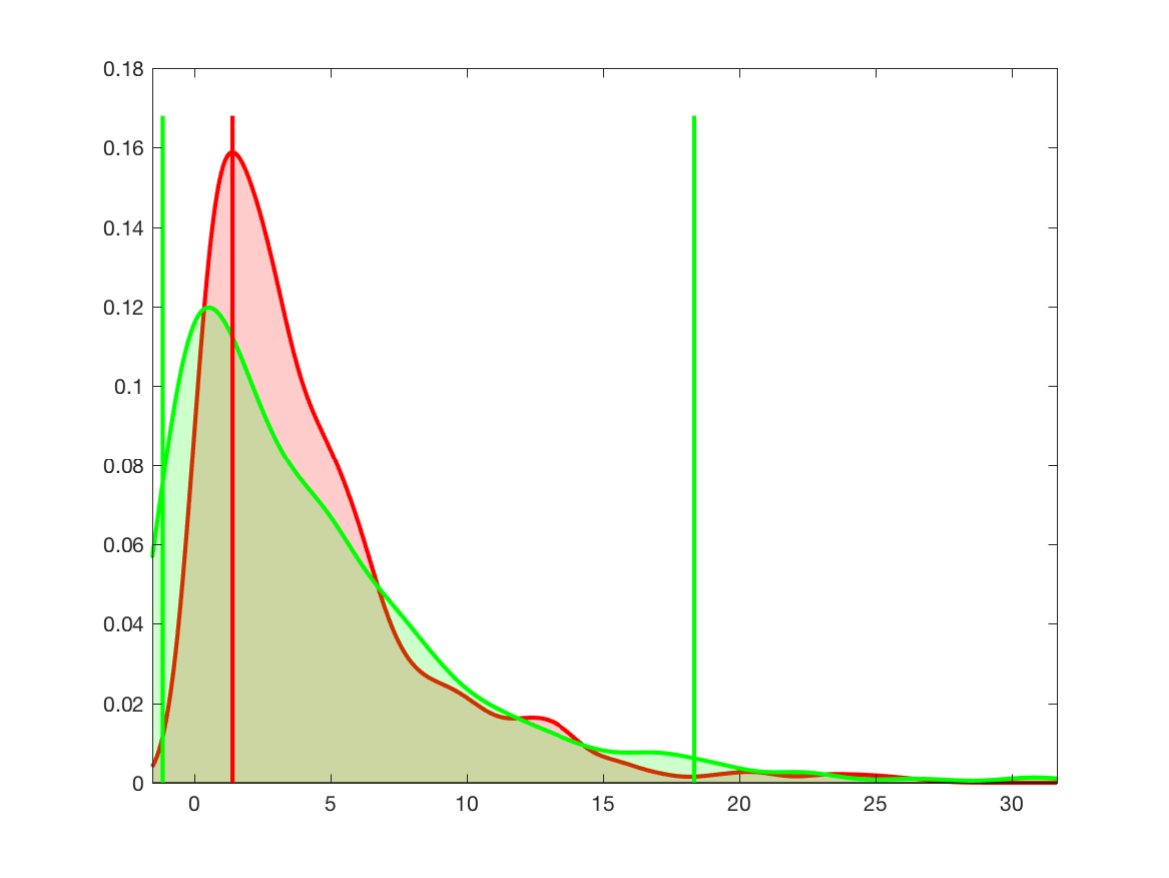} &
\includegraphics[width=0.15\textwidth,height=0.22\textwidth]{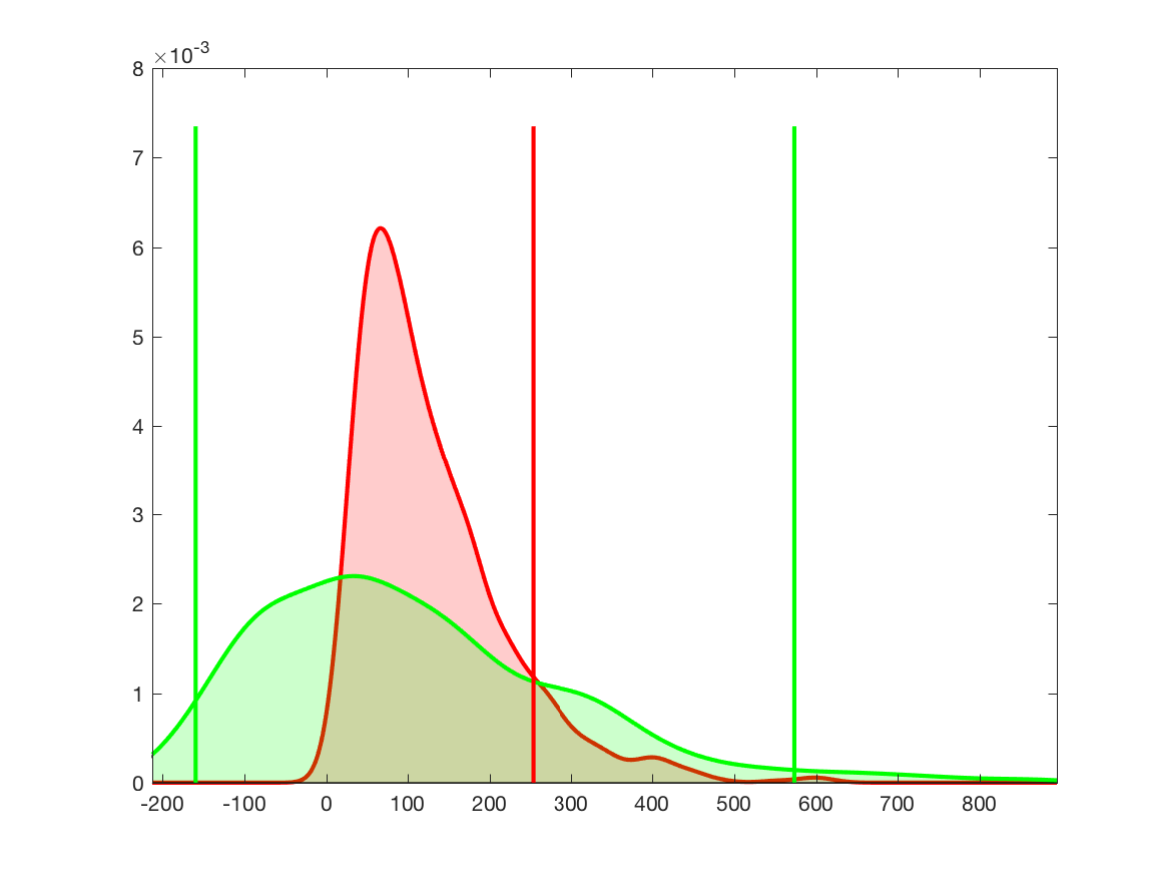} &
\includegraphics[width=0.15\textwidth,height=0.22\textwidth]{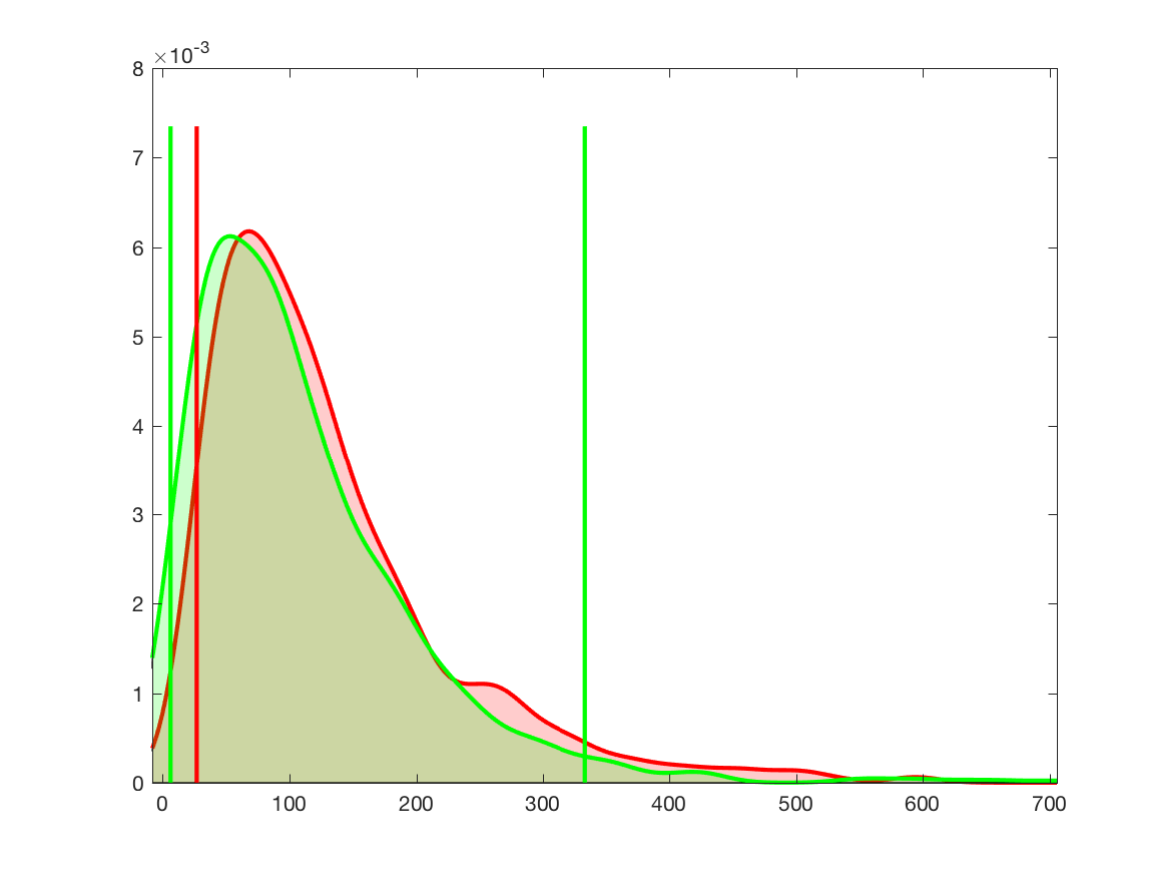}\\
&\multicolumn{3}{c}{Grid $5\times 5$}&\multicolumn{2}{c}{Grid $20\times 20$}
\end{tabular}
\caption{\label{fig:H0two}Case $a=b$ with two samples. Illustration of the bootstrap with $\varepsilon=1, 10, 100$ and two grids of size $5\times 5$ (left) and $20\times 20$ (right) to approximate the non-asymptotic distribution of the empirical Sinkhorn loss. Densities in red represent the distribution of $\rho_{n,m}^2\WB_{p,\varepsilon}^p(\hat{a}_n,\hat{b}_m)$. The green density represents the distribution of the random variable $\rho_{n,m}^2(\WB_{p,\varepsilon}^p(\hat{a}^{\ast}_n,\hat{b}_m^{\ast})-\WB_{p,\varepsilon}^p(\hat{a}_n,\hat{b}_m)-\langle\ualpha^{\hat{a}_n,\hat{b}_m},\hat{a}_n^{\ast}-\hat{a}_n\rangle-\langle\vbeta^{\hat{a}_n,\hat{b}_m},\hat{b}_m^{\ast}-\hat{b}_m\rangle)$ in \eqref{eq:boot_null_loss_two}.}
\end{figure}

\subsection{Estimation of test power using the bootstrap}
\label{subsec:test_power}
\paragraph{One sample - distribution with linear trend and varying slope parameter.}
The consistency and usefulness of the bootstrap procedure is illustrated by studying the statistical power (that is $\P(\mbox{Reject } H_0 | H_1 \mbox{ is true})$) of  statistical tests (at level $5 \%$)  based on the empirical Sinkhorn loss. For this purpose, we choose $a$ to be uniform and $b$ to be a distribution with linear trend whose slope parameter $\theta$ is ranging from $0$ to $0.1$ on a $5\times 5$ grid. We assume that we observe a single realization of an empirical measure $\hat{b}_m$ sampled from $b$ with $m=10^3$. Then, we generate $M=10^3$ bootstrap samples of random measures $\hat{b}^{\ast}_{m,j}$  from $\hat{b}_m$ (with $1 \leq j \leq M$), which allows the computation of the $p$-value
$$\mbox{$p$-value}= \# \{ j \mbox{ such that } n|\WB_{p,\varepsilon}^p(a,\hat{b}^{\ast}_{m,j})-\WB_{p,\varepsilon}^p(a,\hat{b}_{m})-\langle\vbeta^{a,\hat{b}_m},\hat{b}_{m,j}^{\ast}-\hat{b}_m\rangle |\geq n\WB_{p,\varepsilon}^p(a,\hat{b}_m)\} / M.$$
This experiments is repeated 100 times, in order to estimate the power (at level $\ualpha$)  of a test based on $n\WB_{p,\varepsilon}^p(a,\hat{b}_m)$ by comparing the resulting sequence of $p$-values to the value $\ualpha$. The results are reported in Figure \ref{fig:power} (left).
\renewcommand{\sidecap}[1]{ {\begin{sideways}\parbox{0.5\textwidth}{\centering #1}\end{sideways}} }
\begin{figure}[ht!]
\begin{center}
\begin{tabular}{cccc}
& $H(T\vert a\otimes b)$ = Relative entropy & & $H(T)$ = Entropy\vspace{-1.3cm}\\
\sidecap{Test power}&\includegraphics[width=0.4 \textwidth,height=0.4\textwidth]{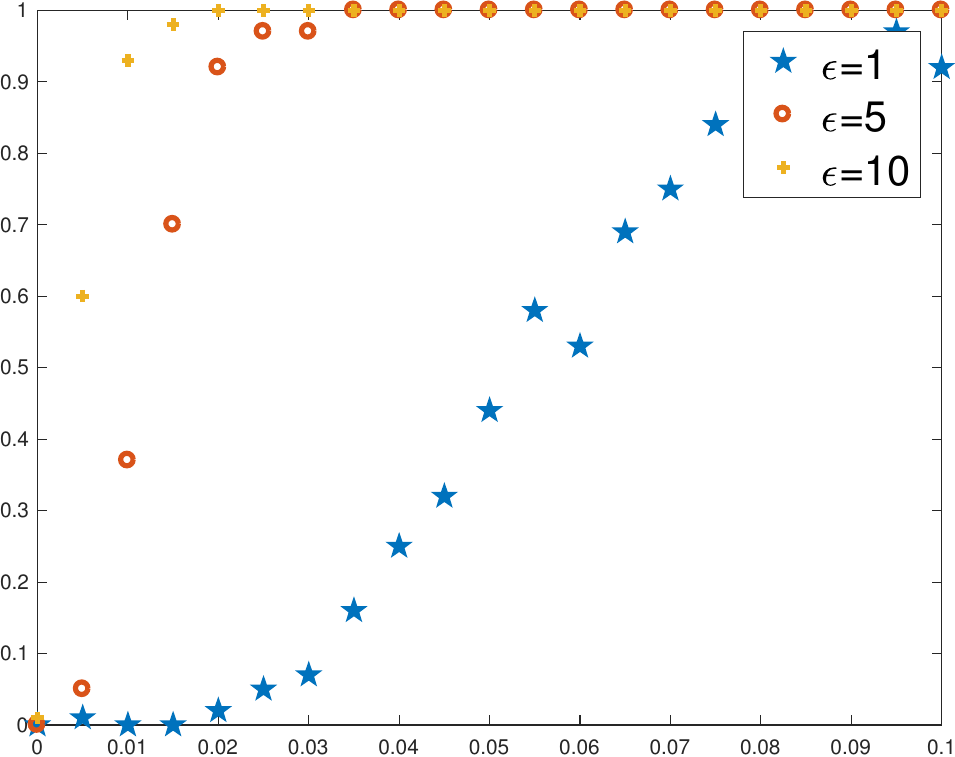}&
\sidecap{Test power}&\includegraphics[width=0.4 \textwidth,height=0.4\textwidth]{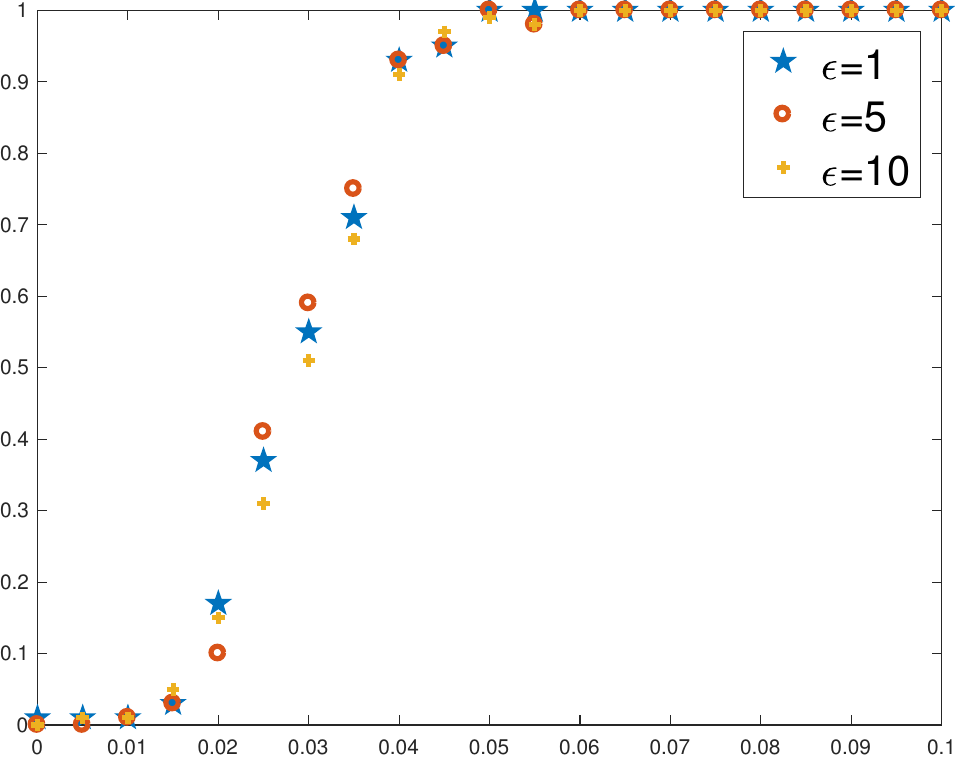}\\
&$\theta$ & & $\theta$
\end{tabular}
\caption{\label{fig:power}Test power (probability of rejecting $H_0$ knowing that $H_1$ is true) on a $5\times 5$ grid in the one sample case, as a function of the slope parameter $\theta$  ranging from $0$ to $0.15$ for $\varepsilon=1$ (blue), $\varepsilon=5$ (orange) and $\varepsilon=10$ (yellow), with $n = 10^3$. (left) $H(T\vert a\otimes b)$ = Relative entropy, (right) $H(T)$ = Entropy.}
\end{center}
\end{figure}
It can be seen that the resulting testing procedures are good discriminants for the three values of the regularization parameters $\varepsilon$ that we considered. As soon as the slope $\theta$ increases then $b$ sufficiently differs from $a$, and the probability of rejecting $H_0$ thus increases. We have also chosen to report results  obtained  with the Sinkhorn loss corresponding to  optimal transport regularized by the  entropy $H(T) = \sum_{ij}t_{ij}\log(t_{ij})$ instead of the relative entropy $H(T|a\otimes b)=\sum_{i,j}\log\left(\frac{t_{ij}}{a_ib_j}\right)t_{ij}$ (see Figure \ref{fig:power} (right)). Indeed, we remark that in the case of the relative entropy, the power of the test seems to highly depend on the value of $\varepsilon$. More precisely, for a fixed value of the slope parameter $\theta$ (or distribution $b$), the test power is larger as $\varepsilon$ increases. On the other hand, when using the Sinkhorn loss computed  with the entropy, the power of the test seems to be the same for any value of $\varepsilon$.

\begin{rmq}
The truly interesting property of the Sinkhorn loss over the Sinkhorn divergence is that in theory, for any  $\varepsilon>0$, we will obtain a steady $\varepsilon$-dependent asymptotic distribution, and that any regularization allows us to perform test statistics. In practice, more regularization leads to a blending of information. More precisely, the entropy will spread the mass of the distributions, and in some points of the grid, the differences of masses between the two distributions can be the result of regularization. On the other hand, when very few observations are available, and that measures are sparsely distributed on the grid, a large $\varepsilon$ will still allow to perform a statistical study.
\end{rmq}

\section{Analysis of real data}\label{sec:real}

We consider a dataset of colored images representing landscapes and foliage taken during Autumn (20 images) and Winter (17 images), see Figure \ref{fig:Dataseason} for examples. These images, provided by \cite{olmos2004biologically},  are available at \url{http://tabby.vision.mcgill.ca/html/welcome.html}. %Citation asked by the authors.

\begin{figure}[ht!]
\begin{center}
\begin{tabular}{cccc}
\includegraphics[width=0.215\textwidth,height=0.215\textwidth]{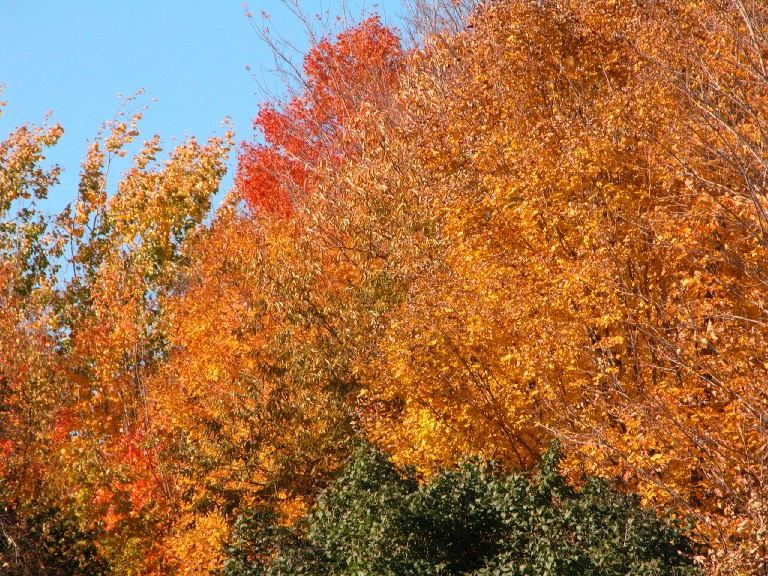} &
\includegraphics[width=0.215\textwidth,height=0.215\textwidth]{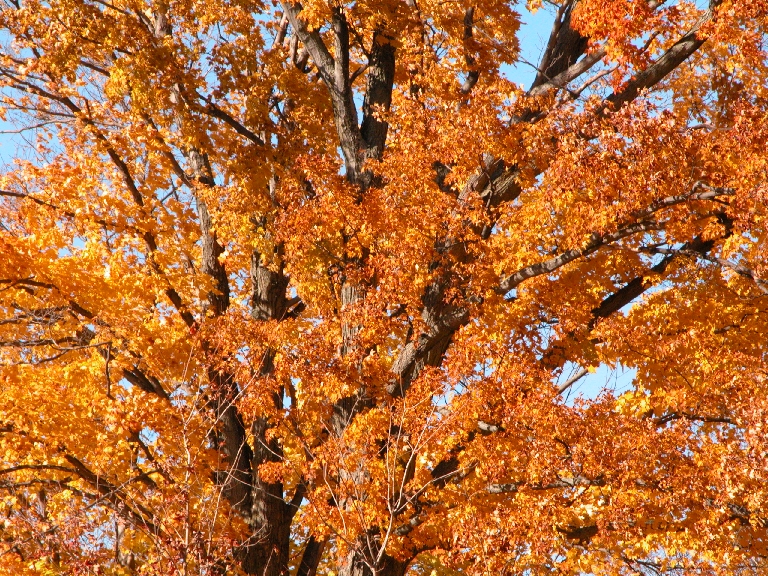} & 
\includegraphics[width=0.215\textwidth,height=0.215\textwidth]{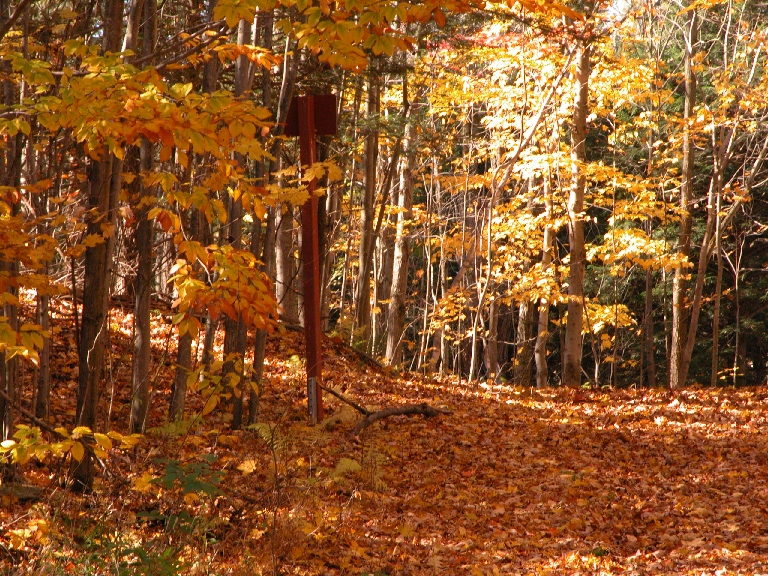} & 
\includegraphics[width=0.215\textwidth,height=0.215\textwidth]{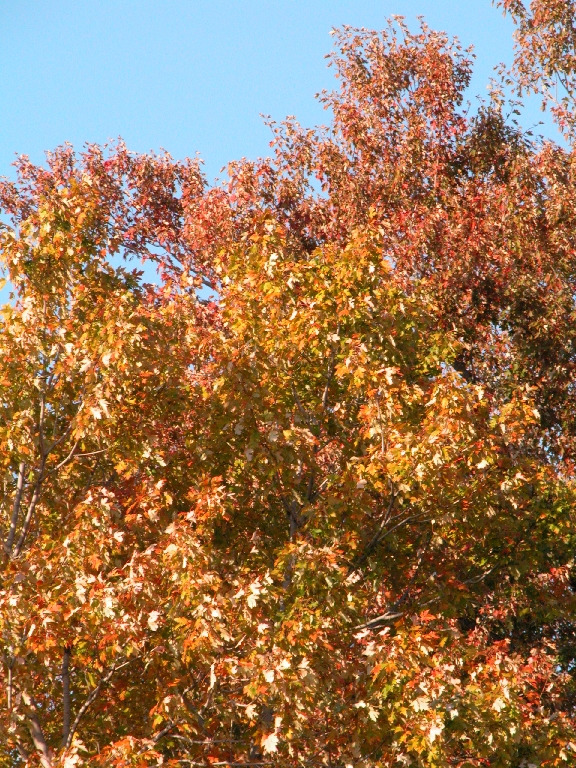} \\
\includegraphics[width=0.215\textwidth,height=0.215\textwidth]{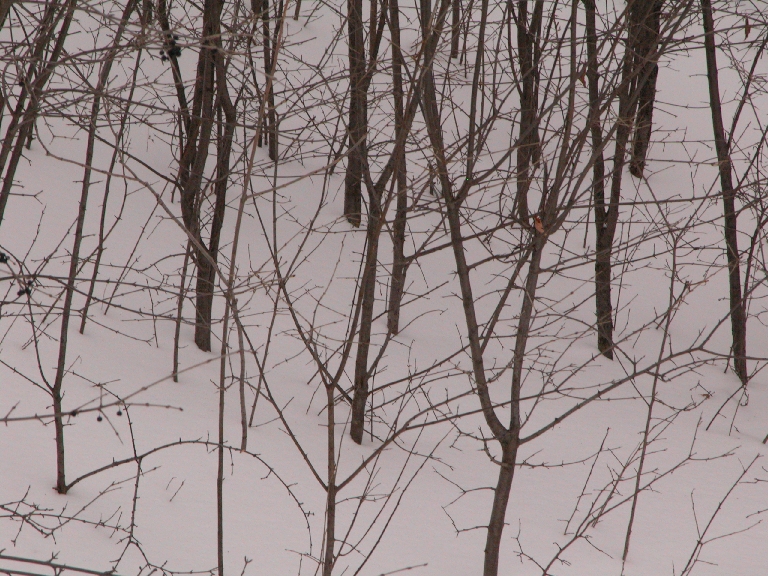} &
\includegraphics[width=0.215\textwidth,height=0.215\textwidth]{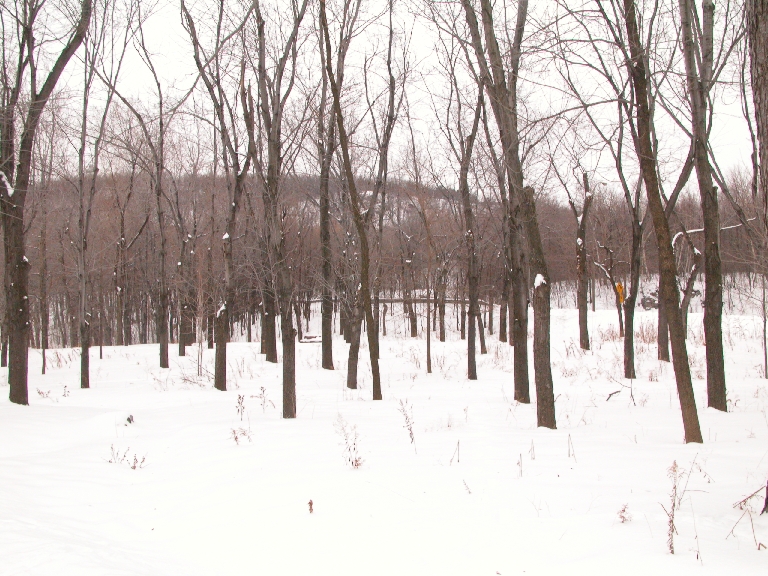} & 
\includegraphics[width=0.215\textwidth,height=0.215\textwidth]{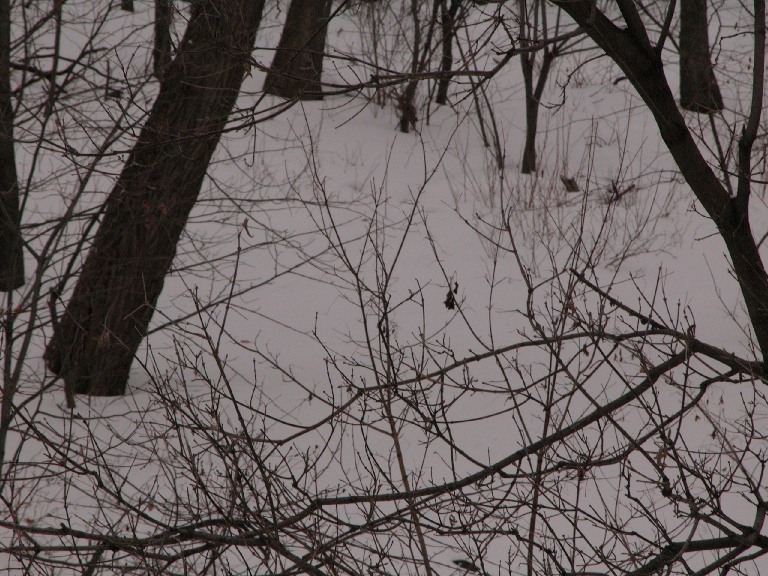} & 
\includegraphics[width=0.215\textwidth,height=0.215\textwidth]{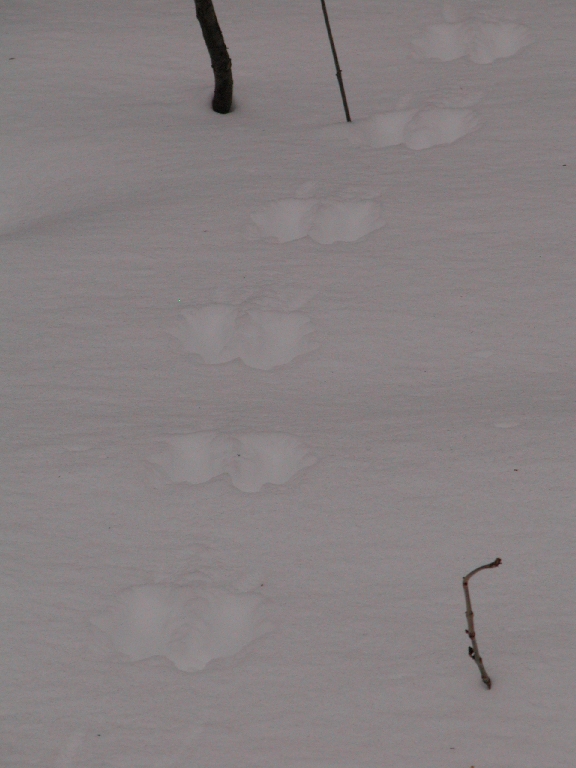}
\end{tabular}
\end{center}
\caption{\label{fig:Dataseason}Samples of $768\times 576$ colored images from autumn (first row) and winter (second row).}
\end{figure}

Each  image is transformed into a color histogram  on a three-dimensional grid (RGB colors) of size $N^3=16^3=4096$ of equi-spaced points. We will denote by $a_1,\ldots,a_{20}$ the autumn histograms and $w_1,\ldots,w_{17}$ the winter histograms. To compute the cost matrix $C$, we again use the squared  Euclidean distance between the spatial integer locations $x_i \in [0; 255] ^{3}$.

\subsection{Testing the hypothesis of equal color distribution between seasons}\label{subsec:H0AW}

We first test the null hypothesis that the  color distribution of the images in Autumn is the same as the  color distribution of the images in Winter. To this end, we consider the mean histogram of the dataset for each season, that we denote
$$\bar{a}_{20}=\frac{1}{20}\sum_{k=1}^{20} a_k \qquad \mbox{and} \qquad \bar{w}_{17}=\frac{1}{17}\sum_{k=1}^{17} w_k.$$
Notice that both $\bar{a}_{20}$ and $\bar{w}_{17}$ are discrete empirical measures admitting a zero mass for many locations $x_i$.\\

 We use the two samples testing procedure described previously, and a bootstrap approach to estimate the distribution of the test statistics
$$
\rho_{n,m}^2\WB_{p,\varepsilon}^p(\bar{a}_{20},\bar{w}_{17}).
$$
Notice also that $n$ and $m$ respectively correspond to the number of observations  for the empirical Autumn distribution $\bar{a}_{20}$ and the empirical Winter distribution $\bar{w}_{17}$, which is the total number of pixels times the number of images. Therefore, $n=20*768*576=8 847 360$ and $m=17*768*576=7 520 256$.
We report the results of the testing procedure for $\varepsilon = 10, 100$ by displaying in Figure \ref{fig:autumnVSwinter} an estimation of $M=100$ observations of the bootstrap statistic's density

$$\rho_{n,m}^2\left(\WB_{p,\varepsilon}^p(\hat{a}^{\ast}_n,\hat{w}^{\ast}_m)-\WB_{p,\varepsilon}^p(\bar{a}_{20},\bar{w}_{17})-(
\langle\ualpha^{\bar{a}_{20},\bar{w}_{17}},\hat{a}_n^{\ast}-\bar{a}_{20}\rangle +
 \langle\vbeta^{\bar{a}_{20},\bar{w}_{17}},\hat{w}_m^{\ast}-\bar{w}_{17})\rangle\right),$$
 
where $\hat{a}^{\ast}_n$ and $\hat{w}^{\ast}_m$ are respectively bootstrap samples of  $\bar{a}_{20}$ and $\bar{w}_{17}$, and $(\ualpha^{\bar{a}_{20},\bar{w}_{17}},\vbeta^{\bar{a}_{20},\bar{w}_{17}})$ are the optimal dual variables associated to $(\bar{a}_{20},\bar{w}_{17})$ in  problem \eqref{dual}.

For $\varepsilon= 10, 100$, the value of $\rho_{n,m}^2(\WB_{p,\varepsilon}^p(\bar{a}_{20},\bar{w}_{17}))$ is outside the support of this density, and the null hypothesis that the color distributions of images taken during Autumn and Winter are  the same is thus rejected.
 In particular, the test statistic $\rho_{n,m}^2(\WB_{p,\varepsilon}^p(\bar{a}_{20},\bar{w}_{17}))$
is equal to $6.07\times 10^7$ for $\varepsilon=10$ and to  $5.03\times 10^7$ for $\varepsilon=100$. %This suggests that for larger values of the parameter $\varepsilon$, i.e. with more regularization, the Sinkhorn loss is less precise for discriminating distributions. %the Sinkhorn loss gives a better discriminant, as previously seen in the experiments of Subsection \ref{subsec:test_power}.

\begin{figure}[ht!]
\begin{center}
\begin{tabular}{cc} \includegraphics[width=0.3\textwidth,height=0.28\textwidth]{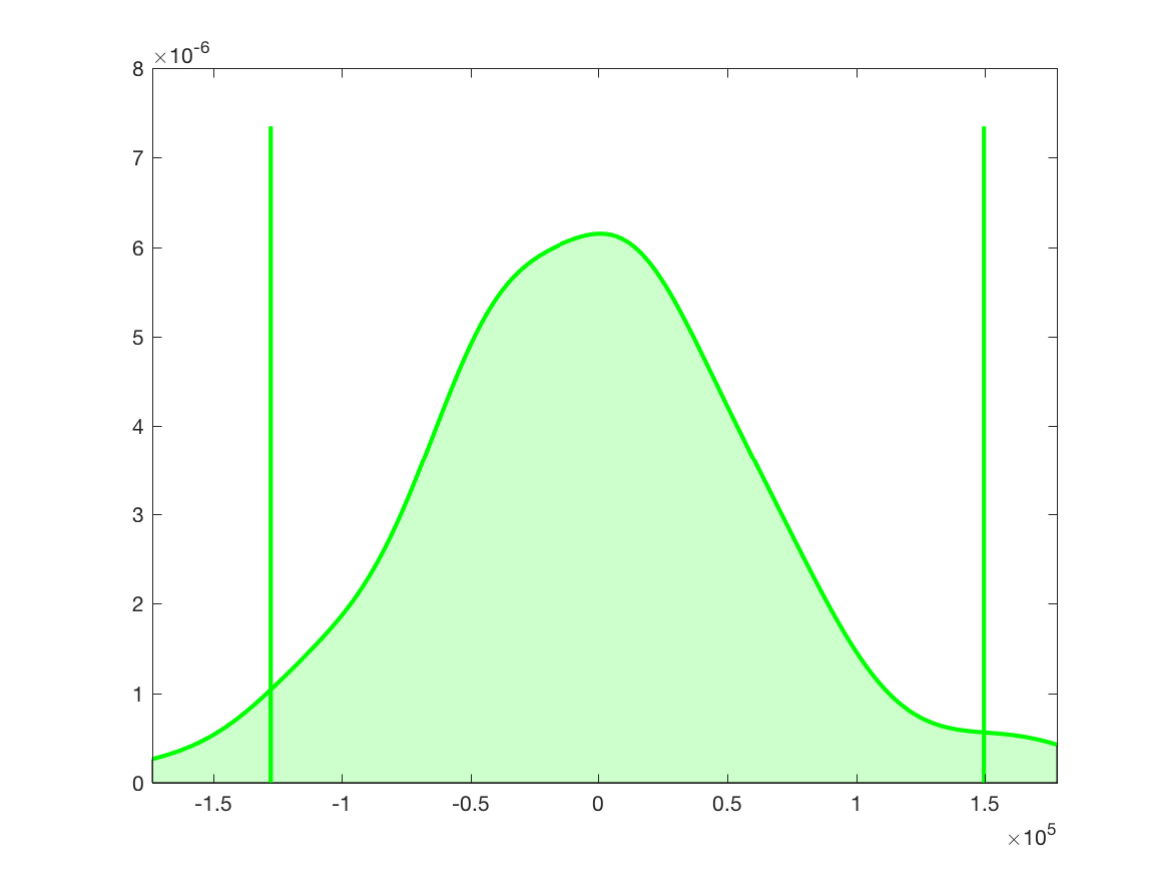}
&\includegraphics[width=0.3\textwidth,height=0.28\textwidth]{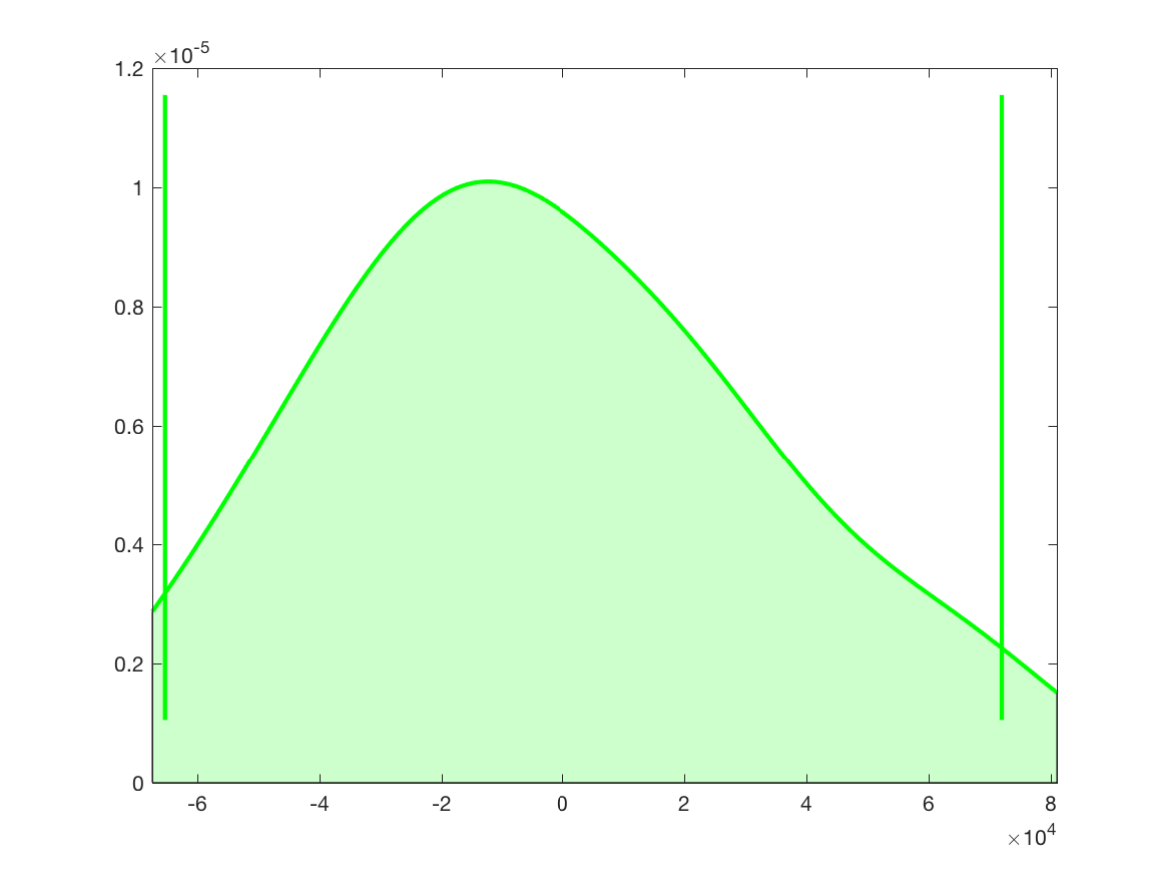}\\
(a) & (b)
\end{tabular}
\caption{\label{fig:autumnVSwinter}Testing equality of color distributions between Autumn and Winter for a grid of size $16^3=4096$. Green densities represent the distribution of the bootstrap statistics $\rho_{n,m}^2(\WB_{p,\varepsilon}^p(\hat{a}^{\ast}_n,\hat{w}^{\ast}_m)-\WB_{p,\varepsilon}^p(\bar{a}_{20},\bar{w}_{17})-(
\langle\ualpha^{\bar{a}_{20},\bar{w}_{17}},\hat{a}_n^{\ast}-\bar{a}_{20}\rangle +
 \langle\vbeta^{\bar{a}_{20},\bar{w}_{17}},\hat{w}_m^{\ast}-\bar{w}_{17}\rangle))$ (vertical bars represent a confidence interval of level $95\%$) for (a) $\varepsilon=10$ and (b) $\varepsilon=100$. The value of $\rho_{n,m}^2\WB_{p,\varepsilon}^p(\bar{a}_{20},\bar{w}_{17})$ is outside the support of the green density for each value of $\varepsilon$, and it is thus not represented.}
\end{center}
\end{figure}

We also run the exact same experiments for a smaller grid (size $8^3=512$) and a higher number of observations ($M=1000$). The results are displayed in Figure \ref{fig:autumnVSwinterGrid8}. The distributions  $\rho_{n,m}^2(\WB_{p,\varepsilon}^p(\hat{a}^{\ast}_n,\hat{w}^{\ast}_m)-\WB_{p,\varepsilon}^p(\bar{a}_{20},\bar{w}_{17})-(
\langle\ualpha^{\bar{a}_{20},\bar{w}_{17}},\hat{a}_n^{\ast}-\bar{a}_{20}\rangle +
 \langle\vbeta^{\bar{a}_{20},\bar{w}_{17}},\hat{w}_m^{\ast}-\bar{w}_{17}\rangle))$ are much more centered around $0$ (we gain a factor $10$). However, we obtain the same conclusion as before, with a test statistic equal to $9.39\times 10^6$ for $\varepsilon=10$ and $8.50\times 10^6$ for $\varepsilon=100$. %, which gives smaller values since the grid is smaller (and thus, it is less precise).
 
\begin{figure}[ht!]
\begin{center}
\begin{tabular}{cc} \includegraphics[width=0.3\textwidth,height=0.28\textwidth]{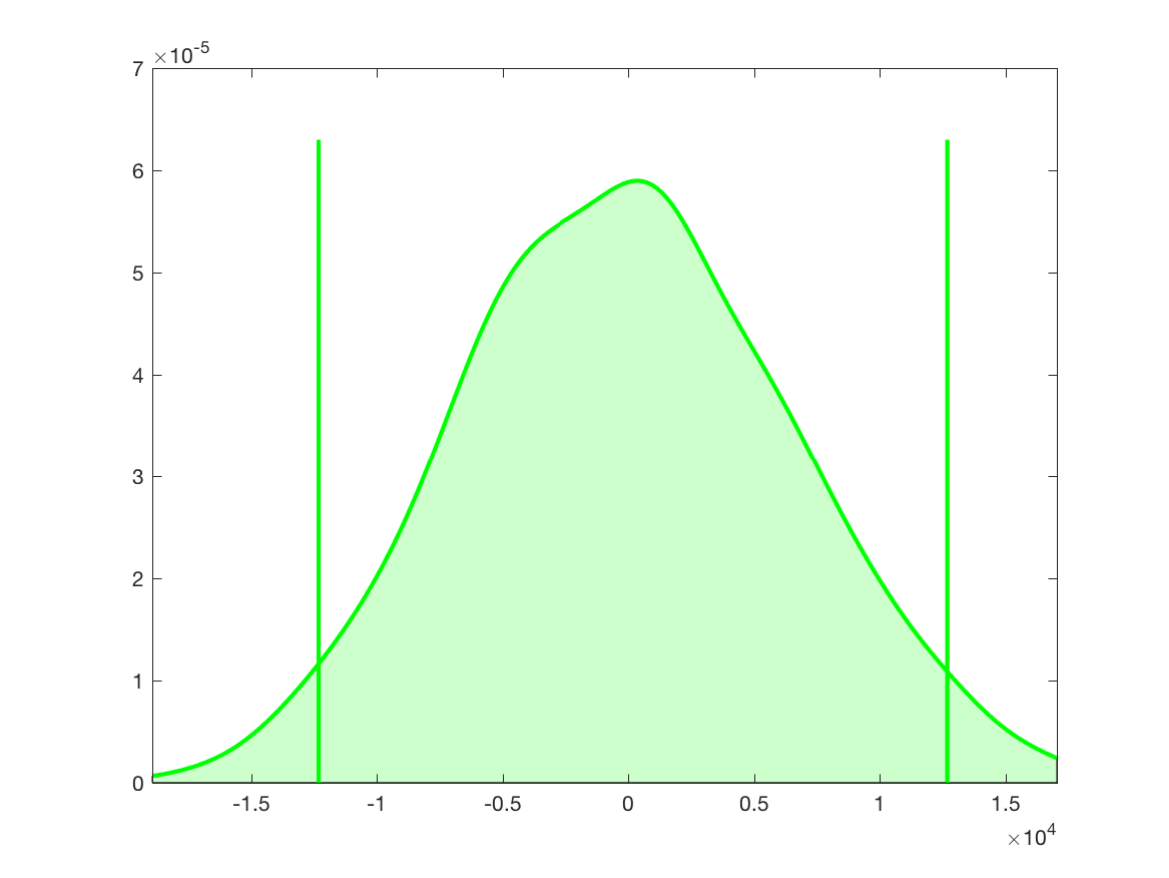}
&\includegraphics[width=0.3\textwidth,height=0.28\textwidth]{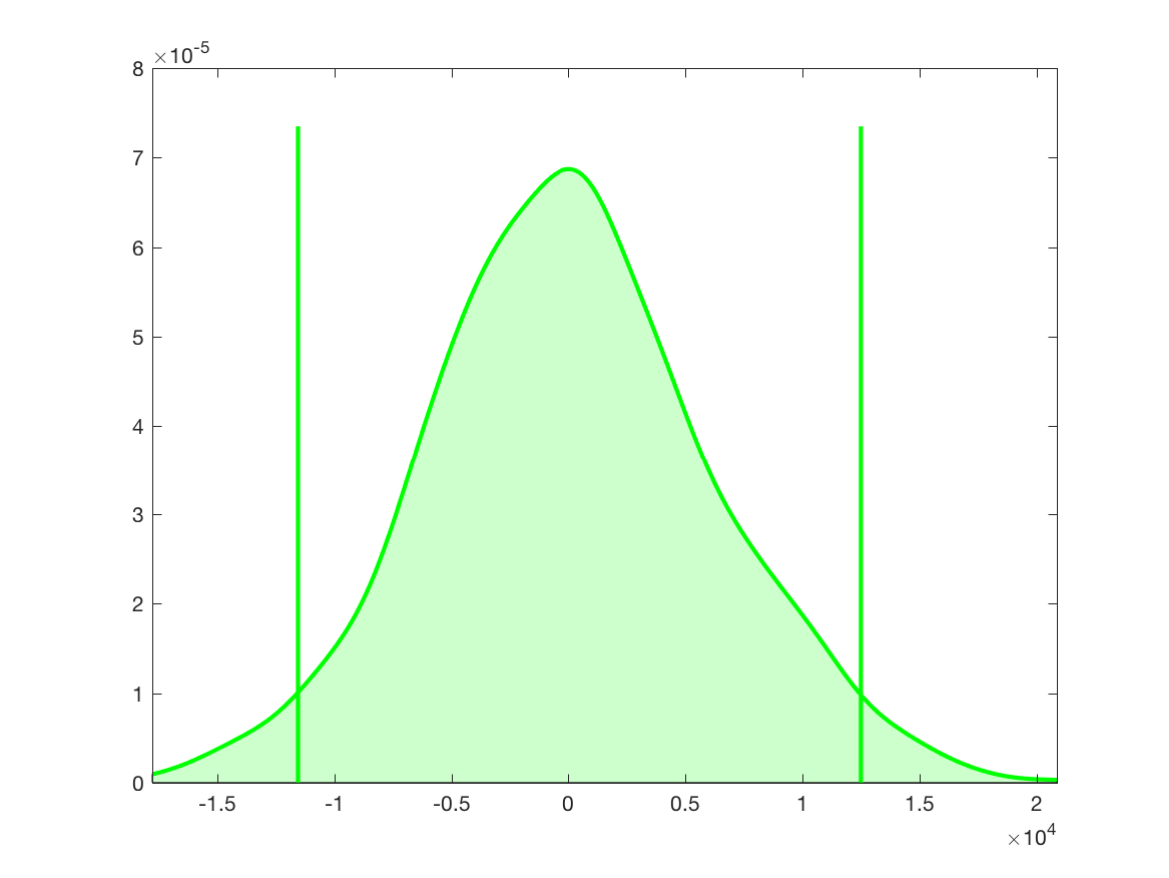}\\
(a) & (b)
\end{tabular}
\caption{\label{fig:autumnVSwinterGrid8}Testing equality of color distributions between Autumn and Winter for a grid of size $8^3=512$. Green densities represent the distribution of the bootstrap statistics $\rho_{n,m}^2(\WB_{p,\varepsilon}^p(\hat{a}^{\ast}_n,\hat{w}^{\ast}_m)-\WB_{p,\varepsilon}^p(\bar{a}_{20},\bar{w}_{17})-(
\langle\ualpha^{\bar{a}_{20},\bar{w}_{17}},\hat{a}_n^{\ast}-\bar{a}_{20}\rangle +
 \langle\vbeta^{\bar{a}_{20},\bar{w}_{17}},\hat{w}_m^{\ast}-\bar{w}_{17}\rangle))$ (vertical bars represent a confidence interval of level $95\%$) for (a) $\varepsilon=10$ and (b) $\varepsilon=100$. The value of $\rho_{n,m}^2\WB_{p,\varepsilon}^p(\bar{a}_{20},\bar{w}_{17})$ is outside the support of the green density for each value of $\varepsilon$, and it is thus not represented.}
\end{center}
\end{figure}

\subsection{Testing the hypothesis of equal distribution when splitting the Autumn dataset}

We propose now to investigate the equality of distributions within the same dataset of Autumn histograms. To this end, we arbitrarily split the Autumn dataset into two subsets of 10 images and we compute their mean distribution
$$
\bar{a}_{1\to 10} = \frac{1}{10} \sum_{k=1}^{10} a_{k} \qquad \mbox{and} \qquad \bar{a}_{11\to 20} = \frac{1}{10} \sum_{k=11}^{20} a_{k},
$$
for which $n=m=10*768*576=4 423 680$.
The procedure is then similar to the Autumn versus Winter case in Subsection \ref{subsec:H0AW}, meaning that we sample $M=100$ bootstrap distributions $\hat{a}^{\ast}_n$ and $\hat{b}^{\ast}_m$ from respectively $\bar{a}_{1\to 10}$ and $\bar{a}_{11\to 20}$. The results are displayed in Figure \ref{fig:autumnVSautumn}. We obtain similar results than in the two seasons case, the null hypothesis that both Autumn distributions follow the same law is thus rejected. On the other hand, the test statistics are smaller in this case as the histogram of color seems to be closer. Indeed, the quantity $\rho_{n,m}^2\WB_{p,\varepsilon}^p(\bar{a}_{1\to 10},\bar{a}_{11\to 20})$ is equal to $11.02\times 10^6$ for $\varepsilon=10$ and $5.50\times 10^6$ for $\varepsilon=100$. %Again, a smaller value of $\varepsilon$ gives a better discriminant.
\begin{figure}[ht!]
\begin{center}
\begin{tabular}{cc}
\includegraphics[width=0.3\textwidth,height=0.28\textwidth]{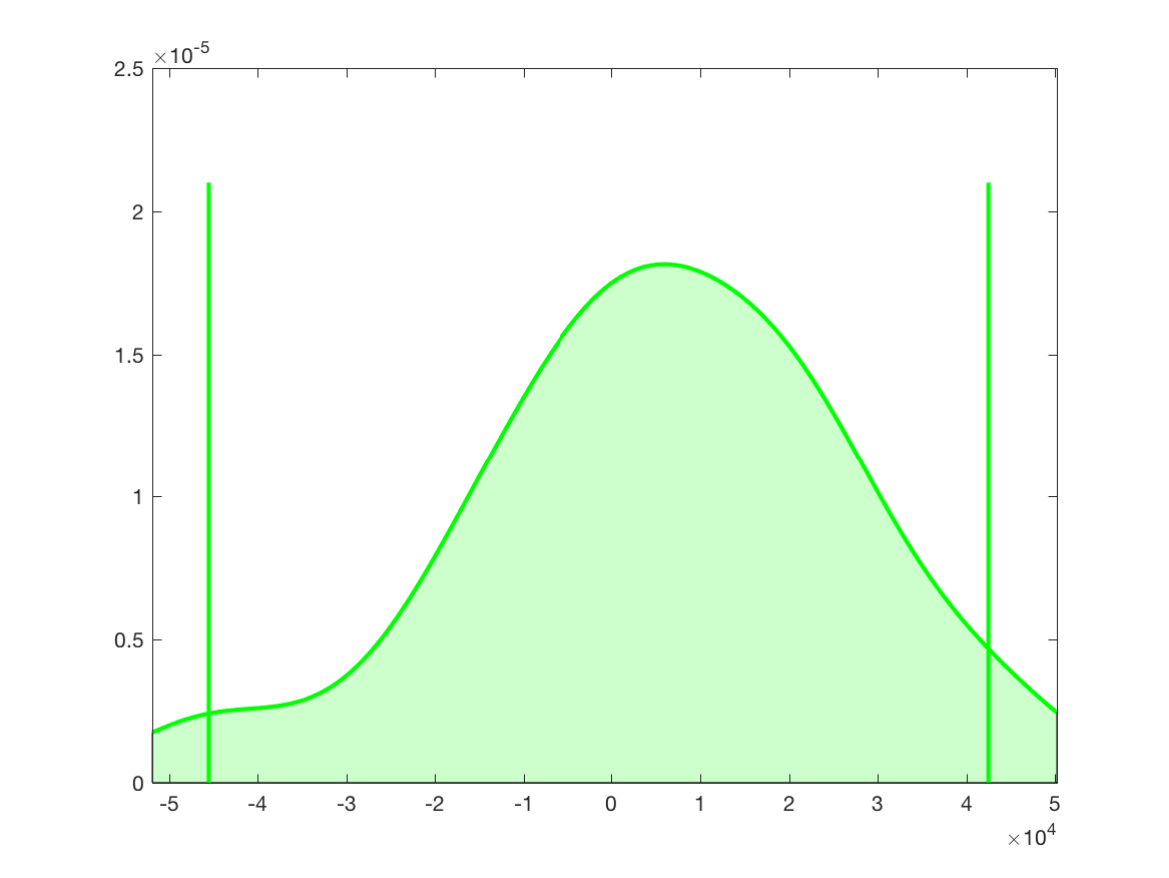}
& \includegraphics[width=0.3\textwidth,height=0.28\textwidth] {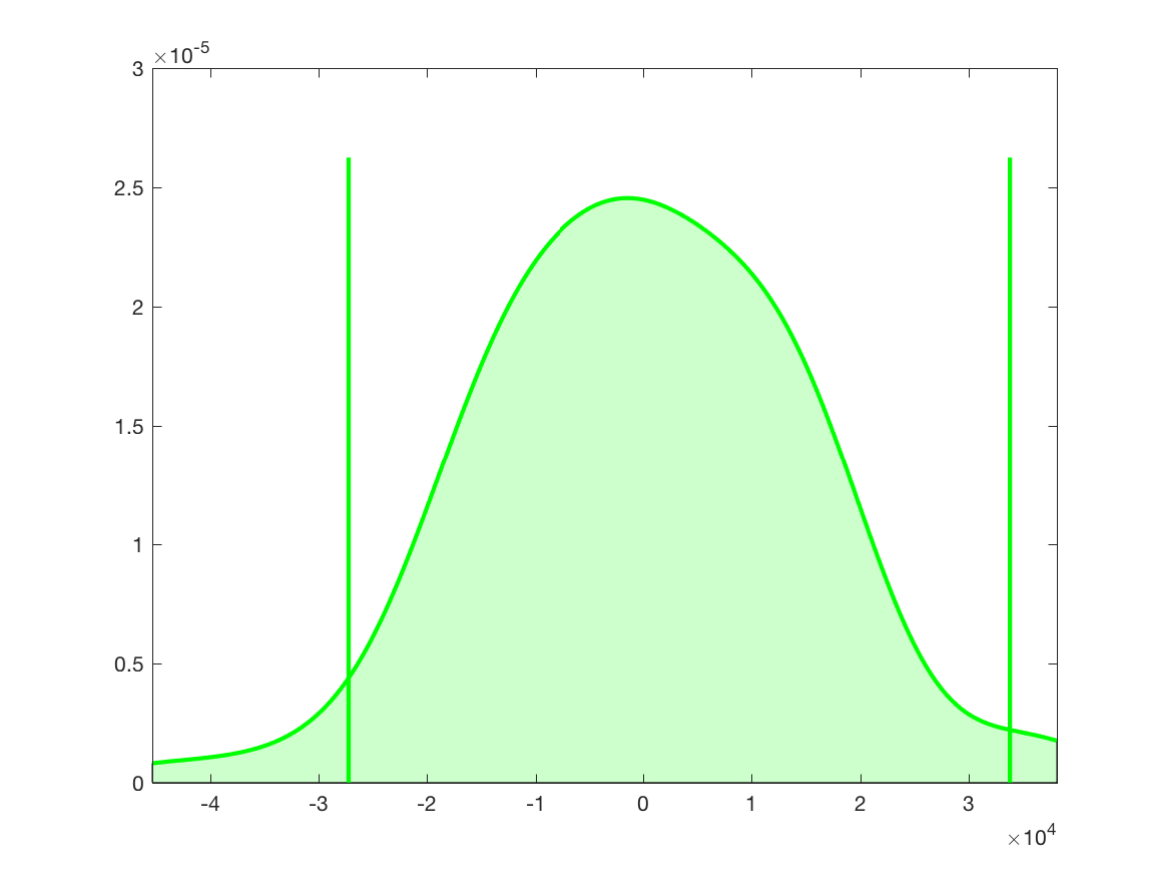}\\
(a) & (b)
\end{tabular}
\caption{\label{fig:autumnVSautumn}Testing equality of color distributions when splitting the autumn dataset into two for a grid of size $16^3=512$. Green densities represent the distribution of the bootstrap statistics $\rho_{n,m}^2(\WB_{p,\varepsilon}^p(\hat{a}^{\ast}_n,\hat{b}^{\ast}_m)-\WB_{p,\varepsilon}^p(\bar{a}_{1\to 10},\bar{a}_{11\to 20})-(
\langle\ualpha^{\bar{a}_{1\to 11},\bar{a}_{11\to 20}},\hat{a}_n^{\ast}-\bar{a}_{1\to 11}\rangle +
 \langle\vbeta^{\bar{a}_{1\to 11},\bar{a}_{11\to 20}},\hat{b}_m^{\ast}-\bar{a}_{11\to 20}\rangle))$ (vertical bars represent a confidence interval of level $95\%$) for (a) $\varepsilon=10$ and (b) $\varepsilon=100$. The value of $\rho_{n,m}^2\WB_{p,\varepsilon}^p(\bar{a}_{1\to 10},\bar{a}_{11\to 20})$ is outside the support of the green density for each value of $\varepsilon$, and it is thus not represented.}.
\end{center}
\end{figure}

Similarly to the Winter VS Autumn case, we also run the same Autumn VS Autumn experiments for a grid of size $8^3=512$ and $M=1000$ observations. The results are displayed in Figure \ref{fig:autumnVSautumnGrid8} for test statistics equal to $14.07\times 10^5$ for $\varepsilon=10$ and $3.41\times 10^5$ for $\varepsilon=100$.

\begin{figure}[ht!]
\begin{center}
\begin{tabular}{cc}
\includegraphics[width=0.3\textwidth,height=0.28\textwidth]{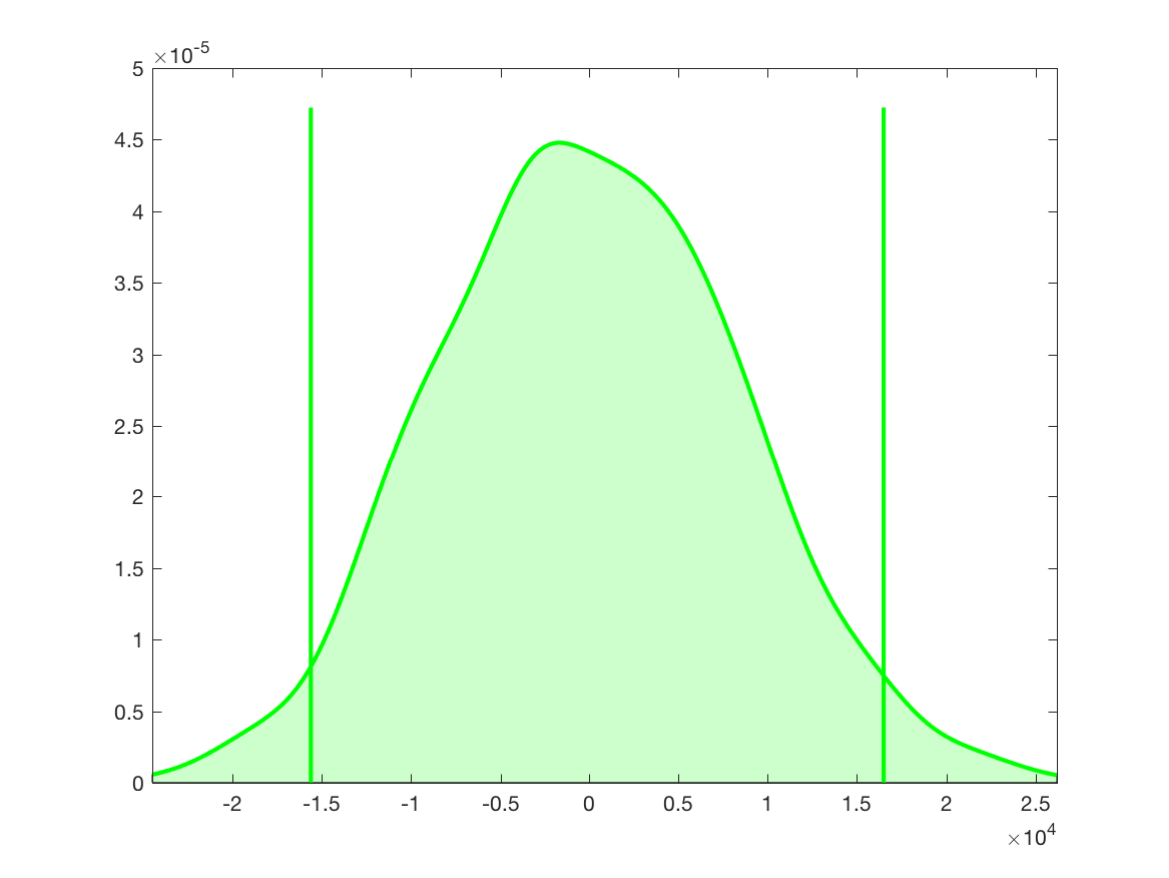}
& \includegraphics[width=0.3\textwidth,height=0.28\textwidth] {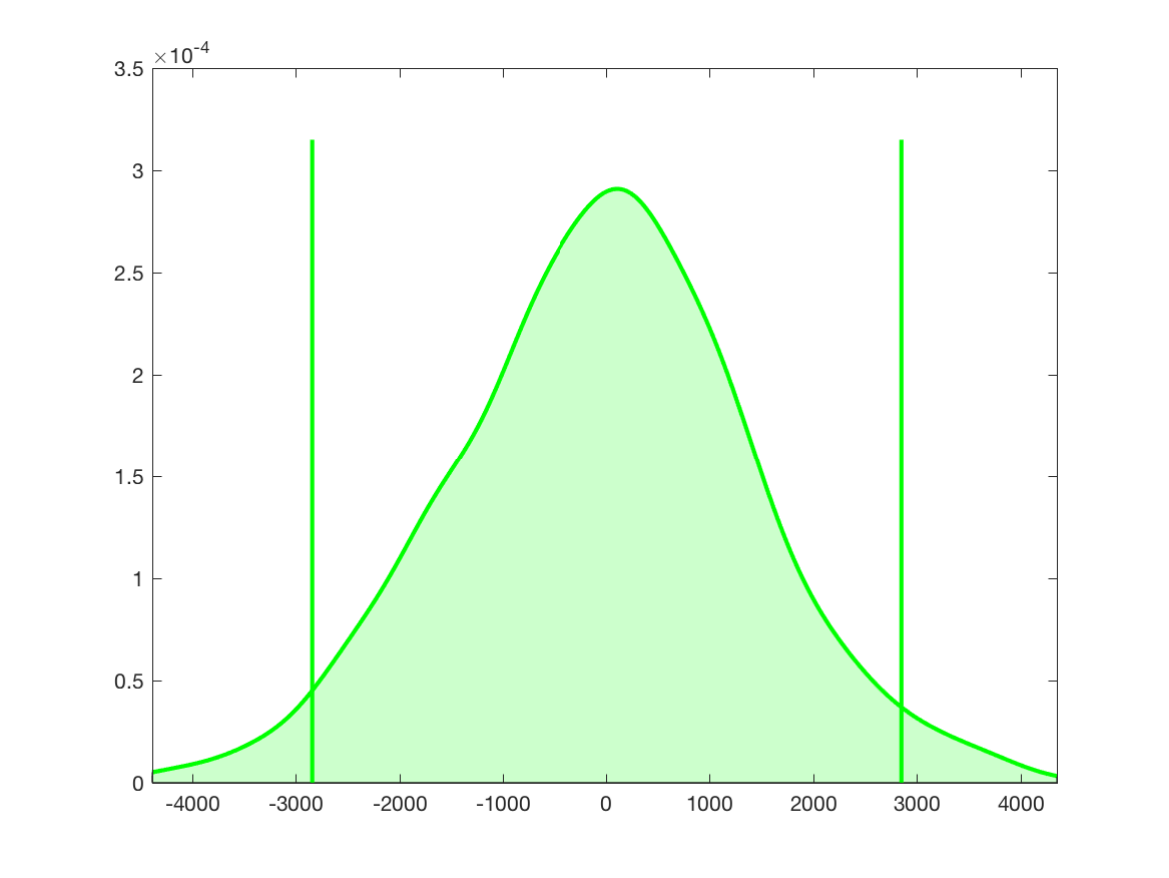}\\
(a) & (b)
\end{tabular}
\caption{\label{fig:autumnVSautumnGrid8}Testing equality of color distributions when splitting the autumn dataset into two for a grid of size $8^3=512$. Green densities represent the distribution of the bootstrap statistics $\rho_{n,m}^2(\WB_{p,\varepsilon}^p(\hat{a}^{\ast}_n,\hat{b}^{\ast}_m)-\WB_{p,\varepsilon}^p(\bar{a}_{1\to 10},\bar{a}_{11\to 20})-(
\langle\ualpha^{\bar{a}_{1\to 11},\bar{a}_{11\to 20}},\hat{a}_n^{\ast}-\bar{a}_{1\to 11}\rangle +
 \langle\vbeta^{\bar{a}_{1\to 11},\bar{a}_{11\to 20}},\hat{b}_m^{\ast}-\bar{a}_{11\to 20}\rangle))$ (vertical bars represent a confidence interval of level $95\%$) for (a) $\varepsilon=10$ and (b) $\varepsilon=100$. The value of $\rho_{n,m}^2\WB_{p,\varepsilon}^p(\bar{a}_{1\to 10},\bar{a}_{11\to 20})$ is outside the support of the green density for each value of $\varepsilon$, and it is thus not represented.}
\end{center}
\end{figure}

\begin{rmq}
For comparison purpose, we ran a $\chi^2$ test of homogeneity for testing the hypothesis of equal distributions of colors. The obtained test statistic in the \textit{Autumn vs Winter} case is equal to $\chi^2_{AW} = 6.96 \times 10^4$, and in the \textit{Autumn splitting} case to  $\chi^2_{AA} = 4.06 \times 10^4$. Even if $\chi^2_{AA}$ is indeed smaller than $\chi^2_{AW}$, the contrast between these two is weaker than with the Sinkhorn loss test.
\end{rmq}

\section{Future works} \label{sec:concl}

As remarked in \cite{sommerfeld2016inference}, there exists a vast literature for two-sample testing using univariate data. However, in a multivariate setting, it is difficult to consider that there exist standard methods to test the equality of two distributions. We thus intend to further investigate the benefits of the use of the empirical Sinkhorn loss to propose novel testing procedures able to compare multivariate distributions for real data analysis. A first perspective is to apply the methodology developed in this paper to more than two samples using the notion of smoothed Wasserstein barycenters (see e.g.\ \cite{CuturiPeyre} and references therein) for the analysis of variance  of multiple and multivariate random measures (MANOVA). However, as pointed out in \cite{CuturiPeyre}, a critical issue in this setting will be the choice of the regularization parameter $\varepsilon$, as it has a large influence on the shape of the estimated Wasserstein  barycenter.
Another interesting extension of the results presented in this paper would be to obtain the eigenvalues of the Hessian matrix of the Sinkhorn loss, in order to compute the distributional limit under the null hypothesis of equality of distributions.


\begin{thebibliography}{10}

\bibitem{arjovsky2017wasserstein}
M.~Arjovsky, S.~Chintala, and L.~Bottou.
\newblock {W}asserstein {GAN}.
\newblock {\em arXiv preprint arXiv:1701.07875}, 2017.

\bibitem{BCP18}
J.~Bigot, E.~Cazelles, and N.~Papadakis.
\newblock Data-driven regularization of wasserstein barycenters with an
  application to multivariate density registration.
\newblock {\em ArXiv e-prints}, 1804.08962, 2018.

\bibitem{bigot2017geodesic}
J.~Bigot, R.~Gouet, T.~Klein, A.~L{\'o}pez, et~al.
\newblock Geodesic pca in the {W}asserstein space by convex pca.
\newblock {\em Annales de l'Institut Henri Poincar{\'e}, Probabilit{\'e}s et
  Statistiques}, 53(1):1--26, 2017.

\bibitem{boyd2004convex}
S.~Boyd and L.~Vandenberghe.
\newblock {\em Convex optimization}.
\newblock Cambridge university press, 2004.

\bibitem{GPCA}
E.~Cazelles, V.~Seguy, J.~Bigot, M.~Cuturi, and N.~Papadakis.
\newblock Geodesic pca versus log-pca of histograms in the wasserstein space.
\newblock {\em SIAM Journal on Scientific Computing}, 40(2):B429--B456, 2018.

\bibitem{chen2018inference}
Q.~Chen and Z.~Fang.
\newblock Inference on functionals under first order degeneracy.
\newblock {\em SSRN}, 2018.

\bibitem{cuturi}
M.~Cuturi.
\newblock Sinkhorn distances: Lightspeed computation of optimal transport.
\newblock In {\em Advances in Neural Information Processing Systems 26}, pages
  2292--2300. 2013.

\bibitem{cuturi2013fast}
M.~Cuturi and A.~Doucet.
\newblock Fast computation of {W}asserstein barycenters.
\newblock In {\em International Conference on Machine Learning 2014, JMLR
  W\&CP}, volume~32, pages 685--693, 2014.

\bibitem{CuturiPeyre}
M.~Cuturi and G.~Peyr\'e.
\newblock A smoothed dual approach for variational {W}asserstein problems.
\newblock {\em SIAM Journal on Imaging Sciences}, 9(1):320--343, 2016.

\bibitem{MR1740113}
E.~del Barrio, J.~A. Cuesta-Albertos, C.~Matr\'an, and J.~M.
  Rodriguez-Rodriguez.
\newblock Tests of goodness of fit based on the {$L_2$}-{W}asserstein distance.
\newblock {\em Ann. Statist.}, 27(4):1230--1239, 1999.

\bibitem{MR2121458}
E.~del Barrio, E.~Gin\'e, and F.~Utzet.
\newblock Asymptotics for {$L_2$} functionals of the empirical quantile
  process, with applications to tests of fit based on weighted {W}asserstein
  distances.
\newblock {\em Bernoulli}, 11(1):131--189, 2005.

\bibitem{delBarrio2017}
E.~del Barrio and J.-M. Loubes.
\newblock Central limit theorems for empirical transportation cost in general
  dimension.
\newblock {\em arXiv:1705.01299v1}, 2017.

\bibitem{efron93bootstrap}
B.~Efron and R.~J. Tibshirani.
\newblock {\em {An Introduction to the Bootstrap}}.
\newblock Chapman \& Hall, New York, 1993.

\bibitem{feydy2018interpolating}
J.~Feydy, T.~S{\'e}journ{\'e}, F.-X. Vialard, S.-I. Amari, A.~Trouv{\'e}, and
  G.~Peyr{\'e}.
\newblock Interpolating between optimal transport and {MMD} using sinkhorn
  divergences.
\newblock {\em arXiv preprint arXiv:1810.08278}, 2018.

\bibitem{MR2161214}
G.~Freitag and A.~Munk.
\newblock On {H}adamard differentiability in {$k$}-sample semiparametric
  models---with applications to the assessment of structural relationships.
\newblock {\em J. Multivariate Anal.}, 94(1):123--158, 2005.

\bibitem{frogner2015learning}
C.~Frogner, C.~Zhang, H.~Mobahi, M.~Araya, and T.~A. Poggio.
\newblock Learning with a {W}asserstein loss.
\newblock In {\em Advances in Neural Information Processing Systems}, pages
  2053--2061, 2015.

\bibitem{genevay2018sample}
A.~Genevay, L.~Chizat, F.~Bach, M.~Cuturi, and G.~Peyr{\'e}.
\newblock Sample complexity of sinkhorn divergences.
\newblock {\em arXiv preprint arXiv:1810.02733}, 2018.

\bibitem{2016-genevay-nips}
A.~Genevay, M.~Cuturi, G.~Peyr{\'e}, and F.~Bach.
\newblock Stochastic optimization for large-scale optimal transport.
\newblock In D.~D. Lee, U.~V. Luxburg, I.~Guyon, and R.~Garnett, editors, {\em
  Proc. NIPS'16}, pages 3432--3440. Curran Associates, Inc., 2016.

\bibitem{2017-Genevay-AutoDiff}
A.~Genevay, G.~Peyr{\'e}, and M.~Cuturi.
\newblock Sinkhorn-autodiff: Tractable {W}asserstein learning of generative
  models.
\newblock Preprint 1706.00292, Arxiv, 2017.

\bibitem{gramfort2015fast}
A.~Gramfort, G.~Peyr{\'e}, and M.~Cuturi.
\newblock Fast optimal transport averaging of neuroimaging data.
\newblock In {\em International Conference on Information Processing in Medical
  Imaging}, pages 261--272. Springer, 2015.

\bibitem{klatt2018empirical}
M.~Klatt, C.~Tameling, and A.~Munk.
\newblock Empirical regularized optimal transport: Statistical theory and
  applications.
\newblock {\em arXiv preprint arXiv:1810.09880}, 2018.

\bibitem{luise2018differential}
G.~Luise, A.~Rudi, M.~Pontil, and C.~Ciliberto.
\newblock Differential properties of {S}inkhorn approximation for learning with
  {W}asserstein distance.
\newblock In {\em Advances in Neural Information Processing Systems}, pages
  5859--5870, 2018.

\bibitem{nguyen2011wasserstein}
X.~Nguyen.
\newblock Convergence of latent mixing measures in finite and infinite mixture
  models.
\newblock {\em Ann. Statist.}, 41(1):370--400, 02 2013.

\bibitem{olmos2004biologically}
A.~Olmos and F.~A. Kingdom.
\newblock A biologically inspired algorithm for the recovery of shading and
  reflectance images.
\newblock {\em Perception}, 33(12):1463--1473, 2004.

\bibitem{rabin2015convex}
J.~Rabin and N.~Papadakis.
\newblock Convex color image segmentation with optimal transport distances.
\newblock In {\em International Conference on Scale Space and Variational
  Methods in Computer Vision}, pages 256--269. Springer, 2015.

\bibitem{ramdas2017wasserstein}
A.~Ramdas, N.~G. Trillos, and M.~Cuturi.
\newblock On wasserstein two-sample testing and related families of
  nonparametric tests.
\newblock {\em Entropy}, 19(2):47, 2017.

\bibitem{MR3545279}
T.~Rippl, A.~Munk, and A.~Sturm.
\newblock Limit laws of the empirical {W}asserstein distance: {G}aussian
  distributions.
\newblock {\em J. Multivariate Anal.}, 151:90--109, 2016.

\bibitem{rolet2016fast}
A.~Rolet, M.~Cuturi, and G.~Peyr{\'e}.
\newblock Fast dictionary learning with a smoothed {W}asserstein loss.
\newblock In {\em Proc. International Conference on Artificial Intelligence and
  Statistics (AISTATS)}, 2016.

\bibitem{schmitz2017wasserstein}
M.~A. Schmitz, M.~Heitz, N.~Bonneel, F.~M.~N. Mboula, D.~Coeurjolly, M.~Cuturi,
  G.~Peyr{\'e}, and J.-L. Starck.
\newblock Wasserstein dictionary learning: Optimal transport-based unsupervised
  non-linear dictionary learning.
\newblock {\em arXiv preprint arXiv:1708.01955}, 2017.

\bibitem{NIPS2015_5680}
V.~Seguy and M.~Cuturi.
\newblock Principal geodesic analysis for probability measures under the
  optimal transport metric.
\newblock In C.~Cortes, N.~Lawrence, D.~Lee, M.~Sugiyama, and R.~Garnett,
  editors, {\em Advances in Neural Information Processing Systems 28}, pages
  3294--3302. Curran Associates, Inc., 2015.

\bibitem{2015-solomon-siggraph}
J.~Solomon, F.~de~Goes, G.~Peyr{\'e}, M.~Cuturi, A.~Butscher, A.~Nguyen, T.~Du,
  and L.~Guibas.
\newblock Convolutional wasserstein distances: Efficient optimal transportation
  on geometric domains.
\newblock In {\em ACM Transactions on Graphics (SIGGRAPH'15)}, 2015.

\bibitem{sommerfeld2016inference}
M.~Sommerfeld and A.~Munk.
\newblock Inference for empirical wasserstein distances on finite spaces.
\newblock {\em Journal of the Royal Statistical Society: Series B (Statistical
  Methodology)}, 2016.

\bibitem{sourati2017asymptotic}
J.~Sourati, M.~Akcakaya, T.~K. Leen, D.~Erdogmus, and J.~G. Dy.
\newblock Asymptotic analysis of objectives based on fisher information in
  active learning.
\newblock {\em Journal of Machine Learning Research}, 18(34):1--41, 2017.

\bibitem{thibault2017overrelaxed}
A.~Thibault, L.~Chizat, C.~Dossal, and N.~Papadakis.
\newblock Overrelaxed sinkhorn-knopp algorithm for regularized optimal
  transport.
\newblock {\em arXiv preprint arXiv:1711.01851}, 2017.

\bibitem{van1996weak}
A.~W. Van Der~Vaart and J.~A. Wellner.
\newblock {\em Weak convergence and empirical processes}.
\newblock Springer, 1996.

\bibitem{villani2003topics}
C.~Villani.
\newblock {\em Topics in optimal transportation}, volume~58 of {\em Graduate
  Studies in Mathematics}.
\newblock American Mathematical Society, 2003.

\bibitem{wasserman2013all}
L.~Wasserman.
\newblock {\em All of statistics: a concise course in statistical inference}.
\newblock Springer Science \& Business Media, 2011.

\bibitem{wilson1969use}
A.~G. Wilson.
\newblock The use of entropy maximising models, in the theory of trip
  distribution, mode split and route split.
\newblock {\em Journal of Transport Economics and Policy}, pages 108--126,
  1969.

\bibitem{ye2015fast}
J.~Ye, P.~Wu, J.~Z. Wang, and J.~Li.
\newblock Fast discrete distribution clustering using {W}asserstein barycenter
  with sparse support.
\newblock {\em {IEEE} Trans. Signal Processing}, 65(9):2317--2332, 2017.

\bibitem{zalinescu2002convex}
C.~Zalinescu.
\newblock {\em Convex analysis in general vector spaces}.
\newblock World Scientific, 2002.

\end{thebibliography}
\end{document}